\documentclass[11pt]{amsart}
\usepackage[top=1.5in, bottom=1.5in, left=1.5in, right=1.5in]{geometry}
\expandafter\let\csname ver@amsthm.sty\endcsname\relax

\usepackage{graphicx}
\usepackage{tikz}
\usetikzlibrary{arrows}
\usetikzlibrary{decorations.markings}
\usetikzlibrary{decorations.pathmorphing}

\usepackage{amsmath}
\usepackage{amssymb}
\usepackage{enumerate}
\usepackage{mathdots}
\usepackage{mathtools}

\usepackage{hyperref}
\usepackage{amsthm}
\usepackage[capitalize,noabbrev]{cleveref}

\allowdisplaybreaks

\numberwithin{equation}{section}

\newtheorem{thm}{Theorem}[section]
\newtheorem{lemma}[thm]{Lemma}
\newtheorem{cor}[thm]{Corollary}
\newtheorem{prop}[thm]{Proposition}
\newtheorem{conj}[thm]{Conjecture}

\newtheorem{Definition}[thm]{Definition}
\newenvironment{definition}
  {\begin{Definition}\rm}{\end{Definition}}

\newtheorem{Example}[thm]{Example}
\newenvironment{example}
  {\begin{Example}\rm}{\end{Example}}

\newtheorem{Remark}[thm]{Remark}

  \newtheorem{construction}{Construction}

\crefname{thm}{Theorem}{Theorems}
\crefname{lemma}{Lemma}{Lemmas}
\crefname{cor}{Corollary}{Corollaries}
\crefname{prop}{Proposition}{Propositions}
\crefname{conj}{Conjecture}{Conjectures}
\crefname{question}{Question}{Questions}

\crefname{definition}{Definition}{Definitions}
\crefname{example}{Example}{Examples}
\crefname{remark}{Remark}{Remarks}

\crefname{construction}{Construction}{Constructions}

\newcommand{\emailhref}[1]{\email{\href{#1}{#1}}}

\newcommand{\dfn}[1]{\textcolor{blue}{\emph{#1}}}

\DeclareMathOperator{\pro}{{Pro}} 
\DeclareMathOperator{\evac}{{Evac}} 

\DeclareMathOperator{\rot}{{Rot}} 
\DeclareMathOperator{\rc}{{RevComp}} 
\DeclareMathOperator{\flip}{{Flip}} 

\newcommand{\Sym}{\mathfrak{S}} 

\newcommand{\M}{\mathcal{M}} 
\newcommand{\N}{\mathcal{N}} 
\newcommand{\D}{\mathcal{D}} 

\newcommand{\W}{\mathcal{W}} 

\newcommand{\A}{\mathrm{A}} 
\newcommand{\B}{\mathrm{B}} 
\newcommand{\C}{\mathrm{C}} 

\makeatletter
\newcommand*\circled[2][0.9]{\tikz[baseline=(char.base)]{
    \node[shape=circle, draw, inner sep=0pt,
        minimum height={\f@size*#1},] (char) {$\vphantom{WAH1g}#2$};}}
\makeatother

\title{Promotion of Kreweras words}

\author{Sam Hopkins}
\address{Department of Mathematics, Howard University, Washington, DC, USA}
\emailhref{samuelfhopkins@gmail.com}
\thanks{Sam Hopkins was supported by NSF grant \#1802920}

\author{Martin Rubey}
\address{Fakult\"{a}t f\"{u}r Mathematik und Geoinformation, TU Wien, Austria}
\emailhref{martin.rubey@tuwien.ac.at}
\thanks{Martin Rubey was supported by the the Austrian Science Fund (FWF): P 29275}


\begin{document}

\begin{abstract}
Kreweras words are words consisting of $n$ $\A$'s, $n$ $\B$'s, and $n$ $\C$'s in which every prefix has at least as many $\A$'s as $\B$'s and at least as many $\A$'s as~$\C$'s. Equivalently, a Kreweras word is a linear extension of the poset ${\sf V}\times [n]$. Kreweras words were introduced in 1965 by Kreweras, who gave a remarkable product formula for their enumeration. Subsequently they became a fundamental example in the theory of lattice walks in the quarter plane. We study Sch\"{u}tzenberger's promotion operator on the set of Kreweras words. In particular, we show that $3n$ applications of promotion on a Kreweras word merely swaps the $\B$'s and $\C$'s. Doing so, we provide the first answer to a question of Stanley from~2009, asking for posets with `good' behavior under promotion, other than the four families of shapes classified by Haiman in~1992. We also uncover a strikingly simple description of Kreweras words in terms of Kuperberg's $\mathfrak{sl}_3$-webs, and Postnikov's trip permutation associated with any plabic graph. In this description, Sch\"{u}tzenberger's promotion corresponds to rotation of the web.
\end{abstract}

\maketitle

\section{Introduction} \label{sec:intro}

The famous ballot problem, whose history stretches back to the 19th century, asks in how many ways we can order the ballots of an election between two candidates~Alice and Bob, who each receive $n$ votes, so that during the counting of ballots Alice never trails Bob. These ballot orderings correspond to words of length $2n$ in the letters~$\A$ and~$\B$, with as many~$\A$'s as~$\B$'s, for which every prefix has at least as many~$\A$'s as~$\B$'s. Such words are called \dfn{Dyck words}, and they are counted by the ubiquitous \dfn{Catalan numbers}
\[C_n  \coloneqq  \frac{1}{n+1}\binom{2n}{n}.\]

In 1965, Kreweras~\cite{kreweras1965classe} considered the following version of a $3$-candidate ballot problem: in how many ways can we order the ballots of an election between three candidates~Alice, Bob, and Charlie, who each receive $n$ votes, so that during the counting Alice never trails Bob and Alice never trails Charlie -- although the relative position of Bob and Charlie may change during the counting?
These ballot orderings correspond to words of length~$3n$ in the letters~$\A$, $\B$, and~$\C$, with equally many~$\A$'s, $\B$'s, and $\C$'s, for which every prefix has at least as many~$\A$'s as~$\B$'s and also at least as many~$\A$'s as~$\C$'s.
We call such words \dfn{Kreweras words}. Kreweras proved that they are counted by the formula
\[ K_n  \coloneqq  \frac{4^n}{(n+1)(2n+1)}\binom{3n}{n}.\]

For many years Kreweras's formula seemed like an isolated enumerative curiosity, although simplified proofs were presented by Niederhausen~\cite{niederhausen1980sheffer, niederhausen1983ballot} and Kreweras--Niederhausen~\cite{kreweras1981solution} in the 1980s.
Gessel~\cite{gessel1986probabilistic} gave yet another proof which demonstrated that the generating function $\sum_{n=0}^{\infty}K_n \, x^n$ for this sequence of numbers is algebraic.
Interest in Kreweras's result was revived decades later in the context of lattice walk enumeration. Kreweras words evidently correspond to walks in $\mathbb{Z}^2$ with steps of the form $\A=(1,1)$, $\B=(-1,0)$, and $\C=(0,-1)$ from the origin to itself which always remain in the nonnegative orthant.
Such walks are called \dfn{Kreweras walks}.
Bousquet-M\'{e}lou~\cite{bousquetmelou2005walks} gave another proof of Kreweras's product formula counting Kreweras walks using the kernel method from analytic combinatorics.
Indeed, the Kreweras walks are nowadays a fundamental example in the study of ``walks with small step sizes in the quarter plane,'' a program successfully carried out over a number of years in the 2000s by Bousquet-M\'{e}lou and others (see, e.g.,~\cite{bousquetmelou2010walks}).
Finally, we note that Bernardi~\cite{bernardi2007bijective} gave a purely combinatorial proof of the product formula for the number of Kreweras walks via a bijection with (decorated) cubic maps.

In this paper we study a cyclic group action on Kreweras words.

Let $w=(w_1,w_2,\ldots,w_{3n})$ be a Kreweras word of length $3n$. The \dfn{promotion} of~$w$, denoted $\pro(w)$, is obtained from $w$ as follows. Let $\iota(w)$ be the smallest index $\iota \geq 1$ for which the prefix $(w_1,w_2,\ldots,w_{\iota})$ has either the same number of $\A$'s as $\B$'s or the same number of $\A$'s as $\C$'s. Then
\[ \pro(w)  \coloneqq  (w_2,w_3,\ldots,w_{\iota(w)-1}, \A, w_{\iota(w)+1},  w_{\iota(w)+2}, \ldots, w_{3n}, w_{\iota(w)}).\]
It is easy to verify that $\pro(w)$ is also a Kreweras word, and that promotion is an invertible action on the set of Kreweras words.

\begin{example} \label{ex:kword_pro}
Let $w = \A \A \B \circled{\B} \C \A \C \C \B$. Here we circled the letter $w_{\iota(w)}$, and hence $\pro(w) = \A \B \A \C \A \C \C \B \B$. We can further compute that the first several iterates of promotion applied to $w$ are
\begin{align*}
\pro(w) &= \A \circled{\B} \A \C \A \C \C \B \B \\
\pro^2(w) &= \A \A \C \A \C \circled{\C} \B \B \B \\
\pro^3(w) &= \A \circled{\C} \A \C \A \B \B \B \C \\
\pro^4(w) &= \A \A \C \A \B \B \circled{\B} \C \C \\
\pro^5(w) &= \A \circled{\C} \A \B \B \A \C \C \B \\
\pro^6(w) &= \A \A \B \circled{\B} \A \C \C \B \C \\
\pro^7(w) &= \A \circled{\B} \A \A \C \C \B \C \B \\
\pro^8(w) &= \A \A \A \C \C \B \circled{\C} \B \B \\
\pro^9(w) &= \A \A \C \C \B \A \B \B \C
\end{align*}
Note that $\pro^9(w)$ is obtained from $w$ by swapping all $\B$'s for~$\C$'s and vice-versa.
\end{example}

Our first result predicts the order of promotion on Kreweras words:

\begin{thm} \label{thm:main}
Let $w$ be a Kreweras word of length $3n$. Then $\pro^{3n}(w)$ is obtained from~$w$ by swapping all $\B$'s for $\C$'s and vice-versa. In particular, $\pro^{6n}(w)=w$.
\end{thm}

Promotion of Kreweras words comes from the theory of partially ordered sets. In a series of papers from the 60s and 70s, Sch\"{u}tzenberger~\cite{schutzenberger1963quelques, schutzenberger1972promotion, schutzenberger1973evacuations} introduced and developed the theory of a cyclic action called \dfn{promotion}, as well as a closely related involutive action called \dfn{evacuation}, on the \dfn{linear extensions} of any poset. Let $V(n)$ denote the Cartesian product of the $3$-element ``V''-shaped poset~$\begin{tikzpicture}[scale=0.3] \node[shape=circle,fill=black,inner sep=1] (B) at (-1,0) {}; \node[shape=circle,fill=black,inner sep=1] (C) at (1,0) {}; \node[shape=circle,fill=black,inner sep=1] (A) at (0,-1) {}; \draw (B)--(A); \draw (C)--(A); \end{tikzpicture}$ and the $n$-element chain $[n]$. Then, as observed by Kreweras--Niederhausen~\cite{kreweras1981solution}, the linear extensions of $V(n)$ are in obvious bijection with the Kreweras words of length~$3n$. And promotion of Kreweras words as described above is the same as Sch\"{u}tzenberger's promotion on the linear extensions of $V(n)$.

Previously there were only four known (non-trivial) families of posets for which the order of promotion can be predicted; see \cref{fig:haiman_posets}. These were classified by Haiman in the 1990s~\cite{haiman1992dual, haiman1992characterization}. In a survey on promotion and evacuation, Stanley~\cite[\S4, Question~3]{stanley2009promotion} asked whether there were any other families of posets for which the order of promotion is given by a simple formula. Our work shows that $V(n)$ is such an example.

\begin{figure}
\begin{center}
\begin{tikzpicture}
\node (A) at (0,0) {\begin{tabular}{c c c c c c c c}$\A$ & $\A$ & $\B$ & $\A$ & $\B$ & $\circled{\B}$ & $\A$ & $\B$ \\ 1 & 2 & 3 & 4 & 5 & 6 & 7 & 8\end{tabular}};
\node (B) at (0,-5) {\begin{tabular}{c c c c c c c c}$\A$ & $\B$ & $\A$ & $\B$ & $\A$ & $\A$ & $\B$ & $\B$ \\ 1 & 2 & 3 & 4 & 5 & 6 & 7 & 8\end{tabular}};
\node at (-1,-2.5) {Promotion};
\draw [thick,->] (A)--(B);
\node (C) at (6,0) {\begin{tikzpicture}[scale=1.25]
\node (1) at (-0.3,0) {1};
\node (2) at (-0.75,-0.45) {2};
\node (3) at (-0.75,-1.05) {3};
\node (4) at (-0.3,-1.5) {4};
\node (5) at (0.3,-1.5) {5};
\node (6) at (0.75,-1.05) {6};
\node (7) at (0.75,-0.45) {7};
\node (8) at (0.3,0) {8};
\draw[ultra thick] (2) to [out=180-200, in=180-160] (3);
\draw[ultra thick] (4) to [out=180-110, in=180-70]  (5);
\draw[ultra thick] (1) to [out=180-250,in=180-20] (6);
\draw[ultra thick] (7) to [out=180-0, in=180-270] (8);
\end{tikzpicture}};
\node (D) at (6,-5) {\begin{tikzpicture}[scale=1.25]
\node (1) at (-0.3,0) {1};
\node (2) at (-0.75,-0.45) {2};
\node (3) at (-0.75,-1.05) {3};
\node (4) at (-0.3,-1.5) {4};
\node (5) at (0.3,-1.5) {5};
\node (6) at (0.75,-1.05) {6};
\node (7) at (0.75,-0.45) {7};
\node (8) at (0.3,0) {8};
\draw[ultra thick] (1) to [out=180-270,in=180-180] (2);
\draw[ultra thick] (3) to [out=180-180,in=180-90] (4);
\draw[ultra thick] (5) to [out=180-70,in=180--70] (8);
\draw[ultra thick] (6) to [out=180-20, in=180--20] (7);
\end{tikzpicture}};
\node at (7,-2.5) {Rotation};
\draw[thick,->] (C) -- (D);
\draw[thick,<->] (A) -- (C);
\draw[thick,<->] (B) -- (D);
\end{tikzpicture}
\end{center}
\caption{Promotion of Dyck words as rotation of noncrossing matchings.} \label{fig:dyck_noncrossing}
\end{figure}
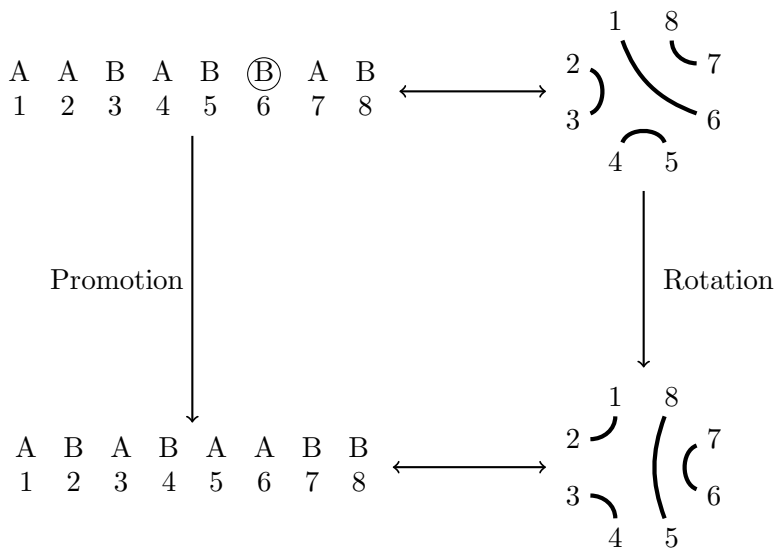

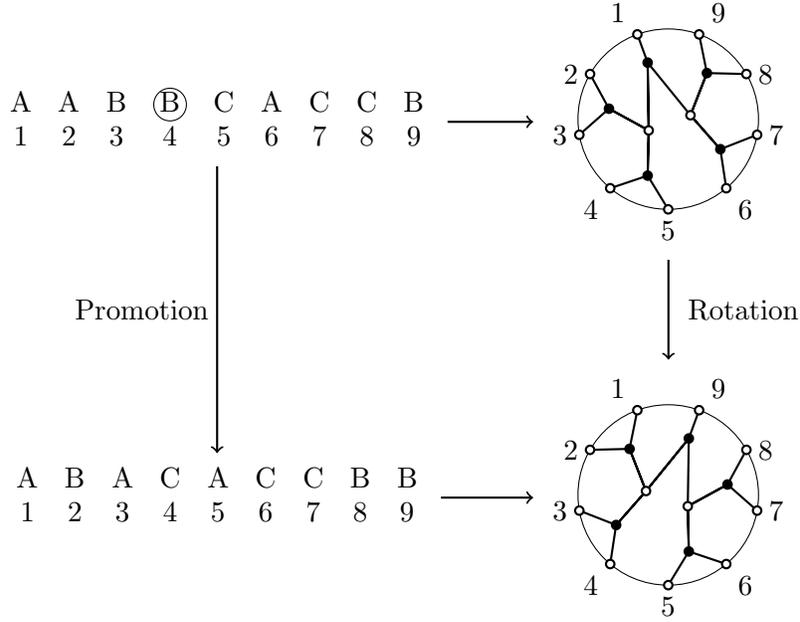
\begin{figure}
\begin{tikzpicture}
\node (Ax) at (0,0) {\begin{tabular}{c c c c c c c c c}$\A$ & $\A$ & $\B$ & $\circled{\B}$ & $\C$ & $\A$ & $\C$ & $\C$ & $\B$ \\ 1 & 2 & 3 & 4 & 5 & 6 & 7 & 8 & 9\end{tabular}};
\node (Bx) at (0,-5) {\begin{tabular}{c c c c c c c c c}$\A$ & $\B$ & $\A$ & $\C$ & $\A$ & $\C$ & $\C$ & $\B$ & $\B$ \\ 1 & 2 & 3 & 4 & 5 & 6 & 7 & 8 & 9\end{tabular}};
\node at (-1,-2.5) {Promotion};
\draw [thick,->] (Ax)--(Bx);
\node (Cx) at (6,0) {\begin{tikzpicture}[scale=0.4]
\draw (0,0) circle (3);
\node[label={above left:1},inner sep=0] (1) at (110:3) {};
\node[label={left:2},inner sep=0] (2) at (150:3) {};
\node[label={left:3},inner sep=0] (3) at (190:3) {};
\node[label={below left:4},inner sep=0] (4) at (230:3) {};
\node[label={below:5},inner sep=0] (5) at (270:3) {};
\node[label={below right:6},inner sep=0] (6) at (310:3) {};
\node[label={right:7},inner sep=0] (7) at (350:3) {};
\node[label={right:8},inner sep=0] (8) at (390:3) {};
\node[label={above right:9},inner sep=0] (9) at (430:3) {};
\node[inner sep=0] (A) at (170:2) {};
\node[inner sep=0] (B) at (250:2) {};
\node[inner sep=0] (C) at (330:2) {};
\node[inner sep=0] (D) at (410:2) {};
\node[inner sep=0] (E) at (110:2) {};
\node[inner sep=0] (F) at (210:0.75) {};
\node[inner sep=0] (G) at (370:0.75) {};
\draw[thick] (1)--(E)--(F)--(E)--(G);
\draw[thick] (2)--(A)--(F)--(A)--(3);
\draw[thick] (4)--(B)--(F)--(B)--(5);
\draw[thick] (6)--(C)--(G)--(C)--(7);
\draw[thick] (8)--(D)--(G)--(D)--(9);
\draw[thick,draw=black,fill=white]  (1) circle (4pt);
\draw[thick,draw=black,fill=white]  (2) circle (4pt);
\draw[thick,draw=black,fill=white]  (3) circle (4pt);
\draw[thick,draw=black,fill=white]  (4) circle (4pt);
\draw[thick,draw=black,fill=white]  (5) circle (4pt);
\draw[thick,draw=black,fill=white]  (6) circle (4pt);
\draw[thick,draw=black,fill=white]  (7) circle (4pt);
\draw[thick,draw=black,fill=white]  (8) circle (4pt);
\draw[thick,draw=black,fill=white]  (9) circle (4pt);
\draw[thick,draw=black,fill=black]  (A) circle (4pt);
\draw[thick,draw=black,fill=black]  (B) circle (4pt);
\draw[thick,draw=black,fill=black]  (C) circle (4pt);
\draw[thick,draw=black,fill=black]  (D) circle (4pt);
\draw[thick,draw=black,fill=black]  (E) circle (4pt);
\draw[thick,draw=black,fill=white]  (F) circle (4pt);
\draw[thick,draw=black,fill=white]  (G) circle (4pt);
\end{tikzpicture}};
\node (Dx) at (6,-5) {\begin{tikzpicture}[scale=0.4]
\draw (0,0) circle (3);
\node[label={above right:9},inner sep=0] (1) at (70:3) {};
\node[label={above left:1},inner sep=0] (2) at (110:3) {};
\node[label={left:2},inner sep=0] (3) at (150:3) {};
\node[label={left:3},inner sep=0] (4) at (190:3) {};
\node[label={below left:4},inner sep=0] (5) at (230:3) {};
\node[label={below:5},inner sep=0] (6) at (270:3) {};
\node[label={below right:6},inner sep=0] (7) at (310:3) {};
\node[label={right:7},inner sep=0] (8) at (350:3) {};
\node[label={right:8},inner sep=0] (9) at (390:3) {};
\node[inner sep=0] (A) at (130:2) {};
\node[inner sep=0] (B) at (210:2) {};
\node[inner sep=0] (C) at (290:2) {};
\node[inner sep=0] (D) at (370:2) {};
\node[inner sep=0] (E) at (70:2) {};
\node[inner sep=0] (F) at (170:0.75) {};
\node[inner sep=0] (G) at (330:0.75) {};
\draw[thick] (1)--(E)--(F)--(E)--(G);
\draw[thick] (2)--(A)--(F)--(A)--(3);
\draw[thick] (4)--(B)--(F)--(B)--(5);
\draw[thick] (6)--(C)--(G)--(C)--(7);
\draw[thick] (8)--(D)--(G)--(D)--(9);
\draw[thick,draw=black,fill=white]  (1) circle (4pt);
\draw[thick,draw=black,fill=white]  (2) circle (4pt);
\draw[thick,draw=black,fill=white]  (3) circle (4pt);
\draw[thick,draw=black,fill=white]  (4) circle (4pt);
\draw[thick,draw=black,fill=white]  (5) circle (4pt);
\draw[thick,draw=black,fill=white]  (6) circle (4pt);
\draw[thick,draw=black,fill=white]  (7) circle (4pt);
\draw[thick,draw=black,fill=white]  (8) circle (4pt);
\draw[thick,draw=black,fill=white]  (9) circle (4pt);
\draw[thick,draw=black,fill=black]  (A) circle (4pt);
\draw[thick,draw=black,fill=black]  (B) circle (4pt);
\draw[thick,draw=black,fill=black]  (C) circle (4pt);
\draw[thick,draw=black,fill=black]  (D) circle (4pt);
\draw[thick,draw=black,fill=black]  (E) circle (4pt);
\draw[thick,draw=black,fill=white]  (F) circle (4pt);
\draw[thick,draw=black,fill=white]  (G) circle (4pt);
\end{tikzpicture}};
\node at (7,-2.5) {Rotation};
\draw[thick,->] (Cx) -- (Dx);
\draw[thick,->] (Ax) -- (Cx);
\draw[thick,->] (Bx) -- (Dx);
\end{tikzpicture}
\caption{Promotion of Kreweras words as rotation of webs.} \label{fig:web-rotation}
\end{figure}

Dyck words of length $2n$ correspond to linear extensions of $[2]\times [n]$, and hence carry an action of promotion. \Cref{fig:dyck_noncrossing} depicts a well-known bijection between Dyck words of length~$2n$ and noncrossing matchings of~$[2n]  \coloneqq  \{1,2,\ldots,2n\}$, and shows how under this bijection promotion of Dyck words corresponds to rotation of noncrossing matchings (this was first observed by Dennis White: see~\cite[\S 8]{rhoades2010cyclic}). This observation immediately implies that $\pro^{2n}(w)=w$ for $w$ a Dyck word of length~$2n$.

Our proof of \cref{thm:main} is also essentially based on a diagrammatic representation of Kreweras words for which promotion corresponds to rotation; see~\cref{fig:web-rotation}. However, these diagrams are not coming from Bernardi's cubic map bijection. Instead, they are related to Kuperberg's webs.

\dfn{Webs} are certain trivalent bipartite planar graphs which Kuperberg~\cite{kuperberg1996spiders} introduced in order to study the invariant theory of Lie algebras. Khovanov and Kuperberg~\cite{khovanov1999web} showed that a particular class of webs (namely, irreducible $\mathfrak{sl}_3$-webs with $3n$ white boundary vertices) are in bijection with linear extensions of $[3]\times [n]$. Petersen, Pylyavskyy, and Rhoades~\cite{petersen2009promotion} (see also Tymoczko~\cite{tymoczko2012simple}) showed that, via the Khovanov-Kuperberg bijection, rotation of webs corresponds to promotion of linear extensions.

We say that $\W$ is a \dfn{Kreweras web} if $\W$ is an irreducible $\mathfrak{sl}_3$-web with all boundary vertices white and having no internal face with a multiple of four sides.  We define a surjective map $w \mapsto \W_w$ from Kreweras words to Kreweras webs.  This map behaves well with respect to  Sch\"{u}tzenberger's operators:

\begin{thm} \label{thm:intro_web}
Let $w$ be a Kreweras word. Then,
\begin{enumerate}[(a)]
\item $\W_{\pro(w)} = \rot(\W_{w})$;
\item $\W_{\evac(w)} = \flip(\W_{w})$,
\end{enumerate}
where $\rot$ denotes the rotation of a web and $\flip$ its reflection across a diameter.
\end{thm}

The map between Kreweras words and Kreweras webs can be made bijective by keeping track of a certain $3$-edge-coloring of the web. We then obtain the following enumerative corollaries.

\begin{thm} \label{thm:weighted_kreweras_web_count}
We have
\[ \sum_{\W} 2^{\kappa(\W)} = K_n = \frac{4^n}{(n+1)(2n+1)}\binom{3n}{n}  \qquad \textrm{(c.f.~\url{http://oeis.org/A006335})},  \]
where the sum is over all Kreweras webs $\W$ with $3n$ boundary vertices, and $\kappa(\W)$ is the number of connected components of $\W$. Moreover, the number of \emph{connected} Kreweras webs $\W$ with $3n$ boundary vertices is
\[ 2^{n} \, \frac{(4n-3)!}{(3n-1)!n!}  \qquad \textrm{(c.f.~\url{http://oeis.org/A000260})}.\]
\end{thm}

A curious feature of our work not present in any previous work we are aware of is the use of \dfn{trip permutations}, in the sense of Postnikov's theory of \dfn{plabic graphs}~\cite{postnikov2006total}, to study webs.

The rest of the paper is organized as follows. In \cref{sec:promotion} we review promotion and evacuation of linear extensions. In \cref{sec:proof} we prove \cref{thm:main} by using what we call the \dfn{Kreweras bump diagram} of a Kreweras word $w$, and extracting from this diagram a permutation $\sigma_w$ which rotates under promotion. This permutation~$\sigma_w$ is in fact a trip permutation. In \cref{sec:evac} we discuss evacuation of Kreweras words, using the theory of \dfn{growth diagrams}. In \cref{sec:webs} we reinterpret our results in the language of webs and plabic graphs. Finally, in \cref{sec:future} we briefly discuss some possible future directions.

\medskip

\subsection*{Acknowledgements}

S.H. thanks Ira Gessel, whose answer to a MathOverflow question of his~\cite{gessel2019MO} made him aware of the paper~\cite{kreweras1981solution} and the poset $V(n)$, and thus initiated this research.

\section{Promotion and evacuation of linear extensions} \label{sec:promotion}

In this section we quickly review background on promotion of linear extensions of posets, and explain precisely how promotion of Kreweras words as defined in \cref{sec:intro} fits into that framework. We assume that the reader is familiar with basic notions from poset theory as laid out for instance in~\cite[Chapter 3]{stanley2012ec1}. Throughout, all posets are assumed to be finite.

Let $P$ be a poset with $\ell$ elements. For us, a \dfn{linear extension} of $P$ will be a list $(p_1,p_2,\ldots,p_{\ell})$ of all of the elements of $P$, each appearing once, for which $p_i \leq p_j$ implies that~$i \leq j$. We use~$\mathcal{L}(P)$ denote the set of linear extensions of~$P$.

For each $1 \leq i \leq \ell-1$, we define the involution $\tau_i\colon \mathcal{L}(P) \to \mathcal{L}(P)$ by
\[ \tau_i(p_1,\ldots,p_{\ell})  \coloneqq  \begin{cases} (p_1,\ldots,p_{i-1},p_{i+1},p_{i},p_{i+2},\ldots,p_{\ell}) &\parbox{1.15in}{if $p_{i}$ and $p_{i+1}$ are incomparable;} \\ (p_1,\ldots,p_{\ell}) &\textrm{otherwise}. \end{cases} \]
In other words, $\tau_i$ swaps the $i$th and $(i+1)$st entries of a linear extension, if possible.

\begin{definition}
\dfn{Promotion}, $\pro\colon \mathcal{L}(P)\to \mathcal{L}(P)$, is the following composition of~$\tau_i$:
\[ \pro  \coloneqq  \tau_{\ell-1} \circ \tau_{\ell-2} \circ \cdots \circ \tau_1. \]
\end{definition}

Being a composition of involutions, promotion is evidently invertible. There is another definition of promotion in terms of \dfn{jeu de taquin} slides, but we find the definition as composition of the $\tau_i$ more convenient; see Stanley's survey~\cite[\S 2]{stanley2009promotion} for the details.

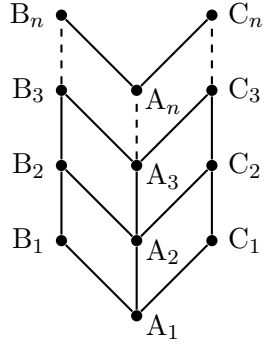
\begin{figure}
\begin{center}
\begin{tikzpicture}
\node[shape=circle,fill=black,inner sep=1.5,label={[yshift=-15,xshift=10]:$\A_1$}] (A1) at (0,-1) {};
\node[shape=circle,fill=black,inner sep=1.5,label={left:$\B_1$}] (B1) at (-1,0) {};
\node[shape=circle,fill=black,inner sep=1.5,label={right:$\C_1$}] (C1) at (1,0) {};
\node[shape=circle,fill=black,inner sep=1.5,label={[yshift=-15,xshift=10]:$\A_2$}] (A2) at (0,0) {};
\node[shape=circle,fill=black,inner sep=1.5,label={left:$\B_2$}] (B2) at (-1,1) {};
\node[shape=circle,fill=black,inner sep=1.5,label={right:$\C_2$}] (C2) at (1,1) {};
\node[shape=circle,fill=black,inner sep=1.5,label={[yshift=-15,xshift=10]:$\A_3$}] (A3) at (0,1) {};
\node[shape=circle,fill=black,inner sep=1.5,label={left:$\B_3$}] (B3) at (-1,2) {};
\node[shape=circle,fill=black,inner sep=1.5,label={right:$\C_3$}] (C3) at (1,2) {};
\node[shape=circle,fill=black,inner sep=1.5,label={[yshift=-15,xshift=10]:$\A_n$}] (An) at (0,2) {};
\node[shape=circle,fill=black,inner sep=1.5,label={left:$\B_n$}] (Bn) at (-1,3) {};
\node[shape=circle,fill=black,inner sep=1.5,label={right:$\C_n$}] (Cn) at (1,3) {};
\draw[thick] (B1)--(A1);
\draw[thick] (C1)--(A1);
\draw[thick] (A1)--(A2);
\draw[thick] (B1)--(B2);
\draw[thick] (C1)--(C2);
\draw[thick] (B2)--(A2);
\draw[thick] (C2)--(A2);
\draw[thick] (A2)--(A3);
\draw[thick] (B2)--(B3);
\draw[thick] (C2)--(C3);
\draw[thick] (B3)--(A3);
\draw[thick] (C3)--(A3);
\draw[thick,dashed] (A3)--(An);
\draw[thick,dashed] (B3)--(Bn);
\draw[thick,dashed] (C3)--(Cn);
\draw[thick] (Bn)--(An);
\draw[thick] (Cn)--(An);
\end{tikzpicture}
\end{center}
\caption{The poset $V(n)$ whose linear extensions are Kreweras words.}\label{fig:vn_poset}
\end{figure}

Recall that in \cref{sec:intro} we defined $V(n)  \coloneqq  \begin{tikzpicture}[scale=0.3] \node[shape=circle,fill=black,inner sep=1] (B) at (-1,0) {}; \node[shape=circle,fill=black,inner sep=1] (C) at (1,0) {}; \node[shape=circle,fill=black,inner sep=1] (A) at (0,-1) {}; \draw (B)--(A); \draw (C)--(A); \end{tikzpicture} \times [n]$ to be the poset which is the Cartesian product of the $3$-element ``V''-shaped poset and the $n$-element chain. We label the $3n$ elements of $V(n)$ by the symbols $\A_1,\A_2,\ldots,\A_n$, $\B_1,\B_2,\ldots,\B_n$, and $\C_1,\C_2,\ldots,\C_n$ as depicted in \cref{fig:vn_poset}.

Now let us show that promotion of Kreweras words is the same as promotion of linear extensions of $V(n)$:

\begin{prop} \label{prop:kword_poset_pro}
``Removing the subscripts'' is a bijection from linear extensions of~$V(n)$ to Kreweras words of length $3n$, and under this bijection, promotion of linear extensions corresponds to promotion of Kreweras words as was defined in \cref{sec:intro}.
\end{prop}
\begin{proof}
The claim about removing the subscripts being a bijection is straightforward. The comparison of the two promotions is also essentially straightforward. Let $w$ be a Kreweras word of length $3n$, and consider the $\tau_i$ as acting on $w$ via the aforementioned bijection. Then the effect of $\tau_i$ on~$w$ is
\[ \tau_i(w) = \begin{cases} (w_1,\ldots,w_{i-1},w_{i+1},w_i,w_{i+2},\ldots,w_{3n}) &\textrm{if this is a Kreweras word}; \\ w &\textrm{otherwise}.\end{cases}\]
Consider the application of $\tau_{3n-1} \circ \cdots \circ \tau_1$ to $w$. Recall the definition of $\iota(w)$ from \cref{sec:intro}. When applying $\tau_i$ for $i=1,\ldots,\iota(w)-2$ to $w$, we will never be in the ``otherwise'' case above; so the effect of $\tau_{\iota(w)-2} \circ \cdots \circ \tau_{1}$ will be to cyclically rotate the substring $(w_1,\ldots,w_{\iota(w)-1})$. In other words, we ``push'' the initial $\A$ in $w$ as far to the right as we can. Then, when applying $\tau_{\iota(w)-1}$ we will for the first time be in the ``otherwise'' case, so $\tau_{\iota(w)-1}$ will act as the identity. Finally, when applying $\tau_i$ for $i=\iota(w)+1,\ldots,3n-1$ to $w$, we will again never be in the ``otherwise'' case; so the effect of $\tau_{3n-1} \circ \cdots \circ \tau_{\iota(w)}$ will be to cyclically rotate the substring $(w_{\iota(w)},\ldots,w_{3n})$. In other words, we ``push'' the $\B$ or $\C$ in position $\iota(w)$ all the way to the end. Thus $\tau_{3n-1} \circ \cdots \circ \tau_1$ indeed acts on $w$ in the same way as $\pro$, as was defined in \cref{sec:intro}.
\end{proof}

Next, in order to motivate the study of (powers of) promotion, let us recall the basic results concerning promotion and its close cousin evacuation which were established by Sch\"{u}tzenberger.

\begin{definition}
\dfn{Evacuation}, $\evac\colon \mathcal{L}(P)\to \mathcal{L}(P)$, is the following composition of~$\tau_i$:
\[ \evac  \coloneqq  (\tau_1) \circ (\tau_2 \circ \tau_{1}) \circ \cdots \circ (\tau_{\ell-2} \circ \cdots \circ \tau_{2} \circ \tau_1) \circ (\tau_{\ell-1} \circ \cdots \circ \tau_{2} \circ \tau_1).\]
\end{definition}

Again, there is an alternative definition of evacuation in terms of jeu de taquin. Historically, interest in evacuation arose because of its close connection with the Robinson-Schensted correspondence.

There is an obvious duality between the linear extensions of $P$ and of its dual poset~$P^*$: for $L = (p_1,\ldots,p_{\ell}) \in \mathcal{L}(P)$, we use $L^*$ to denote the linear extension $L^*  \coloneqq  (p_{\ell},p_{\ell-1},\ldots,p_1) \in \mathcal{L}(P^*)$ of the dual poset.

\dfn{Dual evacuation}, $\evac^{*}\colon \mathcal{L}(P)\to \mathcal{L}(P)$, is defined by $\evac^{*}(L)  \coloneqq  (\evac(L^*))^*$. In some sense, dual evacuation is ``just as natural'' as evacuation. In terms of the involutions $\tau_i$, we have
\[\evac^{*} = (\tau_{\ell-1}\circ \cdots \circ \tau_1) \circ (\tau_{\ell-1}\circ \cdots \circ \tau_2) \circ \cdots \circ (\tau_{\ell-1}\circ\tau_{\ell-2}) \circ (\tau_{\ell-1}) \]
Note that $L \mapsto (\pro(L^*))^*$ is just $\pro^{-1}$, so we do not give it a different name.

The basic results of Sch\"{u}tzenberger~\cite{schutzenberger1963quelques, schutzenberger1972promotion, schutzenberger1973evacuations} on promotion and evacuation of an arbitrary poset $P$ are summarized in the following proposition; again, see Stanley~\cite[\S 2]{stanley2009promotion} for a modern presentation.

\begin{prop} \label{prop:pro_evac_basics}
For any poset $P$, we have that:
\begin{itemize}
\item $\evac$ and $\evac^{*}$ are both involutions;
\item $\evac \circ \pro = \pro^{-1} \circ \evac$;
\item $\pro^{\#P} =  \evac^{*} \circ \evac$.
\end{itemize}
\end{prop}

Because $\pro^{\#P}$ is the composition of the ``natural'' involutions $\evac$ and $\evac^{*}$, this power of $\pro$ is somehow the ``right'' power to consider. (Sometimes older sources refer to $\pro^{\#P}$ as \dfn{total promotion}, but we will eschew this terminology as it tends to confuse.) Following~\cite[\S 4]{stanley2009promotion}, we now review the posets $P$ for which the behavior of~$\pro^{\#P}$ is understood.

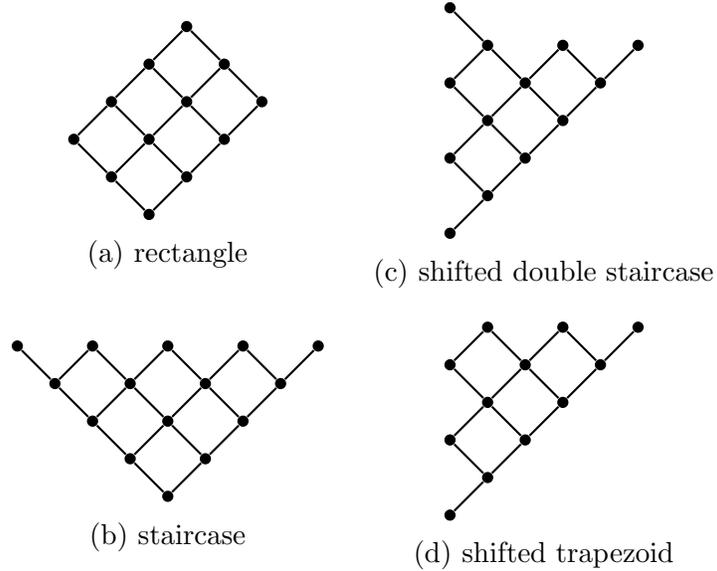
\begin{figure}
\begin{center}
\begin{tikzpicture}
\node [label={below:(a) rectangle}] at (0,0) {\begin{tikzpicture}[scale=0.5]
\node[shape=circle,fill=black,inner sep=1.5] (1) at (0,0) {};
\node[shape=circle,fill=black,inner sep=1.5] (2) at (-1,1) {};
\node[shape=circle,fill=black,inner sep=1.5] (3) at (1,1) {};
\node[shape=circle,fill=black,inner sep=1.5] (4) at (-2,2) {};
\node[shape=circle,fill=black,inner sep=1.5] (5) at (0,2) {};
\node[shape=circle,fill=black,inner sep=1.5] (6) at (2,2) {};
\node[shape=circle,fill=black,inner sep=1.5] (7) at (-1,3) {};
\node[shape=circle,fill=black,inner sep=1.5] (8) at (1,3) {};
\node[shape=circle,fill=black,inner sep=1.5] (9) at (3,3) {};
\node[shape=circle,fill=black,inner sep=1.5] (10) at (0,4) {};
\node[shape=circle,fill=black,inner sep=1.5] (11) at (2,4) {};
\node[shape=circle,fill=black,inner sep=1.5] (12) at (1,5) {};
\draw[thick] (1)--(2);
\draw[thick] (1)--(3);
\draw[thick] (2)--(4);
\draw[thick] (2)--(5);
\draw[thick] (3)--(5);
\draw[thick] (3)--(6);
\draw[thick] (4)--(7);
\draw[thick] (5)--(7);
\draw[thick] (5)--(8);
\draw[thick] (6)--(8);
\draw[thick] (6)--(9);
\draw[thick] (7)--(10);
\draw[thick] (8)--(10);
\draw[thick] (8)--(11);
\draw[thick] (9)--(11);
\draw[thick] (10)--(12);
\draw[thick] (11)--(12);
\end{tikzpicture}};
\node [label={below:(b) staircase}] at (0,-4) {\begin{tikzpicture}[scale=0.5]
\node[shape=circle,fill=black,inner sep=1.5] (1) at (0,0) {};
\node[shape=circle,fill=black,inner sep=1.5] (2) at (-1,1) {};
\node[shape=circle,fill=black,inner sep=1.5] (3) at (1,1) {};
\node[shape=circle,fill=black,inner sep=1.5] (4) at (-2,2) {};
\node[shape=circle,fill=black,inner sep=1.5] (5) at (0,2) {};
\node[shape=circle,fill=black,inner sep=1.5] (6) at (2,2) {};
\node[shape=circle,fill=black,inner sep=1.5] (7) at (-3,3) {};
\node[shape=circle,fill=black,inner sep=1.5] (8) at (-1,3) {};
\node[shape=circle,fill=black,inner sep=1.5] (9) at (1,3) {};
\node[shape=circle,fill=black,inner sep=1.5] (10) at (3,3) {};
\node[shape=circle,fill=black,inner sep=1.5] (11) at (-4,4) {};
\node[shape=circle,fill=black,inner sep=1.5] (12) at (-2,4) {};
\node[shape=circle,fill=black,inner sep=1.5] (13) at (0,4) {};
\node[shape=circle,fill=black,inner sep=1.5] (14) at (2,4) {};
\node[shape=circle,fill=black,inner sep=1.5] (15) at (4,4) {};
\draw[thick] (1)--(2);
\draw[thick] (1)--(3);
\draw[thick] (2)--(4);
\draw[thick] (2)--(5);
\draw[thick] (3)--(5);
\draw[thick] (3)--(6);
\draw[thick] (4)--(7);
\draw[thick] (4)--(8);
\draw[thick] (5)--(8);
\draw[thick] (5)--(9);
\draw[thick] (6)--(9);
\draw[thick] (6)--(10);
\draw[thick] (7)--(11);
\draw[thick] (7)--(12);
\draw[thick] (8)--(12);
\draw[thick] (8)--(13);
\draw[thick] (9)--(13);
\draw[thick] (9)--(14);
\draw[thick] (10)--(14);
\draw[thick] (10)--(15);
\end{tikzpicture}};
\node [label={below:(c) shifted double staircase}] at (5,0) {\begin{tikzpicture}[scale=0.5]
\node[shape=circle,fill=black,inner sep=1.5] (1) at (0,0) {};
\node[shape=circle,fill=black,inner sep=1.5] (2) at (1,1) {};
\node[shape=circle,fill=black,inner sep=1.5] (3) at (0,2) {};
\node[shape=circle,fill=black,inner sep=1.5] (4) at (2,2) {};
\node[shape=circle,fill=black,inner sep=1.5] (5) at (1,3) {};
\node[shape=circle,fill=black,inner sep=1.5] (6) at (3,3) {};
\node[shape=circle,fill=black,inner sep=1.5] (7) at (0,4) {};
\node[shape=circle,fill=black,inner sep=1.5] (8) at (2,4) {};
\node[shape=circle,fill=black,inner sep=1.5] (9) at (4,4) {};
\node[shape=circle,fill=black,inner sep=1.5] (10) at (1,5) {};
\node[shape=circle,fill=black,inner sep=1.5] (11) at (3,5) {};
\node[shape=circle,fill=black,inner sep=1.5] (12) at (5,5) {};
\node[shape=circle,fill=black,inner sep=1.5] (13) at (0,6) {};
\draw[thick] (1)--(2);
\draw[thick] (2)--(3);
\draw[thick] (2)--(4);
\draw[thick] (3)--(5);
\draw[thick] (4)--(5);
\draw[thick] (4)--(6);
\draw[thick] (5)--(7);
\draw[thick] (5)--(8);
\draw[thick] (6)--(8);
\draw[thick] (6)--(9);
\draw[thick] (7)--(10);
\draw[thick] (8)--(10);
\draw[thick] (8)--(11);
\draw[thick] (9)--(11);
\draw[thick] (9)--(12);
\draw[thick] (10)--(13);
\end{tikzpicture}};
\node [label={below:(d) shifted trapezoid}] at (5,-4) {\begin{tikzpicture}[scale=0.5]
\node[shape=circle,fill=black,inner sep=1.5] (1) at (0,0) {};
\node[shape=circle,fill=black,inner sep=1.5] (2) at (1,1) {};
\node[shape=circle,fill=black,inner sep=1.5] (3) at (0,2) {};
\node[shape=circle,fill=black,inner sep=1.5] (4) at (2,2) {};
\node[shape=circle,fill=black,inner sep=1.5] (5) at (1,3) {};
\node[shape=circle,fill=black,inner sep=1.5] (6) at (3,3) {};
\node[shape=circle,fill=black,inner sep=1.5] (7) at (0,4) {};
\node[shape=circle,fill=black,inner sep=1.5] (8) at (2,4) {};
\node[shape=circle,fill=black,inner sep=1.5] (9) at (4,4) {};
\node[shape=circle,fill=black,inner sep=1.5] (10) at (1,5) {};
\node[shape=circle,fill=black,inner sep=1.5] (11) at (3,5) {};
\node[shape=circle,fill=black,inner sep=1.5] (12) at (5,5) {};
\draw[thick] (1)--(2);
\draw[thick] (2)--(3);
\draw[thick] (2)--(4);
\draw[thick] (3)--(5);
\draw[thick] (4)--(5);
\draw[thick] (4)--(6);
\draw[thick] (5)--(7);
\draw[thick] (5)--(8);
\draw[thick] (6)--(8);
\draw[thick] (6)--(9);
\draw[thick] (7)--(10);
\draw[thick] (8)--(10);
\draw[thick] (8)--(11);
\draw[thick] (9)--(11);
\draw[thick] (9)--(12);
\end{tikzpicture}};
\end{tikzpicture}
\end{center}
\caption{The previously known posets with good behavior of promotion.}\label{fig:haiman_posets}
\end{figure}

The fundamental properties of jeu de taquin which Sch\"{u}tzenberger established in~\cite{schutzenberger1977correspondance} imply that for $P=[k] \times [n]$ a product of two chains, a.k.a.~a \dfn{rectangle}, $\pro^{\#P}$ is the identity. While studying reduced decompositions in the symmetric group, Edelman and Greene~\cite{edelman1987balanced} showed that for $P$ a \dfn{staircase} (depicted in \cref{fig:haiman_posets}(b) by way of example), $\pro^{\#P}$ is transposition (i.e., the poset automorphism that is the reflection across the vertical axis of symmetry). Finally, Haiman~\cite{haiman1992dual} showed that for $P$ a \dfn{shifted double staircase} (\cref{fig:haiman_posets}(c)) or a \dfn{shifted trapezoid} (\cref{fig:haiman_posets}(d)), $\pro^{\#P}$ is the identity; and he did this using a general method (based on his notion of dual equivalence) which recaptures the rectangle and staircase results as well. In a follow-up paper, Haiman and Kim~\cite{haiman1992characterization} proved moreover that among Young diagram shapes and shifted shapes, the posets in \cref{fig:haiman_posets} are the only families for which $\pro^{\#P}$ is the identity or transposition.

In~\cite[\S4, Question 3]{stanley2009promotion}, Stanley asked whether there are any other (non-trivial) families of posets, beyond those depicted in \cref{fig:haiman_posets}, for which~$\pro^{\#P}$ can be described. Our main result, \cref{thm:main}, shows that $V(n)$ is such an example: for $P=V(n)$, $\pro^{\#P}$ is ``transposition'' (i.e., reflection across the vertical axis of symmetry).

\section{The order of promotion} \label{sec:proof}

In this section we prove \cref{thm:main}. Throughout we fix a nonnegative integer~$n$ as in the statement of that theorem. Also, from now on we adopt the notational convention $-\B: = \C$ and $-\C  \coloneqq  \B$.

Our strategy is to associate to each Kreweras word a permutation, such that promotion of the Kreweras word corresponds to rotation of the permutation.

\begin{definition}
We use $\Sym_{m}$ to denote the symmetric group on $m$ letters. We represent permutations $\sigma \in \Sym_{m}$ either via cycle notation, or via one-line notation as $\sigma = [\sigma(1),\ldots,\sigma(m)]$. For $\sigma \in \Sym_{m}$, the \dfn{rotation} of $\sigma$, denoted $\rot(\sigma)$, is the (right) conjugation of $\sigma$ by the \dfn{long cycle} $(1,2,\ldots,m) \in \Sym_{m}$; i.e.,
\[\rot(\sigma)  \coloneqq  (1,2,\ldots,m)^{-1} \circ \sigma \circ (1,2,\ldots,m).\]
Rotation of a permutation as we have just defined it corresponds to rotation of its functional digraph representation.
\end{definition}

We first associate a diagram to a Kreweras word, which we will then use to obtain the desired permutation. This diagram will be built out of arcs, and the crossings of the arcs in the diagram will play a significant role. So let us review notions related to arcs and crossings.

\begin{definition}
  An \dfn{arc} is a pair $(i,j)$ of positive integers with $i < j$.  A~\dfn{crossing} is a set $\{(i,j),(k,\ell)\}$ of two arcs such that $i\leq k < j < \ell$.
\end{definition}

Note that this definition slightly deviates from the usual notion, in that the pairs~$(i,j)$ and~$(i,\ell)$ form a crossing. However, this modification is only relevant when considering Kreweras bump diagrams, defined below.

\begin{definition}
  Let $\mathcal{A}$ be a collection of arcs. For a set of positive integers  $S$, we say that $\mathcal{A}$ is a \dfn{noncrossing matching of $S$} if for every $(i,j)\in \mathcal{A}$ we have $i,j \in S$, every $i\in S$ belongs to a unique arc in~$\mathcal{A}$, and no two arcs in $\mathcal{A}$ form a crossing. The set of \dfn{openers of $\mathcal{A}$} is $\{i\colon (i,j)\in \mathcal{A}\}$ and the set the set of \dfn{closers of~$\mathcal{A}$} is $\{j\colon (i,j)\in \mathcal{A}\}$.
\end{definition}

The bijection between Dyck words and noncrossing matchings suggested by \cref{fig:dyck_noncrossing} enters into the definition of the diagram we associate to a Kreweras word, so let us now formalize this bijection. Let $w$ be a Dyck word of length $2n$. We associate to $w$ the unique noncrossing matching of~$[2n]$ whose set of openers is~$\{i\in [2n]\colon w_i=\textrm{$\A$}\}$ and whose set of closers is~$\{i\in [2n]\colon w_i=B\}$. This sets up a (well-known) bijection between the Dyck words of length $2n$ and the noncrossing matchings of $[2n]$.

A Kreweras word can be thought of as two overlapping Dyck words, and hence has two noncrossing matchings naturally associated to it. As we now explain, the diagram for our Kreweras word will be the union of these two noncrossing matchings.

\begin{definition}
Let $w$ be a Kreweras word of length $3n$. Let $\varepsilon \in \{\textrm{$\B$},\textrm{$\C$}\}$. We use~$\M^\varepsilon_w$ to denote the noncrossing matching of $\{i\in [3n]\colon w_i\neq -\varepsilon\}$ whose set of openers is~$\{i\in [3n]\colon w_i=\textrm{$\A$}\}$ and whose set of closers is~$\{i\in [3n]\colon w_i=\varepsilon\}$.

The \dfn{Kreweras bump diagram} $\D_w$ of $w$ is obtained by placing the numbers $1,\ldots,3n$ in this order on a line, and drawing a semicircle above the line connecting $i$ and $j$ for each arc~$(i, j)\in \M^{B}_w \cup \M^{C}_w$.  The arc is solid {\bf b}lue if~$(i,j)\in\M^{B}_w$ and dashed {\bf c}rimson (i.e.,~red) if~$(i,j) \in \M^{C}_w$. The arcs are drawn in such a fashion that only pairs of arcs which form a crossing intersect, and any two arcs intersect at most once.
\end{definition}

\begin{example} \label{ex:bump_diagram}
As in \cref{ex:kword_pro}, let $w = \A \A \B \B \C \A \C \C \B$. The two noncrossing matchings $\M^B_w$ and $\M^C_w$, drawn as arc diagrams, are
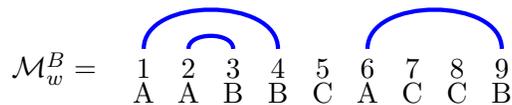
\begin{center}
\begin{tikzpicture}[scale=0.85]
\def \x {0.7};
\node at (-1*\x,0) {$\M^B_w=$};
\node (1) at (1*\x,0) {1};
\node (2) at (2*\x,0) {2};
\node (3) at (3*\x,0) {3};
\node (4) at (4*\x,0) {4};
\node (5) at (5*\x,0) {5};
\node (6) at (6*\x,0) {6};
\node (7) at (7*\x,0) {7};
\node (8) at (8*\x,0) {8};
\node (9) at (9*\x,0) {9};
\node at (1*\x,-0.55*\x) {$\A$};
\node at (2*\x,-0.55*\x) {$\A$};
\node at (3*\x,-0.55*\x) {$\B$};
\node at (4*\x,-0.55*\x) {$\B$};
\node at (5*\x,-0.55*\x) {$\C$};
\node at (6*\x,-0.55*\x) {$\A$};
\node at (7*\x,-0.55*\x) {$\C$};
\node at (8*\x,-0.55*\x) {$\C$};
\node at (9*\x,-0.55*\x) {$\B$};
\draw [ultra thick, blue] (1) to [out=90, in=90] (4);
\draw [ultra thick, blue] (2) to [out=90, in=90] (3);
\draw [ultra thick, blue] (6) to [out=90, in=90] (9);
\end{tikzpicture} \\ \begin{tikzpicture}[scale=0.85]
\def \x {0.7};
\node at (-1*\x,0) {$\M^C_w=$};
\node (1) at (1*\x,0) {1};
\node (2) at (2*\x,0) {2};
\node (3) at (3*\x,0) {3};
\node (4) at (4*\x,0) {4};
\node (5) at (5*\x,0) {5};
\node (6) at (6*\x,0) {6};
\node (7) at (7*\x,0) {7};
\node (8) at (8*\x,0) {8};
\node (9) at (9*\x,0) {9};
\node at (1*\x,-0.55*\x) {$\A$};
\node at (2*\x,-0.55*\x) {$\A$};
\node at (3*\x,-0.55*\x) {$\B$};
\node at (4*\x,-0.55*\x) {$\B$};
\node at (5*\x,-0.55*\x) {$\C$};
\node at (6*\x,-0.55*\x) {$\A$};
\node at (7*\x,-0.55*\x) {$\C$};
\node at (8*\x,-0.55*\x) {$\C$};
\node at (9*\x,-0.55*\x) {$\B$};
\draw [ultra thick, red, dashed] (1) to [out=90, in=90] (8);
\draw [ultra thick, red, dashed] (2) to [out=90, in=90] (5);
\draw [ultra thick, red, dashed] (6) to [out=90, in=90] (7);
\end{tikzpicture}
\end{center}
The Kreweras bump diagram $\D_w$ is obtained by placing these two arc diagrams on top of one another:
\begin{center}
\begin{tikzpicture}[scale=0.85]
\def \x {0.7};
\node at (-1*\x,0) {$\D_w=$};
\node (1) at (1*\x,0) {1};
\node (2) at (2*\x,0) {2};
\node (3) at (3*\x,0) {3};
\node (4) at (4*\x,0) {4};
\node (5) at (5*\x,0) {5};
\node (6) at (6*\x,0) {6};
\node (7) at (7*\x,0) {7};
\node (8) at (8*\x,0) {8};
\node (9) at (9*\x,0) {9};
\node at (1*\x,-0.55*\x) {$\A$};
\node at (2*\x,-0.55*\x) {$\A$};
\node at (3*\x,-0.55*\x) {$\B$};
\node at (4*\x,-0.55*\x) {$\B$};
\node at (5*\x,-0.55*\x) {$\C$};
\node at (6*\x,-0.55*\x) {$\A$};
\node at (7*\x,-0.55*\x) {$\C$};
\node at (8*\x,-0.55*\x) {$\C$};
\node at (9*\x,-0.55*\x) {$\B$};
\draw [ultra thick,blue] (1) to [out=90, in=90] (4);
\draw [ultra thick,blue] (2) to [out=90, in=90] (3);
\draw [ultra thick,blue] (6) to [out=90, in=90] (9);
\draw [ultra thick,red,dashed] (1) to [out=90, in=90] (8);
\draw [ultra thick,red,dashed] (2) to [out=90, in=90] (5);
\draw [ultra thick,red,dashed] (6) to [out=90, in=90] (7);
\end{tikzpicture}
\end{center}
\end{example}

Now we extract from $\D_w$ a permutation $\sigma_w \in \Sym_{3n}$.

\begin{figure}
\begin{center}
\begin{tikzpicture}[scale=1.35]
\node at (-1,1){(a)};
\tikzset{->-/.style={decoration={ markings, mark=at position 0.25 with {\arrow{>}}},postaction={decorate}}};
\draw [thin,blue] (0,0) to [out=70, in=180] (1,1);
\draw [thin,red,dashed] (-0.5,0) to [out=90, in=180] (1,0.5);
\filldraw[black] (0.25,0.5) circle (1.5pt);
\draw [thick,->-,->] (-0.7,0.1) to [out=90, in=180] (1,0.58);
\draw [thick,->-,->] (1,1.2) to [out=180, in=70] (-0.05,0.1);
\draw [thick,->-,->] (1.1,0.7) to [out=180, in=0] (0.3,0.7) to [out=60, in=180] (1.1,1.1);
\draw [thick,->-,->] (-0.25,0.1) to [out=55,in=180+60] (0,0.5) to [out=180,in=90] (-0.6,0.1) ;
\end{tikzpicture}  \qquad  \qquad \qquad \begin{tikzpicture}[scale=1.35]
\node at (-0.5,1){(b)};
\tikzset{->-/.style={decoration={ markings, mark=at position 0.25 with {\arrow{>}}},postaction={decorate}}};
\draw [thin,blue] (0,0) to [out=70, in=170] (1,0.35);
\draw [thin,red,dashed] (0,0) to [out=90, in=190] (1,0.8);
\filldraw[black] (0,0) circle (1.5pt);
\draw [thick,->-,->] (0.15,0) to [out=70, in=180] (1,0.2);
\draw [thick,->-,->] (1,0.9) to [out=180, in=90] (-0.15,0);
\draw [thick,->-,->] (1,0.45) to [out=170, in=10] (0.25,0.4) to [out=45,in=180+15] (1,0.7);
\end{tikzpicture}
\end{center}
\caption{The rules of the road when taking a trip in a Kreweras bump diagram: (a) depicts what happens at an internal crossing, and (b) depicts what happens at a boundary crossing. Note that the color of the arcs in the crossing is irrelevant. } \label{fig:rules_road}
\end{figure}

\begin{definition} \label{def:trip_perm}
  Let $w$ be a Kreweras word of length $3n$ and let $\D_w$ be its associated Kreweras bump diagram.  We define the \dfn{trip permutation} $\sigma_w \in \Sym_{3n}$ of $w$ as follows. For each $i\in\{1,\ldots,3n\}$, we define  $\sigma_w(i)$ by taking a \dfn{trip in $\D_w$ starting at $i$}, as we now describe. If $i$ is a closer of $\M^B_w\cup\M^C_w$, then we start our trip by walking from~$i$ towards $i'$ along the unique arc $(i',i)$ incident to~$i$; if $i$ is an opener of $\M^B_w\cup\M^C_w$, then we start our trip by walking from $i$ towards $i'$ along the arc $(i,i')$ incident to~$i$ with the smallest value of~$i'$. We continue walking until we encounter a crossing. Whenever we encounter a crossing of arcs $(a,b)$ and $(c,d)$ with $a\leq c < b < d$, we follow the \dfn{rules of the road}: we continue towards
\begin{itemize}
  \item $b$, if coming from $a$;
  \item $a$, if coming from $c$ -- this is a \dfn{left turn};
  \item $d$, if coming from $b$ -- this is a \dfn{right turn}; and
  \item $c$, if coming from $d$.
\end{itemize}
These rules of the road are depicted in \cref{fig:rules_road}. Finally, if $j$ is the terminal vertex of the trip, we set $\sigma_w(i)  \coloneqq  j$.
\end{definition}

\begin{example} \label{ex:trip_permutation}
As in \cref{ex:bump_diagram}, let $w = \A \A \B \B \C \A \C \C \B$.

First let us compute $\sigma_w(1)$. We start by walking along the arc~$(1,4)$ from $1$ towards $4$.  We encounter the crossing $\{(1,4),(2,5)\}$, but following the first rule we continue towards $4$, where we terminate. Thus, $\sigma_w(1)=4$.  This trip looks pictorially as follows:
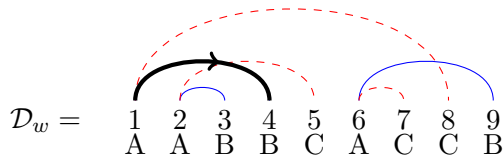
\begin{center}
\begin{tikzpicture}[scale=0.85]
\tikzset{->-/.style={decoration={ markings, mark=at position 0.6 with {\arrow{>}}},postaction={decorate}}};
\def \x {0.7};
\node at (-1*\x,0) {$\D_w=$};
\node (1) at (1*\x,0) {1};
\node (2) at (2*\x,0) {2};
\node (3) at (3*\x,0) {3};
\node (4) at (4*\x,0) {4};
\node (5) at (5*\x,0) {5};
\node (6) at (6*\x,0) {6};
\node (7) at (7*\x,0) {7};
\node (8) at (8*\x,0) {8};
\node (9) at (9*\x,0) {9};
\node at (1*\x,-0.55*\x) {$\A$};
\node at (2*\x,-0.55*\x) {$\A$};
\node at (3*\x,-0.55*\x) {$\B$};
\node at (4*\x,-0.55*\x) {$\B$};
\node at (5*\x,-0.55*\x) {$\C$};
\node at (6*\x,-0.55*\x) {$\A$};
\node at (7*\x,-0.55*\x) {$\C$};
\node at (8*\x,-0.55*\x) {$\C$};
\node at (9*\x,-0.55*\x) {$\B$};
\draw [thin,blue] (1) to [out=90, in=90] (4);
\draw [thin,blue] (2) to [out=90, in=90] (3);
\draw [thin,blue] (6) to [out=90, in=90] (9);
\draw [thin,red,dashed] (1) to [out=90, in=90] (8);
\draw [thin,red,dashed] (2) to [out=90, in=90] (5);
\draw [thin,red,dashed] (6) to [out=90, in=90] (7);
\draw [ultra thick, ->-] (1) to [out=90, in=90] (4);
\end{tikzpicture}
\end{center}

Next, let us compute $\sigma_w(4)$.  We start by walking along the arc~$(1,4)$ from $4$ towards $1$.  As we encounter the crossing $\{(1,4),(2,5)\}$ we turn right, and continue along the arc $(2,5)$ from $2$ towards $5$, where we terminate.  Thus, $\sigma_w(4)=5$:
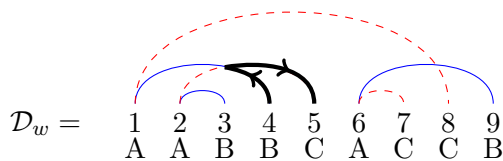
\begin{center}
\begin{tikzpicture}[scale=0.85]
\tikzset{->-/.style={decoration={ markings, mark=at position 0.6 with {\arrow{>}}},postaction={decorate}}};
\def \x {0.7};
\node at (-1*\x,0) {$\D_w=$};
\node (1) at (1*\x,0) {1};
\node (2) at (2*\x,0) {2};
\node (3) at (3*\x,0) {3};
\node (4) at (4*\x,0) {4};
\node (5) at (5*\x,0) {5};
\node (6) at (6*\x,0) {6};
\node (7) at (7*\x,0) {7};
\node (8) at (8*\x,0) {8};
\node (9) at (9*\x,0) {9};
\node at (1*\x,-0.55*\x) {$\A$};
\node at (2*\x,-0.55*\x) {$\A$};
\node at (3*\x,-0.55*\x) {$\B$};
\node at (4*\x,-0.55*\x) {$\B$};
\node at (5*\x,-0.55*\x) {$\C$};
\node at (6*\x,-0.55*\x) {$\A$};
\node at (7*\x,-0.55*\x) {$\C$};
\node at (8*\x,-0.55*\x) {$\C$};
\node at (9*\x,-0.55*\x) {$\B$};
\draw [thin,blue] (1) to [out=90, in=90] (4);
\draw [thin,blue] (2) to [out=90, in=90] (3);
\draw [thin,blue] (6) to [out=90, in=90] (9);
\draw [thin,red,dashed] (1) to [out=90, in=90] (8);
\draw [thin,red,dashed] (2) to [out=90, in=90] (5);
\draw [thin,red,dashed] (6) to [out=90, in=90] (7);
\draw [ultra thick, ->-] (4) to [out=90, in=-10]  (3*\x,1.25*\x);
\draw [ultra thick, ->-] (3*\x,1.25*\x) to [out=10, in=90]  (5);
\end{tikzpicture}
\end{center}

Next, let us compute $\sigma_w(3)$. We start by walking along the arc~$(2,3)$ from $3$ towards $2$.  As we encounter the boundary crossing $\{(2,3),(2,5)\}$ we turn right, and continue along the arc $(2,5)$ from $2$ towards~$5$.  Then we encounter the crossing $\{(1,4),(2,5)\}$ turn left, and continue along the arc $(1,4)$ from $4$ towards $1$.  At the boundary crossing $\{(1,4),(1,8)\}$ we turn right, and continue along the arc $(1,8)$ from $1$ towards $8$.  Next we encounter the crossing $\{(1,8),(6,9)\}$, but continue straight along the arc $(1,8)$ from $1$ towards $8$, where we terminate. Thus, $\sigma_w(3) = 8$.
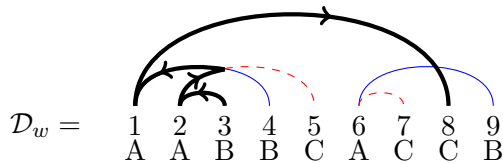
\begin{center}
\begin{tikzpicture}[scale=0.85]
\tikzset{->-/.style={decoration={ markings, mark=at position 0.6 with {\arrow{>}}},postaction={decorate}}};
\def \x {0.7};
\node at (-1*\x,0) {$\D_w=$};
\node (1) at (1*\x,0) {1};
\node (2) at (2*\x,0) {2};
\node (3) at (3*\x,0) {3};
\node (4) at (4*\x,0) {4};
\node (5) at (5*\x,0) {5};
\node (6) at (6*\x,0) {6};
\node (7) at (7*\x,0) {7};
\node (8) at (8*\x,0) {8};
\node (9) at (9*\x,0) {9};
\node at (1*\x,-0.55*\x) {$\A$};
\node at (2*\x,-0.55*\x) {$\A$};
\node at (3*\x,-0.55*\x) {$\B$};
\node at (4*\x,-0.55*\x) {$\B$};
\node at (5*\x,-0.55*\x) {$\C$};
\node at (6*\x,-0.55*\x) {$\A$};
\node at (7*\x,-0.55*\x) {$\C$};
\node at (8*\x,-0.55*\x) {$\C$};
\node at (9*\x,-0.55*\x) {$\B$};
\draw [thin,blue] (1) to [out=90, in=90] (4);
\draw [thin,blue] (2) to [out=90, in=90] (3);
\draw [thin,blue] (6) to [out=90, in=90] (9);
\draw [thin,red,dashed] (1) to [out=90, in=90] (8);
\draw [thin,red,dashed] (2) to [out=90, in=90] (5);
\draw [thin,red,dashed] (6) to [out=90, in=90] (7);
\draw [ultra thick, ->-] (3)  to [out=90, in=90]  (2);
\draw [ultra thick, ->-] (2)  to [out=90, in=190]  (3*\x,1.25*\x);
\draw [ultra thick, ->-] (3*\x,1.25*\x) to [out=170, in=90]  (1);
\draw [ultra thick, ->-] (1) to [out=90, in=90]  (8);
\end{tikzpicture}
\end{center}

Finally, let us compute $\sigma_w(7)$.  We start by walking along the arc~$(6,7)$ from $7$ towards $6$. There we encounter the boundary crossing $\{(6,7),(6,9)\}$ and turn right, so that we continue along the arc $(6,9)$ from~$6$ towards $9$.  As we encounter the crossing $\{(1,8),(6,9)\}$ we turn left, and continue along the arc $(1,8)$ from $8$ towards~$1$.  We encounter the boundary crossing $\{(1,4),(1,8)\}$, and terminate at $1$. Thus, $\sigma_w(7)=1$.
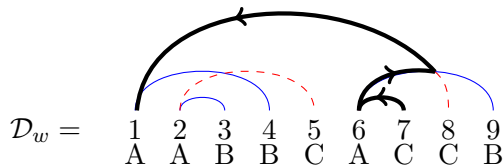
\begin{center}
\begin{tikzpicture}[scale=0.85]
\tikzset{->-/.style={decoration={ markings, mark=at position 0.6 with {\arrow{>}}},postaction={decorate}}};
\def \x {0.7};
\node at (-1*\x,0) {$\D_w=$};
\node (1) at (1*\x,0) {1};
\node (2) at (2*\x,0) {2};
\node (3) at (3*\x,0) {3};
\node (4) at (4*\x,0) {4};
\node (5) at (5*\x,0) {5};
\node (6) at (6*\x,0) {6};
\node (7) at (7*\x,0) {7};
\node (8) at (8*\x,0) {8};
\node (9) at (9*\x,0) {9};
\node at (1*\x,-0.55*\x) {$\A$};
\node at (2*\x,-0.55*\x) {$\A$};
\node at (3*\x,-0.55*\x) {$\B$};
\node at (4*\x,-0.55*\x) {$\B$};
\node at (5*\x,-0.55*\x) {$\C$};
\node at (6*\x,-0.55*\x) {$\A$};
\node at (7*\x,-0.55*\x) {$\C$};
\node at (8*\x,-0.55*\x) {$\C$};
\node at (9*\x,-0.55*\x) {$\B$};
\draw [thin,blue] (1) to [out=90, in=90] (4);
\draw [thin,blue] (2) to [out=90, in=90] (3);
\draw [thin,blue] (6) to [out=90, in=90] (9);
\draw [thin,red,dashed] (7.7*\x,1.3*\x) to [out=-40, in=90] (8);
\draw [thin,red,dashed] (2) to [out=90, in=90] (5);
\draw [thin,red,dashed] (6) to [out=90, in=90] (7);
\draw [ultra thick, ->-] (7)  to [out=90, in=90]  (6);
\draw [ultra thick, ->-] (6)  to [out=90, in=170]  (7.7*\x,1.3*\x) ;
\draw [ultra thick, ->-] (7.7*\x,1.3*\x) to [out=132, in=85]  (1);
\end{tikzpicture}
\end{center}
We could further compute $[\sigma_w(1),\ldots,\sigma_w(9)] = [4,3,8,5,2,7,1,9,6]$.
\end{example}

Next we establish some basic properties of $\sigma_w$.
\begin{prop} \label{prop:basic_perm_props} \,
Let $w$ be a Kreweras word of length $3n$.
\begin{enumerate}[(a)]
\item \label{item:sigma_perm} \Cref{def:trip_perm} yields a permutation $\sigma_w \in \Sym_{3n}$.
\item \label{item:no_double_descents} Let $1 \leq i \leq 3n$ with $w_i=\A$. Then $w_{\sigma(i)} \in \{\B,\C\}$.
\item \label{item:signs_flip} Let $1 \leq i \leq 3n$ with $w_i \in \{\B,\C\}$. Then either $w_{\sigma_w(i)}=-w_i$, or $w_{\sigma_w(i)}=\A$ and  $w_{\sigma_w(\sigma_w(i))}=-w_i$.
\item \label{item:anti_exceedances} We have \begin{align*}
\{i \in [3n]\colon w_i = \A\} &= \{i\in [3n]\colon \sigma_w^{-1}(i) > i\}; \\
\{i \in [3n]\colon \textrm{$w_i = \B$ or $w_i=\C$} \} &= \{i\in [3n]\colon \sigma_w^{-1}(i) < i\}.
\end{align*} In particular, $\sigma_w$ has no fixed points.
\end{enumerate}
\end{prop}

\begin{proof}
For \eqref{item:sigma_perm}: this follows from the fact that the rules of the road permute the incoming and outgoing directions locally at each crossing.

For \eqref{item:no_double_descents}: if $w_i=\A$, then the trip in $\D_w$ starting at~$i$ will never turn at a crossing. Hence $\sigma_w(i)$ will be the index of the nearer~$\B$ or~$\C$ that is matched with the $\A$ at $i$.

For \eqref{item:signs_flip}: observe that if $w_i\neq \A$, then the trip in $\D_w$ starting at~$i$ starts heading towards an index~$j$ with $w_j = \A$. Moreover, the sequence of turns such a trip makes is a right turn, followed by a left turn, followed by a right turn, et cetera. Whenever the trip turns right, it heads towards an index $j$ with $w_j \neq \A$, and whenever it turns left, it heads towards a $j$ with $w_j = \A$. If the trip turns an odd number of times in total, it terminates at a $j$ with $w_j\neq\A$, and because only oppositely colored arcs of~$\D_w$ cross, this means in fact $w_j = -w_i$. If the trip turns an even number of times it terminates at a $j$ with $w_j = \A$, and again because only oppositely colored arcs of~$\D_w$ cross, we see from the proof of \eqref{item:no_double_descents} above that $w_{\sigma_w(j)}=-w_i$.

For \eqref{item:anti_exceedances}: we will show that $\sigma_w(j) < j$ if and only if $\sigma_w(j)=\A$, for any $1 \leq j \leq 3n$. If $w_j=\A$ then this is clear from the proof of \eqref{item:no_double_descents} above. So suppose $w_j \neq \A$. If the trip in $\D_w$ starting at $j$ never turns, the claim is also clear. So suppose further the trip starting at $j$ does turn, and suppose the arcs we traverse along the trip are, in order, $(i_0,j_0=j), (i_1,j_1),\ldots, (i_{\ell},j_{\ell})$. Then, first note that $i_0 < j < j_1$. Furthermore, we claim that for all $2\leq k\leq \ell$, the arc $(i_k,j_k)$ \dfn{nests} $(i_{k-2},j_{k-2})$: i.e., $i_k < i_{k-2} < j_{k-2} < j_k$. Indeed, otherwise either we would not encounter a crossing with $(i_k,j_k)$ when traveling along $(i_{k-1},j_{k-1})$ from the crossing with $(i_{k-2},j_{k-2})$, or we would not turn at such a crossing. To finish the proof of the claim, note that if $\ell$ is even then $\sigma_w(j)=i_{\ell}$ and $w_{i_{\ell}}=\A$ (see the proof of~\eqref{item:signs_flip} above), and then the nesting property implies $i_{\ell} < i_0 < j$. Similarly, if $\ell$ is odd then $\sigma_w(j)=j_{\ell}$ and $w_{j_{\ell}}\neq \A$, and then the nesting property implies $j < j_1 < j_{\ell}$.
\end{proof}

The permutation $\sigma_w$ does not quite determine the Kreweras word $w$. For example, if~$w'$ is obtained from $w$ by swapping all $\B$'s for $\C$'s and vice-versa, then clearly we have $\sigma_w = \sigma_{w'}$. So we need to keep track of a little extra data along with $\sigma_w$. To that end, we define the map $\varepsilon_w\colon \{1,\ldots,3n\} \to \{\B, \C\}$ by setting
\[ \varepsilon_w(i)  \coloneqq  \begin{cases} w_{\sigma_w(i)} &\textrm{if $w_{\sigma_w(i)}\neq \A$}; \\ w_{\sigma_w(\sigma_w(i))} &\textrm{if $w_{\sigma_w(i)}=\A$}, \\ \end{cases}\]
for all $1 \leq i \leq 3n$. \cref{prop:basic_perm_props}~\eqref{item:no_double_descents} guarantees that $\varepsilon_w(i)\in \{\B, \C\}$. As a shorthand we write $\varepsilon_w=[\varepsilon_w(1),\ldots,\varepsilon_w(3n)]$. Thanks to \cref{prop:basic_perm_props}~\eqref{item:anti_exceedances}, the pair $(\sigma_w,\varepsilon_w)$ determines $w$:

\begin{cor} \label{cor:perm_epsilon_determine_kword}
For any Kreweras word $w$ of length $3n$,
\[ w_i = \begin{cases} \A &\textrm{if $\sigma_w^{-1}(i) > i$}; \\ \varepsilon_w(\sigma_w^{-1}(i)) &\textrm{otherwise}, \end{cases}\]
for all $1 \leq i \leq 3n$.
\end{cor}

We now come to the key lemma in the proof of our main result, which says that~$\sigma_w$ and $\varepsilon_w$ evolve in a simple way under promotion.

\begin{lemma} \label{lem:key}
Let $w$ be a Kreweras word of length $3n$. Then,
\begin{enumerate}[(a)]
\item \label{item:perm_rot} $\sigma_{\pro(w)} = \rot(\sigma_w)$;
\item \label{item:eps_rot} $\varepsilon_{\pro(w)} = [\varepsilon_w(2),\varepsilon_w(3),\ldots,\varepsilon_w(3n), -\varepsilon_w(1)]$.
\end{enumerate}
\end{lemma}

Before we prove \cref{lem:key}, let's see an example.

\begin{example} \label{ex:key_lem}
As in \cref{ex:trip_permutation}, let $w = \A \A \B \B \C \A \C \C \B$. We saw above that $\sigma_w = [4,3,8,5,2,7,1,9,6]$. We also have $\varepsilon_w=[\B, \B, \C, \C, \B, \C, \B, \B, \C]$.

As we saw in \cref{ex:kword_pro}, $\pro(w) = \A \B \A \C \A \C \C \B \B$. Its associated bump diagram is
\begin{center}
\begin{tikzpicture}[scale=0.85]
\def \x {0.7};
\node at (-1*\x,0) {$\D_{\pro(w)}=$};
\node (1) at (1*\x,0) {1};
\node (2) at (2*\x,0) {2};
\node (3) at (3*\x,0) {3};
\node (4) at (4*\x,0) {4};
\node (5) at (5*\x,0) {5};
\node (6) at (6*\x,0) {6};
\node (7) at (7*\x,0) {7};
\node (8) at (8*\x,0) {8};
\node (9) at (9*\x,0) {9};
\node at (1*\x,-0.55*\x) {$\A$};
\node at (2*\x,-0.55*\x) {$\B$};
\node at (3*\x,-0.55*\x) {$\A$};
\node at (4*\x,-0.55*\x) {$\C$};
\node at (5*\x,-0.55*\x) {$\A$};
\node at (6*\x,-0.55*\x) {$\C$};
\node at (7*\x,-0.55*\x) {$\C$};
\node at (8*\x,-0.55*\x) {$\B$};
\node at (9*\x,-0.55*\x) {$\B$};
\draw [ultra thick,blue] (1) to [out=90, in=90] (2);
\draw [ultra thick,blue] (3) to [out=90, in=90] (9);
\draw [ultra thick,blue] (5) to [out=90, in=90] (8);
\draw [ultra thick,red,dashed] (1) to [out=90, in=90] (7);
\draw [ultra thick,red,dashed] (3) to [out=90, in=90] (4);
\draw [ultra thick,red,dashed] (5) to [out=90, in=90] (6);
\end{tikzpicture}
\end{center}
From the diagram $\D_{\pro(w)}$ one could compute that $\sigma_{\pro(w)}=[2, 7, 4, 1, 6, 9, 8, 5, 3]$ and $\varepsilon_{\pro(w)}=[\B, \C, \C, \B, \C, \B, \B, \C, \C]$, in agreement with \cref{lem:key}.
\end{example}

We proceed to prove \cref{lem:key}.

\begin{proof}[Proof of \cref{lem:key}]
The key to the proof is the following procedure to obtain the Kreweras bump diagram $\D_{\pro(w)}$ from $\D_w$.

\begin{figure}
\begin{center}
 \begin{tikzpicture}[scale=0.75]
 \node (1) at (0,0) {$1$};
 \node (2) at (2,0) {$i_1$};
 \node (3) at (3,0) {$\ldots$};
 \node (4) at (4,0) {$i_m$};
 \node (5) at (5,0) {$b$};
 \node (6) at (6,0) {$j_m$};
 \node (7) at (7,0) {$\ldots$};
 \node (8) at (8,0) {$j_1$};
 \node (9) at (10,0) {$c$};
 \node (10) at (12,0) {$3n$};
 \draw [ultra thick,blue] (1) to [out=90, in=90] (5);
 \draw [ultra thick,red,dashed] (1) to [out=90, in=90] (9);
 \draw [ultra thick,red,dashed] (2) to [out=90, in=90] (8);
 \draw [ultra thick,red,dashed] (4) to [out=90, in=90] (6);
 \draw [ultra thick] (5.15,0.25) to [out=80, in=-40] (4.95,0.85);
 \draw [ultra thick,->] (4.95,0.85) to [out=0, in=150] (5.6,0.65);
 \draw [ultra thick] (4.25,1.5) to [out=140, in=0] (3.35,1.8);
 \draw [ultra thick,->] (3.35,1.8) to [out=20, in=180] (4.4,1.9);
 \draw [ultra thick] (1.4,1.75) to [out=180, in=40] (0.5,1.4);
 \draw [ultra thick,->] (0.5,1.4) to [out=60, in=200] (1.4,2.25);
 \draw [ultra thick] (3.8,0.45) to [out=70, in=180] (4.55,1.05);
 \draw [ultra thick,->] (4.55,1.05) to [out=180-35, in=-15] (3.9,1.3);
 \draw [ultra thick] (1.8,0.65) to [out=70, in=200] (2.525,1.6);
 \draw [ultra thick,->] (2.5,1.6) to [out=170, in=10] (1.6,1.55);
 \node at (14,2) {$\D_w$};
 \end{tikzpicture} \\ \hrulefill \\
 \begin{tikzpicture}[scale=0.75]
 \node (1) at (0,0) {$\;$};
 \node (2) at (2,0) {$i_1$};
 \node (3) at (3,0) {$\ldots$};
 \node (4) at (4,0) {$i_m$};
 \node (5) at (5,0) {$b$};
 \node (6) at (6,0) {$j_m$};
 \node (7) at (7,0) {$\ldots$};
 \node (8) at (8,0) {$j_1$};
 \node (9) at (10,0) {$c$};
 \node (10) at (12,0) {$3n$};
 \node (11) at (13,0) {$3n{+}1$};
 \draw [ultra thick,blue] (5) to [out=90, in=90] (11);
 \draw [ultra thick,red,dashed] (5) to [out=90, in=70] (6);
 \draw [ultra thick,red,dashed] (2) to [out=90, in=90] (9);
 \draw [ultra thick,red,dashed] (4) to [out=90, in=90] (8);
 \draw [ultra thick] (5.75,0.85) to [out=180,in=70] (5.2,0.7);
 \draw [ultra thick,->] (5.2,0.7) to [out=80, in=270-30] (5.65,1.4);
 \draw [ultra thick] (7.0,1.6) to [out=160,in=10] (5.9,1.7);
 \draw [ultra thick,->] (5.9,1.7) to [out=40,in=200] (7.0,2.2);
 \draw [ultra thick] (9.5,1.9) to [out=130,in=-20] (8.5,2.5);
 \draw [ultra thick,->] (8.5,2.5) to [out=10,in=180] (9.75,2.5);
 \draw [ultra thick] (4.9,0.7) to [out=80,in=220] (5.25,1.4);
 \draw [ultra thick,->] (5.25,1.4) to [out=210,in=60] (4.4,0.7);
 \draw [ultra thick] (6.25,2.2) to [out=30,in=200] (7.2,2.55);
 \draw [ultra thick,->] (7.2,2.55) to [out=170,in=0] (6.0,2.55);
 \node at (14,2) {$\widetilde{\N}^\B \cup \widetilde{\N}^\C$};
 \end{tikzpicture}
\end{center}
\caption{Figure illustrating the proof of \cref{lem:key}.} \label{fig:key_lem}
\end{figure}

Let $(1,b)\in\M^\B_w$ and $(1,c)\in\M^\C_w$, and suppose that~$b < c$ without loss of generality. Let $\widetilde{\N}^\B  \coloneqq  (\M^\B_w\setminus\{(1,b)\} )\cup \{(b, 3n+1)\}$.  This is a noncrossing matching, because $i < b < j < 3n+1$ would imply $1 < i < b < j$. Furthermore, let
\[ \widetilde{\N}^\C  \coloneqq  (\M^\C_w \setminus\{(1,c), (i_1,j_1),\dots,(i_m,j_m)\}) \cup \{(i_1,c),(i_2,j_1),\dots,(i_m,j_{m-1}),(b,j_m)\},\]
where $\{(i_1,j_1),\dots,(i_m,j_m)\}$ is the set of arcs in $\M^\C_w$ with
\[i_1 < \dots < i_m < b < j_m < \dots < j_1 < c,\]
that is, the set of arcs which cross $(1,b)$.  $\widetilde{\N}^\C$ is a noncrossing matching: for example, suppose that an arc  $(i, j)$ satisfies $i < i_\ell < j < j_{\ell-1}$; then we have in fact $j < b$ because $(i,j)$ cannot cross $(1,b)$.  The other cases are dealt with similarly.

Let $\varepsilon\in\{\B, \C\}$ and let $\N^\varepsilon$ be obtained from $\widetilde{\N}^\varepsilon$ by replacing every arc $(i,j)$ with $(i-1, j-1)$.  Then $\N^\varepsilon$  is a noncrossing matching of~$\{i\in[3n]\colon \pro(w)_i \neq \varepsilon\}$ with set of openers~$\{i\in [3n]\colon \pro(w)_i=\textrm{$\A$}\}$ and set of closers~$\{i\in [3n]\colon \pro(w)_i=\varepsilon\}$.  Since the set of openers and the set of closers uniquely determine a noncrossing matching, $\D_{\pro(w)} = \N^\B\cup\N^\C$.

We define the \dfn{trip permutation} of $\widetilde{\N}^\B\cup\widetilde{\N}^\C$ in the obvious way, by taking trips starting at $2 \leq i \leq 3n+1$ and using the rules of the road. We now show that $\sigma_w$ coincides with the trip permutation of $\widetilde{\N}^\B\cup\widetilde{\N}^\C$, provided we identify $1$ and $3n+1$.  To do so, we subdivide every arc of $\D_w$ crossing $(1,b)$ into an initial, a middle and a final part, such that the middle part contains the crossing with $(1,b)$ and no other crossings.  Additionally, slightly abusing language, we say that the arc $(1,b)$ only has a middle part and $(1,c)$ consists only of a middle part (containing only the crossing with $(1,b)$) and a final part.

Similarly, we subdivide every arc of $\widetilde{\N}^\B\cup\widetilde{\N}^\C$ in $\{(i_1,c),(i_2,j_1),\dots,(i_m,j_{m-1})\}$ into an initial, a middle and a final part, such that the middle part contains the crossing with $(b,3n+1)$ and no other crossings.  Additionally, we say that $(b,j_m)$ consists only of a middle part (containing only the crossing with $(b,3n+1)$) and a final part, and $(b,3n+1)$ only has a middle part.

We now note that the initial and the final parts of the arcs in $\D_w$ and $\widetilde{\N}^\B\cup\widetilde{\N}^\C$ are identical. It therefore suffices to check that the portions of a trip proceeding in a middle part also begin and end at the same arcs in $\D_w$ and $\widetilde{\N}^\B\cup\widetilde{\N}^\C$ (provided we identify $1$ and $3n+1$). Labeling the beginnings of the middle parts in~$\D_w$ from left to right with $1, s_1,\dots, s_m$ and their endings with $b, t_m,\dots, t_1, t_0$ and the beginnings of the middle parts in $\widetilde{\N}^\B\cup\widetilde{\N}^\C$ from left to right with $s_1,\dots, s_m, b$ and their endings with $t_m,\dots, t_1, t_0, 1=3n+1$, we find that in both cases the trip connects these as follows:
  \begin{gather*}
    1\to b, b\to t_m, t_0\to 1,\\
    s_k\to t_{k-1} \textrm{ and } t_k\to s_k\text{ for all $k \geq 1$}.
  \end{gather*}
This is depicted in \cref{fig:key_lem}.

This proves~\eqref{item:perm_rot}. For~\eqref{item:eps_rot}: define $\varepsilon'_w = [\varepsilon'_w(1),\ldots,\varepsilon'_w(3n)]$ by
\[ \varepsilon'_w(i)  \coloneqq  \begin{cases} w_{i} &\textrm{if $w_{i}\neq \A$}; \\ w_{\sigma_w(i)} &\textrm{if $w_{i}=\A$}, \\ \end{cases}\]
Recall (see the proof of \cref{prop:basic_perm_props}\eqref{item:no_double_descents}) that for $i$ with $w_i=\A$, $\sigma_w(i)$ is just the index of the nearer $\B$ or $\C$ that is matched with the $\A$ at $i$. Hence
\[ \varepsilon'_{\pro(w)} = [\varepsilon'_2(w), \varepsilon'_3(w), \ldots, \varepsilon'_{\iota(w)-1}(w), -\varepsilon'_{\iota(w)}(w),\varepsilon'_{\iota(w)+1}(w),\ldots,\varepsilon'_{3n}(w),\varepsilon'_{1}(w)] \]
Together with~\eqref{item:perm_rot}, this observation proves~\eqref{item:eps_rot}.
\end{proof}

\Cref{thm:main} follows easily from \cref{lem:key}:

\begin{proof}[Proof of \cref{thm:main}]
\Cref{lem:key} says $\varepsilon_{\pro^{3n}(w)} = [-\varepsilon_{w}(1),-\varepsilon_{w}(2),\ldots,-\varepsilon_{w}(3n)]$ and $\sigma_{\pro^{3n}(w)} = \sigma_w$. Thus \cref{cor:perm_epsilon_determine_kword} implies that $\pro^{3n}(w)$ is obtained from $w$ by swapping all $\B$'s for $\C$'s and vice-versa, as claimed. \end{proof}

\section{Evacuation of Kreweras words} \label{sec:evac}

Via the bijection between Kreweras words of length $3n$ and linear extensions of~$V(n)$ described in \cref{prop:kword_poset_pro}, we can also view evacuation as acting on the set of Kreweras words. In this section we will describe the evacuation of a Kreweras word, using some of the machinery from \cref{sec:proof}.  As we will see, just as with promotion, evacuation has a very simple effect on $\sigma_w$ and $\varepsilon_w$.

In order to study evacuation of Kreweras words, we will employ another formulation of promotion and evacuation of linear extensions in terms of \dfn{growth diagrams}. This approach is discussed in~\cite[\S 5]{stanley2009promotion}; it is essentially due to Fomin (see, e.g.,~\cite[Chapter~7: Appendix 1]{stanley1999ec2}).

Let $P$ be a poset with $\ell$ elements. An \dfn{order ideal} of $P$ is a subset $I\subseteq P$ that is downwards-closed, i.e., for which $q \in I$ and $p \leq q \in P$ implies that~$p \in I$. The set of order ideals of $P$ is denoted $\mathcal{J}(P)$. A linear extension $(p_1,p_2,\ldots,p_{\ell}) \in \mathcal{L}(P)$ corresponds to a chain $I_0 \subset I_1 \subset \cdots \subset I_{\ell} \in \mathcal{J}(P)$ of order ideals of length $\ell$, where we set~$I_{i}  \coloneqq  \{p_1,p_2,\ldots,p_i\}$ for $0 \leq i \leq \ell$. This sets up a (well-known) bijection between linear extensions of $P$ and maximal chains of order ideals of $P$.

\begin{definition} \label{def:growth_diagram}
Let $L \in \mathcal{L}(P)$ be a linear extension. The \dfn{growth diagram} of $L$ is a labeling of the subset $D  \coloneqq  \{(x,y) \in \mathbb{Z}^2\colon -y-\ell \leq x \leq -y\}$ of the two-dimensional grid $\mathbb{Z}^2$ by order ideals $I_{(x,y)} \in \mathcal{J}(P)$, $(x,y)\in D$, subject to the following conditions:
\begin{itemize}
\item $I_{(-\ell+k,-k)}=\varnothing$ and $I_{(k,-k)}=P$ for all $k\in \mathbb{Z}$;
\item $I_{(-\ell,0)} \subset I_{(-\ell+1,0)} \subset \cdots \subset I_{(0,0)}$ is the chain corresponding to $L$;
\item for any four points $(x,y), (x,y+1), (x+1,y), (x+1,y+1) \in D$, the labeling obeys the following \dfn{local rule}: if $I_{(x,y)} = I$, $I_{(x,y+1)} = I\cup \{p\}$, and $I_{(x+1,y+1)} = I\cup \{p,q\}$, then
\[ I_{(x+1,y)} = \begin{cases}  I\cup\{q\}  &\textrm{if $p$ and $q$ are incomparable in $P$}; \\ \ I\cup \{p\} &\textrm{if $p < q$ in $P$}. \end{cases}\]
This rule is depicted in \cref{fig:growth_diagram_rule}.
\end{itemize}
\end{definition}

\begin{figure}
\begin{center}
\begin{tikzpicture}[scale=0.5]
\node[shape=circle,fill=black,inner sep=1.5,label={left:$I$}] (1) at (0,0) {};
\node[shape=circle,fill=black,inner sep=1.5,label={left:$I\cup\{p\}$}] (2) at (0,3) {};
\node[shape=circle,fill=black,inner sep=1.5,label={right:$I\cup\{p,q\}$}] (3) at (3,3) {};
\node[shape=circle,fill=black,inner sep=1.5,label={right:$\begin{cases} I\cup\{q\} &\textrm{if $p \mid\mid q$}; \\ I\cup\{p\} &\textrm{if $p < q$}. \end{cases}$}] (4) at (3,0) {};
\draw[thick] (1)--(2)--(3)--(4)--(1);
\end{tikzpicture}
\end{center}
\caption{The local rule for growth diagrams of linear extensions.} \label{fig:growth_diagram_rule}
\end{figure}
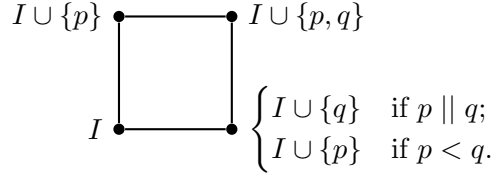

In the following proposition we summarize the basic results about growth diagrams of linear extensions. The essential idea is that all paths of length~$\ell$ from a point of the form $(-\ell+k,-k)$ to a point of the form $(j,-j)$ correspond to linear extensions, and the local rule in \cref{fig:growth_diagram_rule} reflects the behavior of the involutions $\tau_i$ from \cref{sec:promotion}. See~\cite[\S 5]{stanley2009promotion} for the details and references.
\begin{prop} \label{prop:growth_diagram_basics}
For any $L \in \mathcal{L}(P)$,
\begin{itemize}
\item the growth diagram of $L$ is well-defined, i.e., there is a unique order ideal labeling $I_{(x,y)}\in \mathcal{J}(P)$, $(x,y)\in D$ satisfying the conditions in \cref{def:growth_diagram};
\item for any $k \in \mathbb{Z}$, the chain $I_{(-\ell+k,-k)} \subset I_{(-\ell+k+1,-k)} \subset \cdots \subset I_{(k,-k)}$ corresponds to~$\pro^{k}(L)$;
\item the chain $I_{(0,-\ell)} \subset I_{(0,-\ell+1)} \subset \cdots \subset I_{(0,0)}$ corresponds to~$\evac(L)$.
\end{itemize}
\end{prop}

\begin{example}
Let $P$ be the following three-element poset:
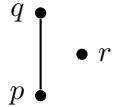
\begin{center}
\begin{tikzpicture}[scale=1.1]
\node[shape=circle,fill=black,inner sep=1.5,label={left:$p$}] (1) at (0,0) {};
\node[shape=circle,fill=black,inner sep=1.5,label={left:$q$}] (2) at (0,1) {};
\node[shape=circle,fill=black,inner sep=1.5,label={right:$r$}] (3) at (0.5,0.5) {};
\draw[thick] (1)--(2);
\end{tikzpicture}
\end{center}
Consider the linear extension $L=(p,q,r)\in \mathcal{L}(P)$. Then, writing subsets as strings for shorthand, the portion of the growth diagram of $L$ with $y$ coordinate between~$0$ and~$-3$ is
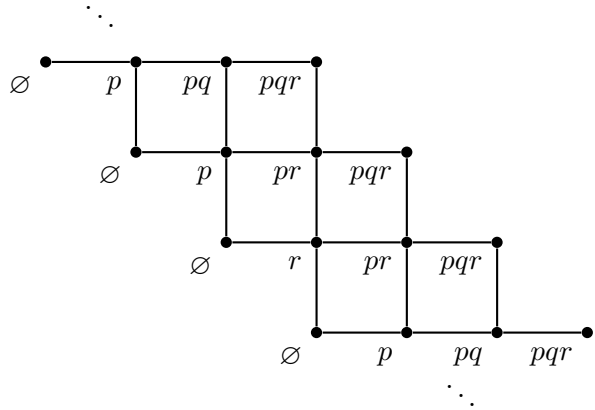
\begin{center}
\begin{tikzpicture}[scale=1.2]
\node at (0.6,0.6) {$\ddots$};
\node[shape=circle,fill=black,inner sep=1.5,label={below left:$\varnothing$}] (A1) at (0,0) {};
\node[shape=circle,fill=black,inner sep=1.5,label={below left:$p$}] (A2) at (1,0) {};
\node[shape=circle,fill=black,inner sep=1.5,label={below left:$pq$}] (A3) at (2,0) {};
\node[shape=circle,fill=black,inner sep=1.5,label={below left:$pqr$}] (A4) at (3,0) {};
\node[shape=circle,fill=black,inner sep=1.5,label={below left:$\varnothing$}] (B1) at (1,-1) {};
\node[shape=circle,fill=black,inner sep=1.5,label={below left:$p$}] (B2) at (2,-1) {};
\node[shape=circle,fill=black,inner sep=1.5,label={below left:$pr$}] (B3) at (3,-1) {};
\node[shape=circle,fill=black,inner sep=1.5,label={below left:$pqr$}] (B4) at (4,-1) {};
\node[shape=circle,fill=black,inner sep=1.5,label={below left:$\varnothing$}] (C1) at (2,-2) {};
\node[shape=circle,fill=black,inner sep=1.5,label={below left:$r$}] (C2) at (3,-2) {};
\node[shape=circle,fill=black,inner sep=1.5,label={below left:$pr$}] (C3) at (4,-2) {};
\node[shape=circle,fill=black,inner sep=1.5,label={below left:$pqr$}] (C4) at (5,-2) {};
\node[shape=circle,fill=black,inner sep=1.5,label={below left:$\varnothing$}] (D1) at (3,-3) {};
\node[shape=circle,fill=black,inner sep=1.5,label={below left:$p$}] (D2) at (4,-3) {};
\node[shape=circle,fill=black,inner sep=1.5,label={below left:$pq$}] (D3) at (5,-3) {};
\node[shape=circle,fill=black,inner sep=1.5,label={below left:$pqr$}] (D4) at (6,-3) {};
\node at (4.6,-3.6) {$\ddots$};
\draw[thick] (A1)--(A2)--(A3)--(A4);
\draw[thick] (B1)--(B2)--(B3)--(B4);
\draw[thick] (C1)--(C2)--(C3)--(C4);
\draw[thick] (D1)--(D2)--(D3)--(D4);
\draw[thick] (A2)--(B1);
\draw[thick] (A3)--(B2)--(C1);
\draw[thick] (A4)--(B3)--(C2)--(D1);
\draw[thick] (B4)--(C3)--(D2);
\draw[thick] (C4)--(D3);
\end{tikzpicture}
\end{center}
From this diagram we can read off $\pro^3(L)=L$ and $\evac(L) = (r,p,q)$.
\end{example}

\Cref{prop:growth_diagram_basics} implies that simple geometric operations on growth diagrams have combinatorial meaning:

\begin{cor} \label{cor:growth_diagram_symmetry}
For any $L \in \mathcal{L}(P)$,
\begin{enumerate}[(a)]
\item \label{item:trans_symmetry} for any $k \in \mathbb{Z}$, the growth diagram for $\pro^{k}(L)$ is obtained from the growth diagram for $L$ by translating by the vector $(-k,k)$;
\item \label{item:xy_symmetry} the growth diagram for $\evac(L)$ is obtained from the growth diagram for~$L$ by reflecting across the line $y=x$.
\end{enumerate}
\end{cor}
\begin{proof}
The first bulleted item is an immediate consequence of \cref{prop:growth_diagram_basics}. The second is also an immediate consequence of \cref{prop:growth_diagram_basics}, as soon as one observes that the local rule in \cref{fig:growth_diagram_rule} is symmetric under swapping $x$- and $y$-coordinates.
\end{proof}

Via the geometric operations described in \cref{cor:growth_diagram_symmetry}, the basic properties concerning promotion and evacuation summarized in \cref{prop:pro_evac_basics} are easily obtained via this growth diagram approach.

Now let's think about what growth diagrams look like for our poset of interest, $V(n)$. There is an obvious bijection
\[\mathcal{J}(V(n)) \simeq J(n)  \coloneqq  \{(a,b,c)\in \mathbb{N}^3\colon b,c\leq a \leq n\},\]
which sends $I$ to $(a,b,c)$ where $a= \max\{i \in \mathbb{N}\colon \{\A_1,\A_2,\ldots,\A_i\}\subseteq I\}$ and similarly for $b$ and $c$. We consider growth diagrams for linear extensions of~$V(n)$ labeled by elements of $J(n)$ via this bijection. The local rule then becomes:  if $I_{(x,y)} = (a,b,c)$, $I_{(x,y+1)} = (a,b,c)+e_i$, and $I_{(x+1,y+1)} = (a,b,c)+e_i+e_j$, then
\[ I_{(x+1,y)} = \begin{cases} (a,b,c)+e_j &\textrm{if $(a,b,c)+e_j \in J(n)$}; \\ (a,b,c)+e_i  &\textrm{otherwise}. \end{cases}\]
Here we have $i,j \in \{\A,\B,\C\}$, and we use the conventions that $e_{\A}  \coloneqq  (1,0,0)$, $e_{\B}  \coloneqq  (0,1,0)$, and $e_{\C}  \coloneqq  (0,0,1)$. This is depicted in \cref{fig:growth_diagram_rule_vn}.

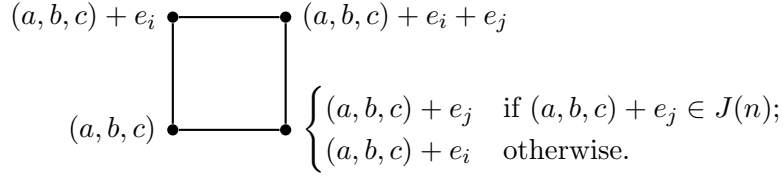
\begin{figure}
\begin{center}
\begin{tikzpicture}[scale=0.5]
\node[shape=circle,fill=black,inner sep=1.5,label={left:$(a,b,c)$}] (1) at (0,0) {};
\node[shape=circle,fill=black,inner sep=1.5,label={left:$(a,b,c)+e_i$}] (2) at (0,3) {};
\node[shape=circle,fill=black,inner sep=1.5,label={right:$(a,b,c)+e_i+e_j$}] (3) at (3,3) {};
\node[shape=circle,fill=black,inner sep=1.5,label={right:$\begin{cases}(a,b,c)+e_j &\textrm{if $(a,b,c)+e_j \in J(n)$}; \\ (a,b,c)+e_i &\textrm{otherwise}. \end{cases}$}] (4) at (3,0) {};
\draw[thick] (1)--(2)--(3)--(4)--(1);
\end{tikzpicture}
\end{center}
\caption{The local rule for growth diagrams of Kreweras words.} \label{fig:growth_diagram_rule_vn}
\end{figure}

Let us further decorate the growth diagram of a linear extension of $V(n)$ in the following way.  We refer to $(x,y), (x,y+1), (x+1,y), (x+1,y+1)$ as the \dfn{square in position $(x,y)$}. If in a growth diagram of a linear extension of~$V(n)$ these four points constitute an ``otherwise'' case in \cref{fig:growth_diagram_rule_vn}, then we \dfn{fill} this square with $j \in \{\B,\C\}$, where $e_j$ is as in that figure.

Via the bijection of \cref{prop:kword_poset_pro}, linear extensions of $V(n)$ are the same as Kreweras words of length $3n$. In this way, for such a Kreweras word~$w$ we obtain a labeling of~$\{(x,y) \in \mathbb{Z}^2\colon -y-3n \leq x \leq -y\}$ by~$J(n) = \{(a,b,c)\in \mathbb{N}^3\colon b,c\leq a \leq n\}$, which furthermore has some of its squares filled with $\B$'s and $\C$'s. We call this whole object the \dfn{decorated growth diagram} of $w$.

From now on in this section we will work with decorated growth diagrams of Kreweras words of length $3n$.

\begin{example} \label{ex:kword_growth_diagram}
As in \cref{ex:key_lem}, let $w = \A \A \B \B \C \A \C \C \B$. Then \cref{fig:kword_growth_diagram_ex} depicts the portion of the decorated growth diagram for $w$ with $y$ coordinate between~$0$ and~$-9$. In this figure we use the string $abc$ as shorthand for $(a,b,c) \in J(n)$.
\end{example}

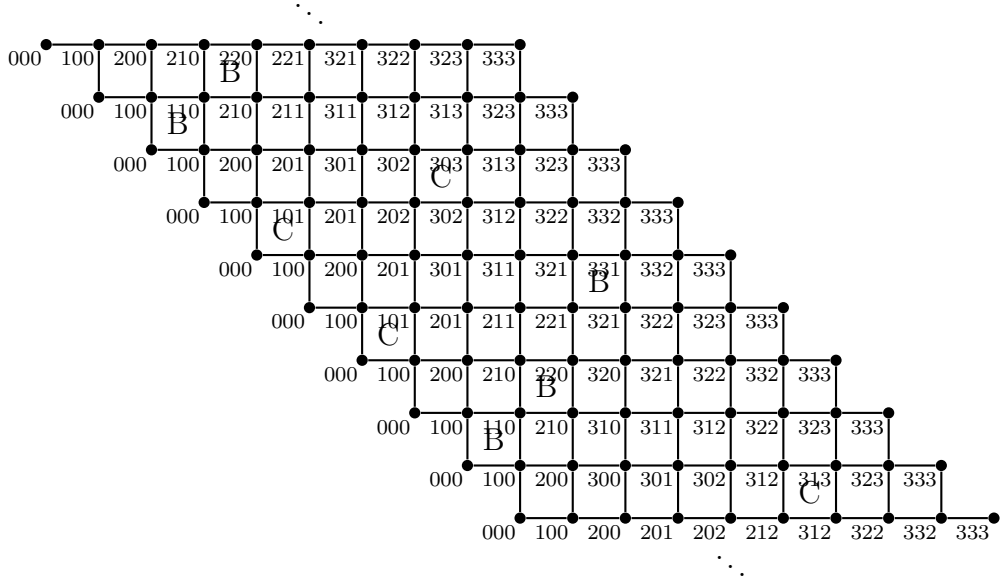
\begin{figure}
\begin{center}
\makebox[\textwidth]{\begin{tikzpicture}[scale=0.7]

\node[shape=circle,fill=black,inner sep=1.5,label={[xshift=-8,yshift=-14]:{\scriptsize $000$}}] (A0) at (0,0) {};
\node[shape=circle,fill=black,inner sep=1.5,label={[xshift=-8,yshift=-14]:{\scriptsize $100$}}] (A1) at (1,0) {};
\node[shape=circle,fill=black,inner sep=1.5,label={[xshift=-8,yshift=-14]:{\scriptsize $200$}}] (A2) at (2,0) {};
\node[shape=circle,fill=black,inner sep=1.5,label={[xshift=-8,yshift=-14]:{\scriptsize $210$}}] (A3) at (3,0) {};
\node[shape=circle,fill=black,inner sep=1.5,label={[xshift=-8,yshift=-14]:{\scriptsize $220$}}] (A4) at (4,0) {};
\node[shape=circle,fill=black,inner sep=1.5,label={[xshift=-8,yshift=-14]:{\scriptsize $221$}}] (A5) at (5,0) {};
\node[shape=circle,fill=black,inner sep=1.5,label={[xshift=-8,yshift=-14]:{\scriptsize $321$}}] (A6) at (6,0) {};
\node[shape=circle,fill=black,inner sep=1.5,label={[xshift=-8,yshift=-14]:{\scriptsize $322$}}] (A7) at (7,0) {};
\node[shape=circle,fill=black,inner sep=1.5,label={[xshift=-8,yshift=-14]:{\scriptsize $323$}}] (A8) at (8,0) {};
\node[shape=circle,fill=black,inner sep=1.5,label={[xshift=-8,yshift=-14]:{\scriptsize $333$}}] (A9) at (9,0) {};
\node at (3.5,-0.5) {\bf \large $\B$};

\node[shape=circle,fill=black,inner sep=1.5,label={[xshift=-8,yshift=-14]:{\scriptsize $000$}}] (B0) at (1,-1) {};
\node[shape=circle,fill=black,inner sep=1.5,label={[xshift=-8,yshift=-14]:{\scriptsize $100$}}] (B1) at (2,-1) {};
\node[shape=circle,fill=black,inner sep=1.5,label={[xshift=-8,yshift=-14]:{\scriptsize $110$}}] (B2) at (3,-1) {};
\node[shape=circle,fill=black,inner sep=1.5,label={[xshift=-8,yshift=-14]:{\scriptsize $210$}}] (B3) at (4,-1) {};
\node[shape=circle,fill=black,inner sep=1.5,label={[xshift=-8,yshift=-14]:{\scriptsize $211$}}] (B4) at (5,-1) {};
\node[shape=circle,fill=black,inner sep=1.5,label={[xshift=-8,yshift=-14]:{\scriptsize $311$}}] (B5) at (6,-1) {};
\node[shape=circle,fill=black,inner sep=1.5,label={[xshift=-8,yshift=-14]:{\scriptsize $312$}}] (B6) at (7,-1) {};
\node[shape=circle,fill=black,inner sep=1.5,label={[xshift=-8,yshift=-14]:{\scriptsize $313$}}] (B7) at (8,-1) {};
\node[shape=circle,fill=black,inner sep=1.5,label={[xshift=-8,yshift=-14]:{\scriptsize $323$}}] (B8) at (9,-1) {};
\node[shape=circle,fill=black,inner sep=1.5,label={[xshift=-8,yshift=-14]:{\scriptsize $333$}}] (B9) at (10,-1) {};
\node at (2.5,-1.5) {\bf \large $\B$};

\node[shape=circle,fill=black,inner sep=1.5,label={[xshift=-8,yshift=-14]:{\scriptsize $000$}}] (C0) at (2,-2) {};
\node[shape=circle,fill=black,inner sep=1.5,label={[xshift=-8,yshift=-14]:{\scriptsize $100$}}] (C1) at (3,-2) {};
\node[shape=circle,fill=black,inner sep=1.5,label={[xshift=-8,yshift=-14]:{\scriptsize $200$}}] (C2) at (4,-2) {};
\node[shape=circle,fill=black,inner sep=1.5,label={[xshift=-8,yshift=-14]:{\scriptsize $201$}}] (C3) at (5,-2) {};
\node[shape=circle,fill=black,inner sep=1.5,label={[xshift=-8,yshift=-14]:{\scriptsize $301$}}] (C4) at (6,-2) {};
\node[shape=circle,fill=black,inner sep=1.5,label={[xshift=-8,yshift=-14]:{\scriptsize $302$}}] (C5) at (7,-2) {};
\node[shape=circle,fill=black,inner sep=1.5,label={[xshift=-8,yshift=-14]:{\scriptsize $303$}}] (C6) at (8,-2) {};
\node[shape=circle,fill=black,inner sep=1.5,label={[xshift=-8,yshift=-14]:{\scriptsize $313$}}] (C7) at (9,-2) {};
\node[shape=circle,fill=black,inner sep=1.5,label={[xshift=-8,yshift=-14]:{\scriptsize $323$}}] (C8) at (10,-2) {};
\node[shape=circle,fill=black,inner sep=1.5,label={[xshift=-8,yshift=-14]:{\scriptsize $333$}}] (C9) at (11,-2) {};
\node at (7.5,-2.5) {\bf \large $\C$};

\node[shape=circle,fill=black,inner sep=1.5,label={[xshift=-8,yshift=-14]:{\scriptsize $000$}}] (D0) at (3,-3) {};
\node[shape=circle,fill=black,inner sep=1.5,label={[xshift=-8,yshift=-14]:{\scriptsize $100$}}] (D1) at (4,-3) {};
\node[shape=circle,fill=black,inner sep=1.5,label={[xshift=-8,yshift=-14]:{\scriptsize $101$}}] (D2) at (5,-3) {};
\node[shape=circle,fill=black,inner sep=1.5,label={[xshift=-8,yshift=-14]:{\scriptsize $201$}}] (D3) at (6,-3) {};
\node[shape=circle,fill=black,inner sep=1.5,label={[xshift=-8,yshift=-14]:{\scriptsize $202$}}] (D4) at (7,-3) {};
\node[shape=circle,fill=black,inner sep=1.5,label={[xshift=-8,yshift=-14]:{\scriptsize $302$}}] (D5) at (8,-3) {};
\node[shape=circle,fill=black,inner sep=1.5,label={[xshift=-8,yshift=-14]:{\scriptsize $312$}}] (D6) at (9,-3) {};
\node[shape=circle,fill=black,inner sep=1.5,label={[xshift=-8,yshift=-14]:{\scriptsize $322$}}] (D7) at (10,-3) {};
\node[shape=circle,fill=black,inner sep=1.5,label={[xshift=-8,yshift=-14]:{\scriptsize $332$}}] (D8) at (11,-3) {};
\node[shape=circle,fill=black,inner sep=1.5,label={[xshift=-8,yshift=-14]:{\scriptsize $333$}}] (D9) at (12,-3) {};
\node at (4.5,-3.5) {\bf \large $\C$};

\node[shape=circle,fill=black,inner sep=1.5,label={[xshift=-8,yshift=-14]:{\scriptsize $000$}}] (E0) at (4,-4) {};
\node[shape=circle,fill=black,inner sep=1.5,label={[xshift=-8,yshift=-14]:{\scriptsize $100$}}] (E1) at (5,-4) {};
\node[shape=circle,fill=black,inner sep=1.5,label={[xshift=-8,yshift=-14]:{\scriptsize $200$}}] (E2) at (6,-4) {};
\node[shape=circle,fill=black,inner sep=1.5,label={[xshift=-8,yshift=-14]:{\scriptsize $201$}}] (E3) at (7,-4) {};
\node[shape=circle,fill=black,inner sep=1.5,label={[xshift=-8,yshift=-14]:{\scriptsize $301$}}] (E4) at (8,-4) {};
\node[shape=circle,fill=black,inner sep=1.5,label={[xshift=-8,yshift=-14]:{\scriptsize $311$}}] (E5) at (9,-4) {};
\node[shape=circle,fill=black,inner sep=1.5,label={[xshift=-8,yshift=-14]:{\scriptsize $321$}}] (E6) at (10,-4) {};
\node[shape=circle,fill=black,inner sep=1.5,label={[xshift=-8,yshift=-14]:{\scriptsize $331$}}] (E7) at (11,-4) {};
\node[shape=circle,fill=black,inner sep=1.5,label={[xshift=-8,yshift=-14]:{\scriptsize $332$}}] (E8) at (12,-4) {};
\node[shape=circle,fill=black,inner sep=1.5,label={[xshift=-8,yshift=-14]:{\scriptsize $333$}}] (E9) at (13,-4) {};
\node at (10.5,-4.5) {\bf \large $\B$};

\node[shape=circle,fill=black,inner sep=1.5,label={[xshift=-8,yshift=-14]:{\scriptsize $000$}}] (F0) at (5,-5) {};
\node[shape=circle,fill=black,inner sep=1.5,label={[xshift=-8,yshift=-14]:{\scriptsize $100$}}] (F1) at (6,-5) {};
\node[shape=circle,fill=black,inner sep=1.5,label={[xshift=-8,yshift=-14]:{\scriptsize $101$}}] (F2) at (7,-5) {};
\node[shape=circle,fill=black,inner sep=1.5,label={[xshift=-8,yshift=-14]:{\scriptsize $201$}}] (F3) at (8,-5) {};
\node[shape=circle,fill=black,inner sep=1.5,label={[xshift=-8,yshift=-14]:{\scriptsize $211$}}] (F4) at (9,-5) {};
\node[shape=circle,fill=black,inner sep=1.5,label={[xshift=-8,yshift=-14]:{\scriptsize $221$}}] (F5) at (10,-5) {};
\node[shape=circle,fill=black,inner sep=1.5,label={[xshift=-8,yshift=-14]:{\scriptsize $321$}}] (F6) at (11,-5) {};
\node[shape=circle,fill=black,inner sep=1.5,label={[xshift=-8,yshift=-14]:{\scriptsize $322$}}] (F7) at (12,-5) {};
\node[shape=circle,fill=black,inner sep=1.5,label={[xshift=-8,yshift=-14]:{\scriptsize $323$}}] (F8) at (13,-5) {};
\node[shape=circle,fill=black,inner sep=1.5,label={[xshift=-8,yshift=-14]:{\scriptsize $333$}}] (F9) at (14,-5) {};
\node at (6.5,-5.5) {\bf \large $\C$};

\node[shape=circle,fill=black,inner sep=1.5,label={[xshift=-8,yshift=-14]:{\scriptsize $000$}}] (G0) at (6,-6) {};
\node[shape=circle,fill=black,inner sep=1.5,label={[xshift=-8,yshift=-14]:{\scriptsize $100$}}] (G1) at (7,-6) {};
\node[shape=circle,fill=black,inner sep=1.5,label={[xshift=-8,yshift=-14]:{\scriptsize $200$}}] (G2) at (8,-6) {};
\node[shape=circle,fill=black,inner sep=1.5,label={[xshift=-8,yshift=-14]:{\scriptsize $210$}}] (G3) at (9,-6) {};
\node[shape=circle,fill=black,inner sep=1.5,label={[xshift=-8,yshift=-14]:{\scriptsize $220$}}] (G4) at (10,-6) {};
\node[shape=circle,fill=black,inner sep=1.5,label={[xshift=-8,yshift=-14]:{\scriptsize $320$}}] (G5) at (11,-6) {};
\node[shape=circle,fill=black,inner sep=1.5,label={[xshift=-8,yshift=-14]:{\scriptsize $321$}}] (G6) at (12,-6) {};
\node[shape=circle,fill=black,inner sep=1.5,label={[xshift=-8,yshift=-14]:{\scriptsize $322$}}] (G7) at (13,-6) {};
\node[shape=circle,fill=black,inner sep=1.5,label={[xshift=-8,yshift=-14]:{\scriptsize $332$}}] (G8) at (14,-6) {};
\node[shape=circle,fill=black,inner sep=1.5,label={[xshift=-8,yshift=-14]:{\scriptsize $333$}}] (G9) at (15,-6) {};
\node at (9.5,-6.5) {\bf \large $\B$};

\node[shape=circle,fill=black,inner sep=1.5,label={[xshift=-8,yshift=-14]:{\scriptsize $000$}}] (H0) at (7,-7) {};
\node[shape=circle,fill=black,inner sep=1.5,label={[xshift=-8,yshift=-14]:{\scriptsize $100$}}] (H1) at (8,-7) {};
\node[shape=circle,fill=black,inner sep=1.5,label={[xshift=-8,yshift=-14]:{\scriptsize $110$}}] (H2) at (9,-7) {};
\node[shape=circle,fill=black,inner sep=1.5,label={[xshift=-8,yshift=-14]:{\scriptsize $210$}}] (H3) at (10,-7) {};
\node[shape=circle,fill=black,inner sep=1.5,label={[xshift=-8,yshift=-14]:{\scriptsize $310$}}] (H4) at (11,-7) {};
\node[shape=circle,fill=black,inner sep=1.5,label={[xshift=-8,yshift=-14]:{\scriptsize $311$}}] (H5) at (12,-7) {};
\node[shape=circle,fill=black,inner sep=1.5,label={[xshift=-8,yshift=-14]:{\scriptsize $312$}}] (H6) at (13,-7) {};
\node[shape=circle,fill=black,inner sep=1.5,label={[xshift=-8,yshift=-14]:{\scriptsize $322$}}] (H7) at (14,-7) {};
\node[shape=circle,fill=black,inner sep=1.5,label={[xshift=-8,yshift=-14]:{\scriptsize $323$}}] (H8) at (15,-7) {};
\node[shape=circle,fill=black,inner sep=1.5,label={[xshift=-8,yshift=-14]:{\scriptsize $333$}}] (H9) at (16,-7) {};
\node at (8.5,-7.5) {\bf \large $\B$};

\node[shape=circle,fill=black,inner sep=1.5,label={[xshift=-8,yshift=-14]:{\scriptsize $000$}}] (I0) at (8,-8) {};
\node[shape=circle,fill=black,inner sep=1.5,label={[xshift=-8,yshift=-14]:{\scriptsize $100$}}] (I1) at (9,-8) {};
\node[shape=circle,fill=black,inner sep=1.5,label={[xshift=-8,yshift=-14]:{\scriptsize $200$}}] (I2) at (10,-8) {};
\node[shape=circle,fill=black,inner sep=1.5,label={[xshift=-8,yshift=-14]:{\scriptsize $300$}}] (I3) at (11,-8) {};
\node[shape=circle,fill=black,inner sep=1.5,label={[xshift=-8,yshift=-14]:{\scriptsize $301$}}] (I4) at (12,-8) {};
\node[shape=circle,fill=black,inner sep=1.5,label={[xshift=-8,yshift=-14]:{\scriptsize $302$}}] (I5) at (13,-8) {};
\node[shape=circle,fill=black,inner sep=1.5,label={[xshift=-8,yshift=-14]:{\scriptsize $312$}}] (I6) at (14,-8) {};
\node[shape=circle,fill=black,inner sep=1.5,label={[xshift=-8,yshift=-14]:{\scriptsize $313$}}] (I7) at (15,-8) {};
\node[shape=circle,fill=black,inner sep=1.5,label={[xshift=-8,yshift=-14]:{\scriptsize $323$}}] (I8) at (16,-8) {};
\node[shape=circle,fill=black,inner sep=1.5,label={[xshift=-8,yshift=-14]:{\scriptsize $333$}}] (I9) at (17,-8) {};
\node at (14.5,-8.5) {\bf \large $\C$};

\node[shape=circle,fill=black,inner sep=1.5,label={[xshift=-8,yshift=-14]:{\scriptsize $000$}}] (J0) at (9,-9) {};
\node[shape=circle,fill=black,inner sep=1.5,label={[xshift=-8,yshift=-14]:{\scriptsize $100$}}] (J1) at (10,-9) {};
\node[shape=circle,fill=black,inner sep=1.5,label={[xshift=-8,yshift=-14]:{\scriptsize $200$}}] (J2) at (11,-9) {};
\node[shape=circle,fill=black,inner sep=1.5,label={[xshift=-8,yshift=-14]:{\scriptsize $201$}}] (J3) at (12,-9) {};
\node[shape=circle,fill=black,inner sep=1.5,label={[xshift=-8,yshift=-14]:{\scriptsize $202$}}] (J4) at (13,-9) {};
\node[shape=circle,fill=black,inner sep=1.5,label={[xshift=-8,yshift=-14]:{\scriptsize $212$}}] (J5) at (14,-9) {};
\node[shape=circle,fill=black,inner sep=1.5,label={[xshift=-8,yshift=-14]:{\scriptsize $312$}}] (J6) at (15,-9) {};
\node[shape=circle,fill=black,inner sep=1.5,label={[xshift=-8,yshift=-14]:{\scriptsize $322$}}] (J7) at (16,-9) {};
\node[shape=circle,fill=black,inner sep=1.5,label={[xshift=-8,yshift=-14]:{\scriptsize $332$}}] (J8) at (17,-9) {};
\node[shape=circle,fill=black,inner sep=1.5,label={[xshift=-8,yshift=-14]:{\scriptsize $333$}}] (J9) at (18,-9) {};
\node at (5,0.75) {$\ddots$};
\node at (13,-9.75) {$\ddots$};

\draw[thick] (A0)--(A1)--(A2)--(A3)--(A4)--(A5)--(A6)--(A7)--(A8)--(A9);
\draw[thick] (B0)--(B1)--(B2)--(B3)--(B4)--(B5)--(B6)--(B7)--(B8)--(B9);
\draw[thick] (C0)--(C1)--(C2)--(C3)--(C4)--(C5)--(C6)--(C7)--(C8)--(C9);
\draw[thick] (D0)--(D1)--(D2)--(D3)--(D4)--(D5)--(D6)--(D7)--(D8)--(D9);
\draw[thick] (E0)--(E1)--(E2)--(E3)--(E4)--(E5)--(E6)--(E7)--(E8)--(E9);
\draw[thick] (F0)--(F1)--(F2)--(F3)--(F4)--(F5)--(F6)--(F7)--(F8)--(F9);
\draw[thick] (G0)--(G1)--(G2)--(G3)--(G4)--(G5)--(G6)--(G7)--(G8)--(G9);
\draw[thick] (H0)--(H1)--(H2)--(H3)--(H4)--(H5)--(H6)--(H7)--(H8)--(H9);
\draw[thick] (I0)--(I1)--(I2)--(I3)--(I4)--(I5)--(I6)--(I7)--(I8)--(I9);
\draw[thick] (J0)--(J1)--(J2)--(J3)--(J4)--(J5)--(J6)--(J7)--(J8)--(J9);

\draw[thick] (A1)--(B0);
\draw[thick] (A2)--(B1)--(C0);
\draw[thick] (A3)--(B2)--(C1)--(D0);
\draw[thick] (A4)--(B3)--(C2)--(D1)--(E0);
\draw[thick] (A5)--(B4)--(C3)--(D2)--(E1)--(F0);
\draw[thick] (A6)--(B5)--(C4)--(D3)--(E2)--(F1)--(G0);
\draw[thick] (A7)--(B6)--(C5)--(D4)--(E3)--(F2)--(G1)--(H0);
\draw[thick] (A8)--(B7)--(C6)--(D5)--(E4)--(F3)--(G2)--(H1)--(I0);
\draw[thick] (A9)--(B8)--(C7)--(D6)--(E5)--(F4)--(G3)--(H2)--(I1)--(J0);
\draw[thick] (B9)--(C8)--(D7)--(E6)--(F5)--(G4)--(H3)--(I2)--(J1);
\draw[thick] (C9)--(D8)--(E7)--(F6)--(G5)--(H4)--(I3)--(J2);
\draw[thick] (D9)--(E8)--(F7)--(G6)--(H5)--(I4)--(J3);
\draw[thick] (E9)--(F8)--(G7)--(H6)--(I5)--(J4);
\draw[thick] (F9)--(G8)--(H7)--(I6)--(J5);
\draw[thick] (G9)--(H8)--(I7)--(J6);
\draw[thick] (H9)--(I8)--(J7);
\draw[thick] (I9)--(J8);
\end{tikzpicture}}
\end{center}
\caption{A decorated growth diagram for a Kreweras word.} \label{fig:kword_growth_diagram_ex}
\end{figure}

For $i \in \mathbb{Z}$, let us refer to the set of squares in positions of the form~$(x,-i)$ as the $i$th \dfn{row} of a diagram. Similarly, for $j\in \mathbb{Z}$, let us refer to the set of squares in positions of the form $(-3n-1+j,y)$ as the $j$th \dfn{column} of a diagram. \Cref{ex:kword_growth_diagram} may have suggested that in every row of the decorated growth diagram of a Kreweras word there is a unique filled square. This is true:

\begin{prop} \label{prop:kword_filled_rows}
Let $w$ be a Kreweras word of length $3n$ and consider its decorated growth diagram. Recall the definition of $\iota(w)$ from \cref{sec:intro}. Then for any $i \in \mathbb{Z}$, the square in the $i$th row and $(\iota(w')+i-1)$th column is filled with the letter $w'_{\iota(w')}$, where $w'  \coloneqq  \pro^{i-1}(w)$. This is the unique filled square in the $i$th row.
\end{prop}
\begin{proof}
By the translation symmetry of growth diagrams, \cref{cor:growth_diagram_symmetry}~\eqref{item:trans_symmetry}, it is enough to prove this for $i=1$. As mentioned above, the local rule defining growth diagrams of linear extensions reflects the behavior of the involutions~$\tau_i$. In particular, a square in the $1$st row and $j$th column corresponds to an application of $\tau_{j-1}$ when carrying out the product $\tau_{\ell} \circ \cdots \circ \tau_2 \circ \tau_1$ to perform promotion. We can view these $\tau_i$ as acting directly on the Kreweras word~$w$ as in the proof of \cref{prop:kword_poset_pro}, and we will be in an ``otherwise'' for a square in \cref{fig:growth_diagram_rule_vn} exactly when we are in an ``otherwise'' case for corresponding $\tau_i$. As the proof of that proposition explains, the unique $\tau_i$ for which an ``otherwise'' case occurs is $i=\iota(w)-1$.
\end{proof}

Thanks to the $x$/$y$ symmetry of growth diagrams in \cref{cor:growth_diagram_symmetry}~\eqref{item:xy_symmetry}, \cref{prop:kword_filled_rows} also implies that every column of a decorated growth diagram contains a unique filled square.

Now we will demonstrate how $\sigma_w$ and $\varepsilon_w$ from \cref{sec:proof} can easily be read off from decorated growth diagrams.

\begin{lemma} \label{lem:perm_from_growth_diagram}
Let $w$ be a Kreweras word of length $3n$ and consider its decorated growth diagram. Let $1 \leq i \leq 3n$, and suppose the unique filled square in the $i$th row is in the $j$th column and is filled with $\varepsilon \in \{\B,\C\}$. Then we have~$\varepsilon_w(i) = \varepsilon$ and $\sigma_w(i) = \langle j \rangle_{3n}$, where $\langle k \rangle_{3n}$ denotes the unique element of $\{1,\ldots,3n\}$ congruent to~$k$ modulo $3n$.
\end{lemma}
\begin{proof}

By \cref{prop:kword_filled_rows}, we need to show that $\sigma_w(i) = \langle \iota(\pro^{i-1}(w))+i-1 \rangle_{3n}$ and $\varepsilon_w(i) = \pro^{i-1}(w)_{\iota(\pro^{i-1}(w))}$ for all $1 \leq i \leq 3n$.

First let us explain why this holds for $i=1$. Note that $(1,\iota(w))$ is the ``near'' arc emanating from $1$ in the Kreweras bump diagram $\D_w$. Moreover, the rules of the road are such that in the trip starting at $1$ we will never make any turns on our way from $1$ towards~$\iota(w)$. So indeed $\sigma_w(1) = \iota(w)$, and $\varepsilon_w(1) = w_{\iota(w)}$.

Then we have thanks to \cref{lem:key} that $\sigma_w(i) = \langle \sigma_{\pro^{i-1}(w)}(1) +i-1\rangle_{3n}$ and $\varepsilon_w(i) = \varepsilon_{\pro^{i-1}(w)}(1)$ for any $1\leq i \leq n$. Thus, by the result in the previous paragraph applied to $\pro^{i-1}(w)$, we are done.
\end{proof}

\begin{example} \label{ex:perm_growth_diagram}
As in \cref{ex:kword_growth_diagram}, let $w = \A \A \B \B \C \A \C \C \B$. We saw in \cref{ex:key_lem} that $\sigma_w = [4,3,8,5,2,7,1,9,6]$ and $\varepsilon_w=[\B, \B, \C, \C, \B, \C, \B, \B, \C]$. In agreement with \cref{lem:perm_from_growth_diagram}, we can also easily read off this $\sigma_w$ and $\varepsilon_w$ from $w$'s decorated growth diagram, which is depicted in \cref{fig:kword_growth_diagram_ex}.
\end{example}

We could have defined $\sigma_w$ by setting $\sigma_w(i)  \coloneqq  \langle \iota(\pro^{i-1}(w))+i-1 \rangle_{3n}$ in light of \cref{lem:perm_from_growth_diagram}. However, if we did so, it would not be at all clear that $\sigma_w$ is a permutation. This is why we defined $\sigma_w$ in terms of trips in Kreweras bump diagrams.

Let us record just a few more basic properties of decorated growth diagrams in the following proposition.

\begin{prop} \label{prop:kword_growth_diagram_basics}
Let $w$ be a Kreweras word of length $3n$ and consider its decorated growth diagram. Then,
\begin{enumerate}[(a)]
\item \label{item:periodicity} for any $(x,y) \in \mathbb{Z}^2$, if the square in position~$(x,y)$ is filled with~$\varepsilon$, then the square in position $(x+3n,y-3n)$ is filled with~$-\varepsilon$;
\item \label{item:signs_flip_rows} for any $i \in \mathbb{Z}$, if the unique filled square in the $i$th row is filled with~$\varepsilon$, then the unique filled square in the $i$th column is filled with $-\varepsilon$.
\end{enumerate}
\end{prop}
\begin{proof}
For \eqref{item:periodicity}: this is an immediate consequence of the translation symmetry of growth diagrams in \cref{cor:growth_diagram_symmetry}~\eqref{item:trans_symmetry}, and our main result, \cref{thm:main}.

For \eqref{item:signs_flip_rows}: by the translation symmetry of \cref{cor:growth_diagram_symmetry}~\eqref{item:trans_symmetry} it is enough to prove this statement for a single row/column pair. Let us prove it for the $\sigma_w(1)$th row. For the $\sigma_w(1)$th row, this statement is a consequence of the interpretation of $\varepsilon_w$ in \cref{lem:perm_from_growth_diagram} and \cref{prop:basic_perm_props}~\eqref{item:signs_flip}.
\end{proof}

We now state our main results about evacuation of Kreweras words. For these we need the notion of reverse-complementation of permutations.

\begin{definition}
For a permutation $\sigma \in \Sym_m$, the \dfn{reverse-complement} of $\sigma$, denoted $\rc(\sigma)$, is the conjugation of $\sigma$ by the \dfn{longest element} of the symmetric group $w_0  \coloneqq  [m,m-1,\ldots,1] \in \Sym_{m}$; i.e.,
\[\rc(\sigma)  \coloneqq  w_0^{-1} \circ \sigma \circ w_0.\]
Note that reverse-complementation commutes with inversion because $w_0$ is an involution.
\end{definition}

\begin{lemma} \label{lem:perm_evac}
Let $w$ be a Kreweras word of length $3n$. Then,
\begin{enumerate}[(a)]
\item \label{item:perm_evac} $\sigma_{\evac(w)} = \rc(\sigma^{-1}_w)$;
\item \label{item:epsilon_evac} $\varepsilon_{\evac(w)} = [-\varepsilon_w(3n), -\varepsilon_w(3n-1), \ldots, -\varepsilon_w(1)]$.
\end{enumerate}
\end{lemma}

\begin{thm} \label{thm:kword_evac}
Let $w$ be a Kreweras word of length $3n$. Then
\[\evac(w) = (w_{\sigma_w(3n)}, w_{\sigma_w(3n-1)}, \ldots, w_{\sigma_w(1)}).\]
\end{thm}

One nice property of the operations in \cref{lem:perm_evac} is that they are evidently involutive. It is also easy to see that they have the ``right'' interaction (in the sense of  \cref{prop:pro_evac_basics}) with the operations in \cref{lem:key}.

On the other hand, from the definition of $\sigma_w$ in terms of trips in Kreweras bump diagrams it is far from clear why the word $(w_{\sigma_w(3n)}, w_{\sigma_w(3n-1)}, \ldots, w_{\sigma_w(1)})$ appearing in \cref{thm:kword_evac} is a Kreweras word.

Before we prove these results, let us do an example.

\begin{example} \label{ex:evac}
As in \cref{ex:perm_growth_diagram}, let $w = \A \A \B \B \C \A \C \C \B$. We saw above that $\sigma_w = [4,3,8,5,2,7,1,9,6]$ and $\varepsilon_w=[\B, \B, \C, \C, \B, \C, \B, \B, \C]$.

Thanks to \cref{prop:growth_diagram_basics}, we can read off $\evac w$ from $w$'s growth diagram, which is depicted in \cref{fig:kword_growth_diagram_ex}: we have $\evac w = \A \B \A \C \A \C \C \B \B$. This agrees with \cref{thm:kword_evac}.

The Kreweras bump diagram of $\evac(w)$ is
\begin{center}
\begin{tikzpicture}[scale=0.85]
\def \x {0.7};
\node at (-1*\x,0) {$\D_{\evac(w)}=$};
\node (1) at (1*\x,0) {1};
\node (2) at (2*\x,0) {2};
\node (3) at (3*\x,0) {3};
\node (4) at (4*\x,0) {4};
\node (5) at (5*\x,0) {5};
\node (6) at (6*\x,0) {6};
\node (7) at (7*\x,0) {7};
\node (8) at (8*\x,0) {8};
\node (9) at (9*\x,0) {9};
\node at (1*\x,-0.55*\x) {$\A$};
\node at (2*\x,-0.55*\x) {$\B$};
\node at (3*\x,-0.55*\x) {$\A$};
\node at (4*\x,-0.55*\x) {$\C$};
\node at (5*\x,-0.55*\x) {$\A$};
\node at (6*\x,-0.55*\x) {$\C$};
\node at (7*\x,-0.55*\x) {$\C$};
\node at (8*\x,-0.55*\x) {$\B$};
\node at (9*\x,-0.55*\x) {$\B$};
\draw [ultra thick,blue] (1) to [out=90, in=90] (2);
\draw [ultra thick,blue] (3) to [out=90, in=90] (9);
\draw [ultra thick,blue] (5) to [out=90, in=90] (8);
\draw [ultra thick,red,dashed] (1) to [out=90, in=90] (7);
\draw [ultra thick,red,dashed] (3) to [out=90, in=90] (4);
\draw [ultra thick,red,dashed] (5) to [out=90, in=90] (6);
\end{tikzpicture}
\end{center}
From the diagram $\D_{\evac(w)}$ one could compute that $\sigma_{\evac(w)}=[2, 7, 4, 1, 6, 9, 8, 5, 3]$ and $\varepsilon_{\evac(w)}=[\B, \C, \C, \B, \C, \B, \B, \C, \C]$, in agreement with \cref{lem:perm_evac}.

By comparing this example with \cref{ex:key_lem}, we see that $\pro(w)=\evac(w)$  in this case, but that's a coincidence for this particular Kreweras word $w$ which does not always happen.
\end{example}

We proceed to prove \cref{lem:perm_evac} and \cref{thm:kword_evac}.

\begin{proof}[Proof of \cref{lem:perm_evac}]
Thanks to the $x$/$y$ symmetry of growth diagrams in \cref{cor:growth_diagram_symmetry}~\eqref{item:xy_symmetry}, \cref{lem:perm_from_growth_diagram} says that for $1 \leq j \leq 3n$, if the unique filled square in the $3n+1-j$th column of the decorated growth diagram of $w$ is in the $3n+1-i$th row, then $\sigma_{\evac(w)}(j)  \coloneqq  \langle i \rangle_{3n}$; and if this square is filled with $\varepsilon \in \{\B, \C\}$ then $\varepsilon_{\evac(w)}(j)  \coloneqq  \varepsilon$.

Then, the periodicity property of decorated growth diagrams in \cref{prop:kword_growth_diagram_basics}~\eqref{item:periodicity} (along with the interpretation of $\sigma_w$ in \cref{lem:perm_from_growth_diagram}) gives $\sigma_{\evac(w)} = \rc(\sigma^{-1}_w)$.

Meanwhile, \cref{prop:kword_growth_diagram_basics}~\eqref{item:signs_flip_rows} (along with the interpretation of $\varepsilon_w$ in \cref{lem:perm_from_growth_diagram}) gives $\varepsilon_{\evac(w)} = [-\varepsilon_w(3n), -\varepsilon_w(3n-1), \ldots, -\varepsilon_w(1)]$.
\end{proof}

\begin{proof}[Proof of \cref{thm:kword_evac}]
This is easy enough to see from the decorated growth diagram of $w$ directly, but we can also deduce it from \cref{lem:perm_evac}.

For any permutation $\sigma \in \Sym^m$, a straightforward unraveling of the definitions shows that
\[\{i\colon i\in[m], (\rc(\sigma^{-1}))^{-1}(i) > i\} = \{m-\sigma(i)\colon i\in [m], \sigma^{-1}(i) > i\}.\]
Hence \cref{cor:perm_epsilon_determine_kword} and \cref{lem:perm_evac}~\eqref{item:perm_evac} imply that at least the positions of the $\A$'s are the same in $\evac(w)$ and $(w_{\sigma_w(3n)}, w_{\sigma_w(3n-1)}, \ldots, w_{\sigma_w(1)})$.

Now let $1\leq i \leq 3n$ be such that $\evac(w)_i \in \{\B,\C\}$. Then~\cref{cor:perm_epsilon_determine_kword} and \cref{lem:perm_evac} imply that $\evac(w)_i = -\varepsilon_w(\sigma_w(3n+1-i))$. Since $\evac(w)_i \neq\A$, the previous paragraph tells us that $w_{\sigma_w(3n+1-i)}\in \{\B,\C\}$, and hence \cref{prop:basic_perm_props}~\eqref{item:signs_flip} tells us that $-\varepsilon_w(\sigma_w(3n+1-i)) = w_{\sigma_w(3n+1-i)}$. Thus $\evac(w)_i=w_{\sigma_w(3n+1-i)}$ in this case as well.
\end{proof}

Of course, it is also reasonable to ask how dual evacuation acts on Kreweras words. But \cref{prop:pro_evac_basics} says that $\evac^{*}(w)=\evac(\pro^{3n}(w))$ for any Kreweras word $w$ of length $3n$, and thus our main result, \cref{thm:main}, says that $\evac^{*}(w)$ is obtained from $\evac(w)$ by swapping all $\B$'s for $\C$'s and vice-versa. Similarly, we can see that $\sigma_{\evac^{*}(w)} = \rc(\sigma^{-1}_w)$ and $\varepsilon_{\evac^{*}(w)} = [\varepsilon_w(3n), \varepsilon_w(3n-1), \ldots, \varepsilon_w(1)]$ thanks to \cref{lem:key,lem:perm_evac}.

\section{Webs} \label{sec:webs}

In this section we reinterpret our results from the previous sections in the language of \dfn{webs}. We recall the notion of an $\mathfrak{sl}_3$-web, which is due to Kuperberg~\cite{kuperberg1996spiders}:

\begin{definition}
An \dfn{$\mathfrak{sl}_3$-web} $\W$ is a planar graph, embedded in a disk, with \dfn{boundary vertices} labeled $1,2,\ldots,m$ arranged on the rim of the disk in counterclockwise order, and any number of (unlabeled) \dfn{internal vertices} such that
\begin{itemize}
\item $\W$ is \dfn{trivalent}: all the boundary vertices have degree one, while all the internal vertices have degree three;
\item $\W$ is \dfn{bipartite}: the vertices (both boundary and internal) are colored white and black, with edges only between oppositely colored vertices.
\end{itemize}
We call the face of $\W$ containing the boundary vertices the \dfn{outer face}, and all other faces \dfn{internal}. We say that $\W$ is \dfn{irreducible} (or \dfn{non-elliptic}) if it has no internal faces with fewer than $6$ sides.
\end{definition}

Among all the $\mathfrak{sl}_3$-webs, the irreducible ones play a distinguished role. For instance, there are only finitely many irreducible webs with a fixed number of boundary vertices.

We will now explain how to convert a Kreweras bump diagram of a Kreweras word into a web by ``breaking apart'' its crossings.

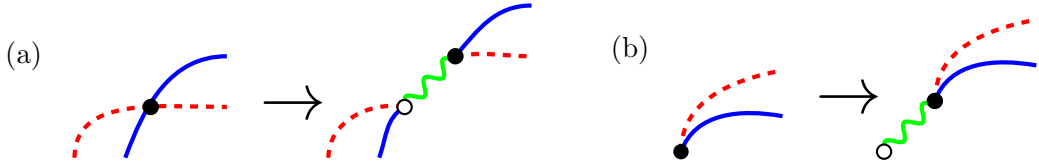
\begin{figure}
\begin{center}
\begin{tikzpicture}[scale=1.35]
\node at (-1,1){(a)};
\tikzset{->-/.style={decoration={ markings, mark=at position 0.25 with {\arrow{>}}},postaction={decorate}}};
\tikzset{snake it/.style={decorate, decoration=snake}};
\draw [ultra thick,blue] (0,0) to [out=70, in=180] (1,1);
\draw [ultra thick,red,dashed] (-0.5,0) to [out=90, in=180] (1,0.5);
\filldraw[thick, black] (0.25,0.5) circle (2pt);
\node at (1.65,0.5){\Huge $\to$};
\draw [ultra thick,blue] (2.5,0) to [out=70, in=180+30] (2.75,0.5);
\draw [ultra thick,red,dashed] (2,0) to [out=90, in=180-10] (2.75,0.5);
\draw [ultra thick,green, snake it] (2.75,0.5) to (3.25,1.0);
\draw [ultra thick,blue] (3.25,1.0) to [out=50, in=180] (4,1.5);
\draw [ultra thick,red,dashed] (3.25,1.0) to [out=10, in=180] (4,1);
\draw[thick,draw=black,fill=white] (2.75,0.5) circle (2pt);
\draw[thick,draw=black,fill=black] (3.25,1.0) circle (2pt);
\end{tikzpicture}  \quad \quad \begin{tikzpicture}[scale=1.35]
\node at (-0.5,1){(b)};
\tikzset{->-/.style={decoration={ markings, mark=at position 0.25 with {\arrow{>}}},postaction={decorate}}};
\tikzset{snake it/.style={decorate, decoration=snake}};
\draw [ultra thick,blue] (0,0) to [out=70, in=170] (1,0.35);
\draw [ultra thick,red,dashed] (0,0) to [out=90, in=190] (1,0.8);
\filldraw[black] (0,0) circle (2pt);
\node at (1.65,0.5){\Huge $\to$};
\draw [ultra thick,blue] (2.5,0.5) to [out=70, in=170] (3.5,0.85);
\draw [ultra thick,red,dashed] (2.5,0.5) to [out=90, in=190] (3.5,1.3);
\draw [ultra thick,green, snake it] (2,0) to (2.5,0.5);
\draw[thick,draw=black,fill=white]  (2,0) circle (2pt);
\draw[thick,draw=black,fill=black]  (2.5,0.5) circle (2pt);
\end{tikzpicture}
\end{center}

\caption{Breaking apart the crossings in a Kreweras bump diagram to obtain a web. In (a) we show what happens at an internal crossing, and in (b) we show what happens at a boundary crossing.} \label{fig:bump_to_web}
\end{figure}

\begin{construction}\label{construction:breaking-apart}
Let $w$ be a Kreweras word and $\D_w$ its associated Kreweras bump diagram.
We obtain a planar graph $\W_w$, embedded into a disk, together with a $3$-coloring $c_w$ of its edges as follows.

We replace each crossing of two arcs in $\D_w$ with a pair of a vertices, one white and one black, joined by a wavy {\bf a}vocado (i.e.~green) edge, as in \cref{fig:bump_to_web}.
The white vertex in this pair is ``to the left'' of the black vertex, that is, closer to the openers of $\D_w$.
We color all vertices of degree one in the resulting graph, corresponding to the openers and closers of $\D_w$, white, and keep the labels of these vertices.
Finally, the color of the non-avocado edges of $\W_w$ is inherited from $\D_w$.
\end{construction}

\begin{example} \label{ex:web}
As in \cref{ex:evac}, let $w = \A \A \B \B \C \A \C \C \B$. Recall that the Kreweras bump diagram $\D_w$ of $w$ is:
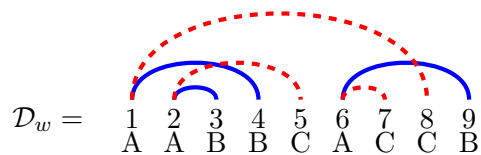
\begin{center}
\begin{tikzpicture}[scale=0.8]
\def \x {0.7};
\node at (-1*\x,0) {$\D_w=$};
\node (1) at (1*\x,0) {1};
\node (2) at (2*\x,0) {2};
\node (3) at (3*\x,0) {3};
\node (4) at (4*\x,0) {4};
\node (5) at (5*\x,0) {5};
\node (6) at (6*\x,0) {6};
\node (7) at (7*\x,0) {7};
\node (8) at (8*\x,0) {8};
\node (9) at (9*\x,0) {9};
\node at (1*\x,-0.55*\x) {$\A$};
\node at (2*\x,-0.55*\x) {$\A$};
\node at (3*\x,-0.55*\x) {$\B$};
\node at (4*\x,-0.55*\x) {$\B$};
\node at (5*\x,-0.55*\x) {$\C$};
\node at (6*\x,-0.55*\x) {$\A$};
\node at (7*\x,-0.55*\x) {$\C$};
\node at (8*\x,-0.55*\x) {$\C$};
\node at (9*\x,-0.55*\x) {$\B$};
\draw [ultra thick,blue] (1) to [out=90, in=90] (4);
\draw [ultra thick,blue] (2) to [out=90, in=90] (3);
\draw [ultra thick,blue] (6) to [out=90, in=90] (9);
\draw [ultra thick,red,dashed] (1) to [out=90, in=90] (8);
\draw [ultra thick,red,dashed] (2) to [out=90, in=90] (5);
\draw [ultra thick,red,dashed] (6) to [out=90, in=90] (7);
\end{tikzpicture}
\end{center}
Breaking apart the crossings of $\D_w$ gives the following $3$-edge-colored web:
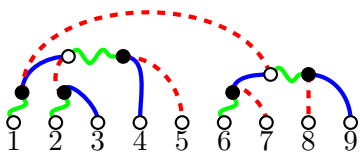
\begin{center}
\begin{tikzpicture}[scale=0.8]
\tikzset{snake it/.style={decorate, decoration=snake}};
\def \x {0.7};
\node (1) at (1*\x,0) {1};
\node (2) at (2*\x,0) {2};
\node (3) at (3*\x,0) {3};
\node (4) at (4*\x,0) {4};
\node (5) at (5*\x,0) {5};
\node (6) at (6*\x,0) {6};
\node (7) at (7*\x,0) {7};
\node (8) at (8*\x,0) {8};
\node (9) at (9*\x,0) {9};
\draw [ultra thick,green,snake it] (1*\x,0.3) to (1.2*\x,0.8);
\draw [ultra thick,green,snake it] (2*\x,0.3) to (2.2*\x,0.8);
\draw [ultra thick,green,snake it] (6*\x,0.3) to (6.2*\x,0.8);
\draw [ultra thick,green,snake it] (2.3*\x,1.4) to (3.6*\x,1.4);
\draw [ultra thick,green,snake it] (7.1*\x,1.1) to (8.0*\x,1.1);
\draw [ultra thick,blue] (1.2*\x,0.8) to [out=90, in=180] (2.3*\x,1.4);
\draw [ultra thick,blue] (3.6*\x,1.4) to [out=0, in=90] (4);
\draw [ultra thick,blue] (2.2*\x,0.8) to [out=90, in=90] (3);
\draw [ultra thick,blue] (6.2*\x,0.8) to [out=90, in=180] (7.1*\x,1.1);
\draw [ultra thick,blue] (8.0*\x,1.1) to [out=0, in=90] (9);
\draw [ultra thick,red,dashed] (1.2*\x,0.8) to [out=90, in=110] (7.1*\x,1.1);
\draw [ultra thick,red,dashed]  (8.0*\x,1.1) to [out=-90, in=90] (8);
\draw [ultra thick,red,dashed] (2.2*\x,0.8) to [out=180, in=180] (2.3*\x,1.4);
\draw [ultra thick,red,dashed] (3.6*\x,1.4) to [out=0, in=90] (5);
\draw [ultra thick,red,dashed] (6.2*\x,0.8) to [out=90, in=90] (7);
\draw[thick,draw=black,fill=white]  (1*\x,0.3) circle (3pt);
\draw[thick,draw=black,fill=white]  (2*\x,0.3) circle (3pt);
\draw[thick,draw=black,fill=white]  (3*\x,0.3) circle (3pt);
\draw[thick,draw=black,fill=white]  (4*\x,0.3) circle (3pt);
\draw[thick,draw=black,fill=white]  (5*\x,0.3) circle (3pt);
\draw[thick,draw=black,fill=white]  (6*\x,0.3) circle (3pt);
\draw[thick,draw=black,fill=white]  (7*\x,0.3) circle (3pt);
\draw[thick,draw=black,fill=white]  (8*\x,0.3) circle (3pt);
\draw[thick,draw=black,fill=white]  (9*\x,0.3) circle (3pt);
\draw[thick,draw=black,fill=black]  (1.2*\x,0.8) circle (3pt);
\draw[thick,draw=black,fill=black]  (2.2*\x,0.8) circle (3pt);
\draw[thick,draw=black,fill=black]  (6.2*\x,0.8) circle (3pt);
\draw[thick,draw=black,fill=white]  (2.3*\x,1.4) circle (3pt);
\draw[thick,draw=black,fill=black]  (3.6*\x,1.4) circle (3pt);
\draw[thick,draw=black,fill=white]  (7.1*\x,1.1) circle (3pt);
\draw[thick,draw=black,fill=black]  (8.0*\x,1.1) circle (3pt);
\end{tikzpicture}
\end{center}
Forgetting the $3$-edge-coloring, and drawing the graph embedded in a disk, we obtain the web $\W_w$:
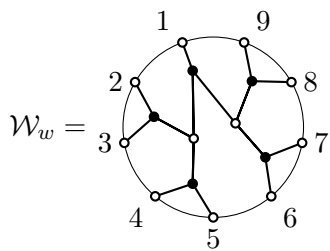
\begin{center}
\begin{tikzpicture}[scale=0.4]
\node at (-5.5,0) {$\W_w=$};
\draw (0,0) circle (3);
\node[label={above left:1},inner sep=0] (1) at (110:3) {};
\node[label={left:2},inner sep=0] (2) at (150:3) {};
\node[label={left:3},inner sep=0] (3) at (190:3) {};
\node[label={below left:4},inner sep=0] (4) at (230:3) {};
\node[label={below:5},inner sep=0] (5) at (270:3) {};
\node[label={below right:6},inner sep=0] (6) at (310:3) {};
\node[label={right:7},inner sep=0] (7) at (350:3) {};
\node[label={right:8},inner sep=0] (8) at (390:3) {};
\node[label={above right:9},inner sep=0] (9) at (430:3) {};
\node[inner sep=0] (A) at (170:2) {};
\node[inner sep=0] (B) at (250:2) {};
\node[inner sep=0] (C) at (330:2) {};
\node[inner sep=0] (D) at (410:2) {};
\node[inner sep=0] (E) at (110:2) {};
\node[inner sep=0] (F) at (210:0.75) {};
\node[inner sep=0] (G) at (370:0.75) {};
\draw[thick] (1)--(E)--(F)--(E)--(G);
\draw[thick] (2)--(A)--(F)--(A)--(3);
\draw[thick] (4)--(B)--(F)--(B)--(5);
\draw[thick] (6)--(C)--(G)--(C)--(7);
\draw[thick] (8)--(D)--(G)--(D)--(9);
\draw[thick,draw=black,fill=white]  (1) circle (4pt);
\draw[thick,draw=black,fill=white]  (2) circle (4pt);
\draw[thick,draw=black,fill=white]  (3) circle (4pt);
\draw[thick,draw=black,fill=white]  (4) circle (4pt);
\draw[thick,draw=black,fill=white]  (5) circle (4pt);
\draw[thick,draw=black,fill=white]  (6) circle (4pt);
\draw[thick,draw=black,fill=white]  (7) circle (4pt);
\draw[thick,draw=black,fill=white]  (8) circle (4pt);
\draw[thick,draw=black,fill=white]  (9) circle (4pt);
\draw[thick,draw=black,fill=black]  (A) circle (4pt);
\draw[thick,draw=black,fill=black]  (B) circle (4pt);
\draw[thick,draw=black,fill=black]  (C) circle (4pt);
\draw[thick,draw=black,fill=black]  (D) circle (4pt);
\draw[thick,draw=black,fill=black]  (E) circle (4pt);
\draw[thick,draw=black,fill=white]  (F) circle (4pt);
\draw[thick,draw=black,fill=white]  (G) circle (4pt);
\end{tikzpicture}
\end{center}
\end{example}

\begin{prop}\label{prop:kword_web_faces}
Let $w$ be a Kreweras word and let $(W_w, c_w)$ be the $3$-edge-colored graph obtained by \cref{construction:breaking-apart}.

Then $W_w$ is an irreducible $\mathfrak{sl}_3$-web with $3n$ boundary vertices, all of which are white.  Moreover, $W_w$ has no internal face having a multiple of four sides.

The $3$-coloring $c_w$ of the edges of $W_w$ is proper, i.e., each vertex is incident to at most one edge in each color class.

Finally, the construction is injective, that is, given $(\W_w,c_w)$ we can recover $w$: the boundary vertices incident to an avocado edge correspond to the~$\A$'s in~$w$, those incident to a blue edge correspond to~$\B$'s, and those incident to a crimson edge correspond to~$\C$'s.
\end{prop}

\begin{figure}
\begin{center}
\begin{tikzpicture}
\node at (0,0) {\begin{tikzpicture}[scale=0.6]
\draw[ultra thick,blue] (0,0) to[out=45,in=180-45] (2,0);
\draw[ultra thick,red,dashed] (2,0) to[out=45,in=180-45] (4,0);
\draw[ultra thick,blue] (4,0) to[out=45,in=180-45] (6,0);
\draw[ultra thick,red,dashed] (6,0) to[out=45,in=180-45] (8,0);
\draw[ultra thick,red,dashed] (0,0) to[out=45,in=180+20] (4,2);
\draw[ultra thick,blue] (4,2) to[out=-20,in=180-45] (8,0);
\end{tikzpicture}};
\node at (3,0) {\Huge $\to$};
\node at (6,0) {\begin{tikzpicture}[scale=0.6]
\tikzset{snake it/.style={decorate, decoration=snake}};
\draw[ultra thick,green,snake it] (-0.5,-0.5) to (0,0);
\draw[ultra thick,green,snake it] (2-0.5,0) to (2+0.5,0);
\draw[ultra thick,green,snake it] (4-0.5,0) to (4+0.5,0);
\draw[ultra thick,green,snake it] (6-0.5,0) to (6+0.5,0);
\draw[ultra thick,green,snake it] (8+0.5,-0.5) to (8,0);
\draw[ultra thick,green,snake it] (4-0.5,2) to (4+0.5,2);
\draw[ultra thick,blue] (0,0) to[out=45,in=180-45] (2-0.5,0);
\draw[ultra thick,red,dashed] (2+0.5,0) to[out=45,in=180-45] (4-0.5,0);
\draw[ultra thick,blue] (4+0.5,0) to[out=45,in=180-45] (6-0.5,0);
\draw[ultra thick,red,dashed] (6+0.5,0) to[out=45,in=180-45] (8,0);
\draw[ultra thick,red,dashed] (0,0) to[out=45,in=180+20] (4-0.5,2);
\draw[ultra thick,blue] (4+0.5,2) to[out=-20,in=180-45] (8,0);
\draw[thick,draw=black,fill=black] (0,0) circle (4pt);
\draw[thick,draw=black,fill=white] (2-0.5,0) circle (4pt);
\draw[thick,draw=black,fill=black] (2+0.5,0) circle (4pt);
\draw[thick,draw=black,fill=white] (4-0.5,0) circle (4pt);
\draw[thick,draw=black,fill=black] (4+0.5,0) circle (4pt);
\draw[thick,draw=black,fill=white] (6-0.5,0) circle (4pt);
\draw[thick,draw=black,fill=black] (6+0.5,0) circle (4pt);
\draw[thick,draw=black,fill=white] (8,0) circle (4pt);
\draw[thick,draw=black,fill=white] (4-0.5,2) circle (4pt);
\draw[thick,draw=black,fill=black] (4+0.5,2) circle (4pt);
\end{tikzpicture}};
\end{tikzpicture}
\end{center}
\caption{For the proof of \cref{prop:kword_web_faces}: how faces of $\D_w$ correspond to faces of $\W_w$.} \label{fig:internal_faces}
\end{figure}

\begin{proof}
The only non-trivial claim is that the number of sides of any face cannot be a multiple of $4$. As depicted in~\cref{fig:internal_faces}, internal faces of $\D_w$ with $k$ sides correspond to internal faces of $\W_w$ with $2k-2$ sides. Since $\D_w$ only has crossings between arcs of different colors, the number of sides of any internal face of $\D_w$ is even.
\end{proof}

The web $\W_w$ without its $3$-edge-coloring is not quite enough to recover $w$.  However, as we now explain, it gives information equivalent to the permutation~$\sigma_w$.
In fact, we can associate a permutation to any $\mathfrak{sl}_3$-web by taking trips in the web, similar to what we did in \cref{sec:proof} for Kreweras bump diagrams.

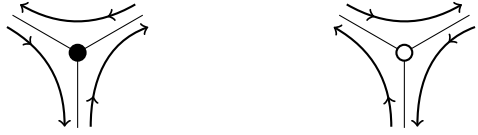
\begin{figure}
\begin{center}
\begin{tikzpicture}
\tikzset{->-/.style={decoration={ markings, mark=at position 0.25 with {\arrow{>}}},postaction={decorate}}};
\draw (0,0) to (30:1);
\draw (0,0) to (150:1);
\draw (0,0) to (270:1);
\draw[thick,draw=black,fill=black]  (0,0) circle (3pt);
\draw[thick,->,->-]  (280:1) to [out=90,in=200] (20:1);
\draw[thick,->,->-]  (40:1) to [out=210,in=150+180] (140:1);
\draw[thick,->,->-]  (160:1) to [out=150+180,in=90] (260:1);
\end{tikzpicture} \qquad \qquad \qquad \begin{tikzpicture}
\tikzset{->-/.style={decoration={ markings, mark=at position 0.25 with {\arrow{>}}},postaction={decorate}}};
\draw (0,0) to (30:1);
\draw (0,0) to (150:1);
\draw (0,0) to (270:1);
\draw[thick,draw=black,fill=white]  (0,0) circle (3pt);
\draw[thick,->,->-]  (20:1) to [out=200,in=90] (280:1) ;
\draw[thick,->,->-]  (140:1) to [out=150+180,in=210] (40:1);
\draw[thick,->,->-]  (260:1) to [out=90,in=150+180] (160:1);
\end{tikzpicture}
\end{center}
\caption{The rules of the road when taking a trip in a web.} \label{fig:rules_road_webs}
\end{figure}

\begin{definition} \label{def:web_trip_perm}
Let $\W$ be an $\mathfrak{sl}_3$-web with $m$ boundary vertices. The \dfn{trip permutation} of $\W$, denoted $\mathrm{trip}_{\W} \in \mathfrak{S}_m$, is obtained as follows. For $1\leq i \leq m$ we take a \dfn{trip in $\W$ starting at $i$}. To do this, we start by walking from boundary vertex $i$ along the unique edge incident to it. When we come to any internal vertex in $\W$, we continue our trip by following the \dfn{rules of the road}:
\begin{itemize}
\item if the vertex is black, we \dfn{turn right}, i.e., we walk out along the next edge counterclockwise from where we came in;
\item if the vertex is white, we \dfn{turn left}, i.e., we walk out along the next edge clockwise from where we came in.
\end{itemize}
These rules of the road are depicted in \cref{fig:rules_road_webs}. We stop our trip when we reach a boundary vertex. If $j$ is the boundary vertex we reach from the trip starting at $i$, then we set $\mathrm{trip}_{\W}(i)  \coloneqq  j$.
\end{definition}

That $\mathrm{trip}_{\W}$ is genuinely a permutation again follows from the fact that the rules of the road around any vertex locally permute the entry and exit points.

Our reason for considering trip permutations is the follow proposition:

\begin{prop} \label{prop:trip_defs_agree}
Let $w$ be a Kreweras word. Then $\sigma_w= \mathrm{trip}_{\W}$.
\end{prop}
\begin{proof}
This is simply a matter of checking that locally at a crossing of arcs, the rules of the road for trips in the Kreweras bump diagram $\D_w$ agree with the rules of the road for the trips in the web $\W_w$. And to do that, we just need to look at \cref{fig:rules_road,fig:bump_to_web,fig:rules_road_webs}.
\end{proof}

The notion of trip permutations is due to Postnikov~\cite{postnikov2006total}, and comes from his theory of plabic graphs.
A \dfn{plabic (``planar bicolored'') graph} is a planar graph, embedded in a disk, whose internal vertices are colored black or white, and whose boundary vertices have degree one.
There are some differences between plabic graphs and $\mathfrak{sl}_3$-webs:
\begin{itemize}
\item the boundary vertices of a plabic graph are not colored;
\item the internal vertices of a plabic graph need not be trivalent;
\item the coloring of internal vertices of a plabic graph does not have to be proper, i.e., vertices of the same color may be adjacent.
\end{itemize}
Except for the small technicality about boundary vertices being colored, an $\mathfrak{sl}_3$-web is a special case of a plabic graph. Postnikov~\cite[\S13]{postnikov2006total} defined trip permutations for plabic graphs in exactly the same way as we have done for webs in \cref{def:web_trip_perm} above: turn right at black vertices and left at white vertices.\footnote{Technically Postnikov considered \dfn{decorated} permutations, which have their fixed points colored either black or white. None of the trip permutations we obtain from irreducible webs will have fixed points (see the proof of \cref{lem:trip_perm_equal}), so this issue of fixed point decoration will not concern us.}

If $\W$ and $\W'$ are two $\mathfrak{sl}_3$-webs with $m$ boundary vertices, and they differ only in the way their boundary vertices are colored, then $\mathrm{trip}_{\W}=\mathrm{trip}_{\W'}$, since the color of boundary vertices does not enter into the definition of trip permutations in any way. However, note that the color of any boundary vertex which is adjacent to an internal vertex has its color determined by the bipartiteness condition. Hence, if $\W$ and $\W'$ differ only in the way their boundary vertices are colored, then $\W'$ is obtained from $\W$ by swapping the colors of pairs of oppositely colored, adjacent boundary vertices. In particular, if $\W$ has all its boundary vertices the same color, then there is no web that differs from $\W'$ only in the way its boundary vertices are colored.

We now explain how Postnikov's work implies that for \emph{irreducible} webs, the situation discussed in the previous paragraph is the only way that trip permutations can coincide.

\begin{lemma} \label{lem:trip_perm_equal}
Let $\W$ and $\W'$ be irreducible $\mathfrak{sl}_3$-webs with $m$ boundary vertices. Suppose that $\mathrm{trip}_{\W} = \mathrm{trip}_{\W'}$. Then $\W$ and $\W'$ differ at most in the way their boundary vertices are colored. In particular, if all the boundary vertices of $\W$ are the same color, then $\W=\W'$.
\end{lemma}
\begin{proof}
Postnikov~\cite[\S12]{postnikov2006total} defined certain transformations of plabic graphs he called \dfn{moves} and \dfn{reductions}. If $\W$ is an irreducible $\mathfrak{sl}_3$-web (viewed as a plabic graph), then the only moves or reductions we can apply to it are ``trivial'' moves which add $2$-valent vertices by subdividing an edge, or remove such $2$-valent vertices by un-subdividing edges. (Crucially, the fact that all internal faces have at least $6$ sides means we will never be able to carry out a \dfn{square move}, which is the fundamental, nontrivial move in the theory.) In particular, we will never be able to apply a reduction to $\W$, so $\W$ is \dfn{reduced}. Then \cite[Theorem 13.2(4)]{postnikov2006total} says that $\mathrm{trip}_{\W}$ has no fixed points, so we don't have to worry about the issue of decorated fixed points. Finally, a key result~\cite[Theorem 13.4]{postnikov2006total} from Postnikov's paper says that two reduced plabic graphs have the same trip permutation if and only if they are related via a series of moves. Since, as mentioned, the only moves we can apply either add or remove $2$-valent vertices, we will not be able to reach any other web than $\W$ via these moves. Hence, Postnikov's result tells us that any other web with the same trip permutation as $\W$ is equal to $\W$ -- except in the way the boundary vertices are colored, which the plabic graph story does not see.
\end{proof}

\Cref{lem:trip_perm_equal} lets us apply our knowledge about how $\pro$ and $\evac$ affect $\sigma_w$ to understand how they affect $\W_w$ (\cref{thm:intro_web} from \cref{sec:intro}). We just need to define the corresponding web operations.

\begin{definition}
Let $\W$ be an $\mathfrak{sl}_3$-web with $m$ boundary vertices. The \dfn{rotation} of~$\W$, denote $\rot(\W)$, is obtained from $\W$ be relabeling its vertices according to the inverse long cycle $(m,m-1,\ldots,2,1) \in \Sym_m$. The \dfn{flip} of $\W$, denoted $\flip(\W)$, is obtained from $\W$ by drawing a chord in the disk separating $1$ and $m$, reflecting $\W$ across this chord, and then relabeling its vertices according to the longest element $[m,m-1,\ldots,1] \in \Sym_m$.
\end{definition}

\begin{thm} \label{thm:web}
Let $w$ be a Kreweras word. Then,
\begin{enumerate}[(a)]
\item $\W_{\pro(w)} = \rot(\W_{w})$;
\item $\W_{\evac(w)} = \flip(\W_{w})$.
\end{enumerate}
\end{thm}
\begin{proof}
We have $\mathrm{trip}_{\W_w} = \sigma_w$ because of \cref{prop:trip_defs_agree}. Thus \cref{lem:key,lem:perm_evac} imply $\mathrm{trip}_{\W_{\pro(w)}}=\rot(\mathrm{trip}_{\W_w})$ and $\mathrm{trip}_{\W_{\evac(w)}}=\rc(\mathrm{trip}_{\W_w}^{-1})$. For any $\mathfrak{sl}_3$-web~$\W$, it is straightforward to verify that $\mathrm{trip}_{\rot(\W)}=\rot(\mathrm{trip}_{\W})$ and $\mathrm{trip}_{\flip(\W)}=\rc(\mathrm{trip}_{\W}^{-1})$. But then thanks to \cref{lem:trip_perm_equal}, we know that $\rot(\W_w)$ and $\flip(\W_w)$ are the \emph{only} irreducible $\mathfrak{sl}_3$-webs with trip permutations equal to  $\mathrm{trip}_{\W_{\pro(w)}}$ and $\mathrm{trip}_{\W_{\evac(w)}}$. Therefore, we must have $\W_{\pro(w)}=\rot(\W_w)$ and $\W_{\evac(w)}=\flip(\W_w)$, as claimed.
\end{proof}

\begin{example}
As in \cref{ex:web}, let $w = \A \A \B \B \C \A \C \C \B$. We saw in \cref{ex:evac} that $\pro(w) = \evac(w) = w'$ where $w' = \A \B \A \C \A \C \C \B \B$. Recall that the Kreweras bump diagram $\D_{w'}$ of $w'$ is:
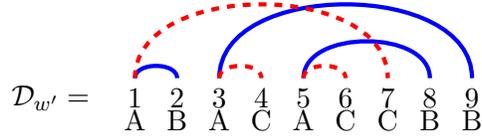
\begin{center}
\begin{tikzpicture}[scale=0.8]
\def \x {0.7};
\node at (-1*\x,0) {$\D_{w'}=$};
\node (1) at (1*\x,0) {1};
\node (2) at (2*\x,0) {2};
\node (3) at (3*\x,0) {3};
\node (4) at (4*\x,0) {4};
\node (5) at (5*\x,0) {5};
\node (6) at (6*\x,0) {6};
\node (7) at (7*\x,0) {7};
\node (8) at (8*\x,0) {8};
\node (9) at (9*\x,0) {9};
\node at (1*\x,-0.55*\x) {$\A$};
\node at (2*\x,-0.55*\x) {$\B$};
\node at (3*\x,-0.55*\x) {$\A$};
\node at (4*\x,-0.55*\x) {$\C$};
\node at (5*\x,-0.55*\x) {$\A$};
\node at (6*\x,-0.55*\x) {$\C$};
\node at (7*\x,-0.55*\x) {$\C$};
\node at (8*\x,-0.55*\x) {$\B$};
\node at (9*\x,-0.55*\x) {$\B$};
\draw [ultra thick,blue] (1) to [out=90, in=90] (2);
\draw [ultra thick,blue] (3) to [out=90, in=90] (9);
\draw [ultra thick,blue] (5) to [out=90, in=90] (8);
\draw [ultra thick,red,dashed] (1) to [out=90, in=90] (7);
\draw [ultra thick,red,dashed] (3) to [out=90, in=90] (4);
\draw [ultra thick,red,dashed] (5) to [out=90, in=90] (6);
\end{tikzpicture}
\end{center}
Breaking apart the crossings of $\D_{w'}$ gives the following $3$-edge-colored web:
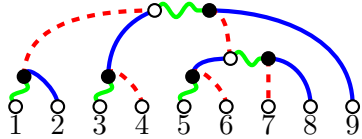
\begin{center}
\begin{tikzpicture}[scale=0.8]
\tikzset{snake it/.style={decorate, decoration=snake}};
\def \x {0.7};
\node (1) at (1*\x,0) {1};
\node (2) at (2*\x,0) {2};
\node (3) at (3*\x,0) {3};
\node (4) at (4*\x,0) {4};
\node (5) at (5*\x,0) {5};
\node (6) at (6*\x,0) {6};
\node (7) at (7*\x,0) {7};
\node (8) at (8*\x,0) {8};
\node (9) at (9*\x,0) {9};
\draw [ultra thick,green,snake it] (1*\x,0.3) to (1.2*\x,0.8);
\draw [ultra thick,green,snake it] (3*\x,0.3) to (3.2*\x,0.8);
\draw [ultra thick,green,snake it] (5*\x,0.3) to (5.2*\x,0.8);
\draw [ultra thick,green,snake it] (4.3*\x,1.9) to (5.6*\x,1.9);
\draw [ultra thick,green,snake it] (6.1*\x,1.1) to (7.0*\x,1.1);
\draw [ultra thick,blue] (1.2*\x,0.8) to [out=90, in=90] (2);
\draw [ultra thick,blue] (3.2*\x,0.8) to [out=90, in=200] (4.3*\x,1.9);
\draw [ultra thick,blue] (5.2*\x,0.8) to [out=90, in=180] (6.1*\x,1.1);
\draw [ultra thick,blue] (7.0*\x,1.1) to [out=0, in=90] (8);
\draw [ultra thick,blue] (5.6*\x,1.9) to [out=0, in=90] (9);
\draw [ultra thick,red,dashed] (1.2*\x,0.8) to [out=90, in=180] (4.3*\x,1.9);
\draw [ultra thick,red,dashed]  (7.0*\x,1.1) to [out=-90, in=90] (7);
\draw [ultra thick,red,dashed] (3.2*\x,0.8) to [out=90, in=90] (4);
\draw [ultra thick,red,dashed] (5.2*\x,0.8) to [out=90, in=90] (6);
\draw [ultra thick,red,dashed] (5.6*\x,1.9) to [out=-10, in=100] (6.1*\x,1.1);
\draw[thick,draw=black,fill=white]  (1*\x,0.3) circle (3pt);
\draw[thick,draw=black,fill=white]  (2*\x,0.3) circle (3pt);
\draw[thick,draw=black,fill=white]  (3*\x,0.3) circle (3pt);
\draw[thick,draw=black,fill=white]  (4*\x,0.3) circle (3pt);
\draw[thick,draw=black,fill=white]  (5*\x,0.3) circle (3pt);
\draw[thick,draw=black,fill=white]  (6*\x,0.3) circle (3pt);
\draw[thick,draw=black,fill=white]  (7*\x,0.3) circle (3pt);
\draw[thick,draw=black,fill=white]  (8*\x,0.3) circle (3pt);
\draw[thick,draw=black,fill=white]  (9*\x,0.3) circle (3pt);
\draw[thick,draw=black,fill=black]  (1.2*\x,0.8) circle (3pt);
\draw[thick,draw=black,fill=black]  (3.2*\x,0.8) circle (3pt);
\draw[thick,draw=black,fill=black]  (5.2*\x,0.8) circle (3pt);
\draw[thick,draw=black,fill=white]  (4.3*\x,1.9) circle (3pt);
\draw[thick,draw=black,fill=black]  (5.6*\x,1.9) circle (3pt);
\draw[thick,draw=black,fill=white]  (6.1*\x,1.1) circle (3pt);
\draw[thick,draw=black,fill=black]  (7.0*\x,1.1) circle (3pt);
\end{tikzpicture}
\end{center}
Forgetting the $3$-edge-coloring, and drawing the graph embedded in a disk, we see that the web $\W_w$ is:
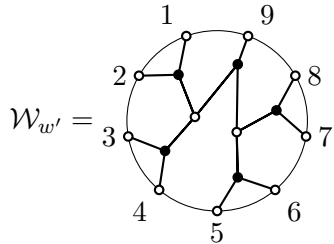
\begin{center}
\begin{tikzpicture}[scale=0.4]
\node at (-5.5,0) {$\W_{w'}=$};
\draw (0,0) circle (3);
\node[label={above right:9},inner sep=0] (1) at (70:3) {};
\node[label={above left:1},inner sep=0] (2) at (110:3) {};
\node[label={left:2},inner sep=0] (3) at (150:3) {};
\node[label={left:3},inner sep=0] (4) at (190:3) {};
\node[label={below left:4},inner sep=0] (5) at (230:3) {};
\node[label={below:5},inner sep=0] (6) at (270:3) {};
\node[label={below right:6},inner sep=0] (7) at (310:3) {};
\node[label={right:7},inner sep=0] (8) at (350:3) {};
\node[label={right:8},inner sep=0] (9) at (390:3) {};
\node[inner sep=0] (A) at (130:2) {};
\node[inner sep=0] (B) at (210:2) {};
\node[inner sep=0] (C) at (290:2) {};
\node[inner sep=0] (D) at (370:2) {};
\node[inner sep=0] (E) at (70:2) {};
\node[inner sep=0] (F) at (170:0.75) {};
\node[inner sep=0] (G) at (330:0.75) {};
\draw[thick] (1)--(E)--(F)--(E)--(G);
\draw[thick] (2)--(A)--(F)--(A)--(3);
\draw[thick] (4)--(B)--(F)--(B)--(5);
\draw[thick] (6)--(C)--(G)--(C)--(7);
\draw[thick] (8)--(D)--(G)--(D)--(9);
\draw[thick,draw=black,fill=white]  (1) circle (4pt);
\draw[thick,draw=black,fill=white]  (2) circle (4pt);
\draw[thick,draw=black,fill=white]  (3) circle (4pt);
\draw[thick,draw=black,fill=white]  (4) circle (4pt);
\draw[thick,draw=black,fill=white]  (5) circle (4pt);
\draw[thick,draw=black,fill=white]  (6) circle (4pt);
\draw[thick,draw=black,fill=white]  (7) circle (4pt);
\draw[thick,draw=black,fill=white]  (8) circle (4pt);
\draw[thick,draw=black,fill=white]  (9) circle (4pt);
\draw[thick,draw=black,fill=black]  (A) circle (4pt);
\draw[thick,draw=black,fill=black]  (B) circle (4pt);
\draw[thick,draw=black,fill=black]  (C) circle (4pt);
\draw[thick,draw=black,fill=black]  (D) circle (4pt);
\draw[thick,draw=black,fill=black]  (E) circle (4pt);
\draw[thick,draw=black,fill=white]  (F) circle (4pt);
\draw[thick,draw=black,fill=white]  (G) circle (4pt);
\end{tikzpicture}
\end{center}
Comparing with \cref{ex:web}, we can see that $\W_{w'} = \rot(\W_w) = \flip(\W_w)$, in agreement with \cref{thm:web}.
\end{example}

It is also possible to describe how promotion and evacuation affect the $3$-edge-coloring $c_w$ (briefly: we ``swap'' colors of edges along trips), but we will not go into details about that here.

However, a question we will answer in the following subsection is: which webs $\W$ are equal to $\W_w$ for some Kreweras word $w$? As we will see, the restriction coming from \cref{prop:kword_web_faces} is the only restriction.

\subsection{Kreweras webs}

\begin{definition}
A \dfn{Kreweras web} is an irreducible $\mathfrak{sl}_3$-web such that all boundary vertices are white and there are no internal faces with a multiple of $4$ sides.
\end{definition}

We note that a simple counting argument shows that any $\mathfrak{sl}_3$-web with all white boundary vertices has a multiple of $3$ boundary vertices.

We have already seen from \cref{prop:kword_web_faces} that any web $\W_w$ corresponding to a Kreweras word $w$ must be a Kreweras web. Our goal in this subsection is to show that all Kreweras webs arise this way.

\begin{thm}\label{thm:kreweras_webs}
  Let $\W$ be an $\mathfrak{sl}_3$-web. Then there is a Kreweras word~$w$ for which $\W=\W_w$ if and only $\W$ is a Kreweras web. Moreover, if $\W$ is a Kreweras web, then the number of Kreweras words $w$ for which $\W=\W_w$ is $2^{\kappa(\W)}$, where $\kappa(\W)$ is the number of connected components of $\W$.
\end{thm}

Let us first note an enumerative consequence (which is stated as part of~\cref{thm:weighted_kreweras_web_count} in~\cref{sec:intro}):

\begin{cor} \label{cor:web_count}
We have
\[ \sum_{\W} 2^{\kappa(W)} = \frac{4^n}{(n+1)(2n+1)}\binom{3n}{n}, \]
where the sum is over all Kreweras webs $\W$ with $3n$ boundary vertices.
\end{cor}
\begin{proof}
This follows from combining \cref{thm:kreweras_webs} with Kreweras's product formula enumerating Kreweras words~\cite{kreweras1965classe}.
\end{proof}

For more discussion of enumeration of webs (including an explanation of the rest of~\cref{thm:weighted_kreweras_web_count}), see \cref{subsec:maps}.

We prove \cref{thm:kreweras_webs} by demonstrating that we can appropriately edge-color any Kreweras web $\W$. This is achieved via the following construction:

\begin{construction}\label{construction:coloring}
Let $\W$ be Kreweras web with boundary vertices labeled counterclockwise $1$ to $3n$, and let $c_1,\dots,c_{\kappa(\W)}$ be a choice of color, either blue or crimson, for each connected component of $\W$. We create a proper $3$-edge-coloring of $\W$ (with colors avocado, blue, and crimson), and a system of $2n$ colored directed paths $1_L, 1_R, \dots, n_L, n_R$ in~$\W$, with the following properties:
\begin{itemize}
\item paths $i_L$ and $i_R$ begin at the same boundary vertex, and $i_L$ turns left when leaving the unique edge $e$ incident to this vertex, while $i_R$ turns right when leaving $e$;
\item the first and every other edge of a path is colored avocado, and all the other edges of the path have the same color (either blue or crimson) -- which we call the color of the path;
\item every avocado edge is traversed by precisely two paths, and every other edge is traversed by precisely one path;
\item any two paths share at most one (necessarily avocado) edge, and if they do, they are of different color;
\item if two paths $i_X$ and $j_Y$ with $i < j$ share an (avocado) edge $e$, then the path~$i_X$ turns to the left when visiting $e$, and continues the the right when leaving it.
\end{itemize}
The system of paths is created inductively. Once paths $1_L, 1_R,\dots, i{-}1_L, i{-}1_R$ are determined, the paths $i_L$ and $i_R$ start at the boundary vertex with smallest label incident to an uncolored edge. If $i_L$ is in a connected component of $\W$ different from the connected components containing $1_L,\dots,i{-}1_L$, the color of $i_L$ is $c_k$, where $k$ is the number of connected components containing $1_L,\dots,i_L$. The $3$-edge-coloring is then inherited from the colors of the paths.
\end{construction}
We say that \cref{construction:coloring} \dfn{succeeds} on $\W$ if the requested properties can be satisfied when creating the paths.

\begin{lemma} \label{lem:success_means_word}
\Cref{construction:coloring} succeeds on $\W$ if and only if $\W=\W_w$ for some Kreweras word $w$. And in this case, the $2^{\kappa(\W)}$ $3$-edge-colorings produced by applying \Cref{construction:coloring} to $\W$ with different choices of $c_1,\ldots,c_{\kappa(\W)}$ are exactly all the $c_w$ for such Kreweras words $w$.
\end{lemma}
\begin{proof}
Let $w$ be a Kreweras word. \Cref{prop:kword_web_faces} says that $\W_w$ must be a Kreweras web. Moreover, it is easy to see that the arcs of $\D_w$ determine a system of colored paths in $\W_w$ satisfying the properties required in \Cref{construction:coloring}.

Conversely, let $\W$ be a Kreweras web on which \cref{construction:coloring} succeeds. Then, by the properties of the construction, the paths of the same color form two noncrossing perfect matchings, with the same set of openers. Thus, they yield a Kreweras bump diagram of a Kreweras word.
\end{proof}

\begin{cor} \label{cor:coloring_rotation}
The set of Kreweras webs $\W$ on which \Cref{construction:coloring} succeeds is closed under $\rot$ and $\flip$.
\end{cor}

\begin{lemma}\label{lem:4k+2-path}
Suppose that the boundary vertices $1$ and $2$ are in the same connected component of a Kreweras web $\W$, and suppose that the shortest path from $1$ to $2$ (i.e., the one that turns right at every vertex) consists of $4k+2$ edges, for $k\geq 1$. Then, if it succeeds, the coloring produced by \cref{construction:coloring} colors the edges incident to $1$ and $2$ avocado. Moreover, the path $1_R$ and the path $2_L$ have the same color.
\end{lemma}
\begin{proof}
The edge incident to $2$ will be colored a non-avocado color if and only if the distance between vertices $1$ and $2$ is two: if not, the path $1_R$ turns left after the second edge and therefore does not visit vertex $2$.  Thus, after the paths $1_L$ and $1_R$ are created, the edge incident to vertex $2$ is uncolored, and is therefore chosen as the initial edge of paths~$2_L$ and $2_R$.

So now let us focus on the claim about the colors of $1_R$ and $2_L$. Let $x$ be the first white non-boundary vertex on the colored path $1_R$.  It suffices to show that every other edge of the shortest path from $2$ to $x$ is colored avocado, and the colors of the remaining edges alternate.

\begin{figure}
\begin{center}
\begin{tikzpicture}[scale=1.25]
\tikzset{snake it/.style={decorate, decoration=snake}};
\tikzset{->-/.style={decoration={ markings, mark=at position 0.5 with {\arrow{>}}},postaction={decorate}}};
\draw [ultra thick,green,snake it] (0,0) to (0,0.75);
\draw [ultra thick,red,dashed] (0,0.75) to (0.5,1.5);
\draw [ultra thick,blue] (0,0.75) to (-0.5,1.5);
\draw [ultra thick,black,dotted]  (0.5,1.5) to (1.5,1.5);
\draw [ultra thick,blue] (1.5,1.5) to (2.25,1.5);
\draw [ultra thick,red,dashed] (1.75,2.0) to (2.25,1.5);
\draw [ultra thick,green,snake it] (2.25,1.5) to (3.0,1.5);
\draw [ultra thick,black,dotted,->-] (3.0,2.5) to (3.0,1.5);
\draw [ultra thick,red,dashed] (3.0,1.5) to (3.75,1.5);
\draw [ultra thick,green,snake it] (3.75,1.5) to (4.5,1.5);
\draw [ultra thick,black,dotted,->-] (4.5,2.5) to (4.5,1.5);
\draw [ultra thick,black,dotted]  (4.5,1.5) to (5.5,1.5);
\draw [ultra thick,red,dashed] (5.5,1.5) to (6.0,0.75);
\draw [ultra thick,blue] (6.5,1.5) to (6.0,0.75);
\draw [ultra thick,green,snake it] (6.0,0.75) to (6.0,0);
\draw[thick,draw=black,fill=white]  (0,0) circle (3pt);
\draw[thick,draw=black,fill=black]  (0,0.75) circle (3pt);
\draw[thick,draw=black,fill=white]  (0.5,1.5) circle (3pt);
\draw[thick,draw=black,fill=white]  (1.5,1.5) circle (3pt);
\draw[thick,draw=black,fill=black]  (2.25,1.5) circle (3pt);
\draw[thick,draw=black,fill=white]  (3.0,1.5) circle (3pt);
\draw[thick,draw=black,fill=black]  (3.75,1.5) circle (3pt);
\draw[thick,draw=black,fill=white]  (4.5,1.5) circle (3pt);
\draw[thick,draw=black,fill=white]  (5.5,1.5) circle (3pt);
\draw[thick,draw=black,fill=black]  (6.0,0.75) circle (3pt);
\draw[thick,draw=black,fill=white]  (6.0,0) circle (3pt);
\node at (0,-0.3) {$1$};
\node at (6.0,-0.3) {$2$};
\node at (4.5,2.8) {$i$};
\node at (3.0,2.8) {$j$};
\node at (0.5,1.8) {$x$};
\node at (0.5,1.0) {$1_R$};
\node at (-0.5,1.0) {$1_L$};
\node at (5.5,1.0) {$2_L$};
\node at (6.5,1.0) {$2_R$};
\node at (4.8,2.1) {$i_X$};
\node at (3.3,2.1) {$j_Y$};
\node at (4.25,1.2) {$i_X$};
\node at (3.45,1.2) {$i_X$};
\node at (2.65,1.2) {$i_X, j_Y$};
\node at (2.15,1.9) {$i_X$};
\node at (2.15,1.9) {$i_X$};
\node at (1.85,1.2) {$j_Y$};
\end{tikzpicture}
\end{center}
\caption{The situation in \cref{lem:4k+2-path}.} \label{fig:4k+2-path}
\end{figure}

Suppose that a non-avocado edge on this path belongs to the colored path $i_X$ and the two following edges, $e_\ell$ to the left and $e_r$ to the right, belong to the colored path~$j_Y$.  Then, since colored paths share at most one (avocado) edge, we have $2 < i < j$. This situation is depicted in \cref{fig:4k+2-path}.

It follows that the colored path $i_X$ continues on edge $e_\ell$,  which is therefore colored avocado.  Furthermore, the left edge after $e_\ell$ belongs to path $j_Y$, whose color is therefore different from the color of $i_X$.
\end{proof}

\begin{lemma}\label{lem:4-cycles}
  A planar cubic bipartite simple graph has at least six $4$-cycles.
\end{lemma}
\begin{proof}
Let $G$ be a planar cubic bipartite simple graph.  Without loss of generality, we can assume that $G$ is connected.

By the handshaking lemma, a cubic graph has an even number of vertices, say~$2n$, and $3n$ edges.  For $k\geq 2$, let $f_{2k}$ be the number of faces bounded by $2k$ edges of $G$. By Euler's formula, the total number of faces of~$G$ equals $3n-2n+2 = n+2$.

Since $G$ is cubic, every vertex is contained in three faces.  Thus
\[6n  = \sum_{k\geq 2} 2k f_{2k} \geq 4 f_4 + 6(n+2 -f_4)\]
which implies that $f_4 \geq 6$.
\end{proof}

\begin{lemma}\label{lem:special_vertices}
Let $\W$ be a Kreweras web with at least one internal face.  Then there is an internal face of $\W$ which has at least three consecutive sides on its boundary with the outer face.
\end{lemma}
\begin{proof}
Let $G$ be the graph obtained from $\W$ by removing all vertices not contained in any internal face.  Let $v_1,\ldots,v_k$ be the list of vertices of degree $2$ in $G$.  The lemma's claim is equivalent to the assertion that there are two vertices among $v_1,\ldots,v_k$ which are adjacent, which we now show.

Let $G'$ be a copy of $G$, and let $v'_1, \ldots, v'_k$ be the vertices of $G'$ corresponding to $v_1,\ldots,v_k$.  We construct a planar bipartite cubic graph $H$ by adding edges $\{v_i, v'_i\}$ to $G\cup G'$ for $i\in\{1,\dots,k\}$, as depicted in \cref{fig:special_vertices}.

\begin{figure}
\begin{center}
\begin{tikzpicture}[scale=0.8]
  \node at (-1,1) {$G$};
  \node at (-1,-1) {$G'$};
  \node at (1,1.3) {$v_1$};
  \node at (3,1.3) {$v_2$};
  \node at (7,1.3) {$v_3$};
  \node at (8,1.3) {$v_4$};
  \draw[thick] (0,1)--(9,1);
  \draw[thick] (0,1) to[out=90,in=90] (5,1);
  \draw[thick] (6,1) to[out=90,in=90] (9,1);
  \draw[thick] (0,1)--(0.5,1.5);
  \draw[thick] (2,1)--(2,1.5);
  \draw[thick] (4,1)--(4,1.5);
  \draw[thick] (9,1)--(8.5,1.5);
  \draw[thick] (1,0.75)--(1,-0.75);
  \draw[thick] (3,0.75)--(3,-0.75);
  \draw[thick] (7,0.75)--(7,-0.75);
  \draw[thick] (8,0.75)--(8,-0.75);
  \draw[thick,draw=black,fill=white]  (0,1) circle (3pt);
  \draw[thick,draw=black,fill=black]  (1,1) circle (3pt);
  \draw[thick,draw=black,fill=white]  (2,1) circle (3pt);
  \draw[thick,draw=black,fill=black]  (3,1) circle (3pt);
  \draw[thick,draw=black,fill=white]  (4,1) circle (3pt);
  \draw[thick,draw=black,fill=black]  (5,1) circle (3pt);
  \draw[thick,draw=black,fill=white]  (6,1) circle (3pt);
  \draw[thick,draw=black,fill=black]  (7,1) circle (3pt);
  \draw[thick,draw=black,fill=white]  (8,1) circle (3pt);
  \draw[thick,draw=black,fill=black]  (9,1) circle (3pt);
  \node at (1,-1.5) {$v'_1$};
  \node at (3,-1.5) {$v'_2$};
  \node at (7,-1.5) {$v'_3$};
  \node at (8,-1.5) {$v'_4$};
  \draw[thick] (0,-1)--(9,-1);
  \draw[thick] (0,-1) to[out=270,in=270] (5,-1);
  \draw[thick] (6,-1) to[out=270,in=270] (9,-1);
  \draw[thick] (0,-1)--(0.5,-1.5);
  \draw[thick] (2,-1)--(2,-1.5);
  \draw[thick] (4,-1)--(4,-1.5);
  \draw[thick] (9,-1)--(8.5,-1.5);
  \draw[thick,draw=black,fill=black]  (0,-1) circle (3pt);
  \draw[thick,draw=black,fill=white]  (1,-1) circle (3pt);
  \draw[thick,draw=black,fill=black]  (2,-1) circle (3pt);
  \draw[thick,draw=black,fill=white]  (3,-1) circle (3pt);
  \draw[thick,draw=black,fill=black]  (4,-1) circle (3pt);
  \draw[thick,draw=black,fill=white]  (5,-1) circle (3pt);
  \draw[thick,draw=black,fill=black]  (6,-1) circle (3pt);
  \draw[thick,draw=black,fill=white]  (7,-1) circle (3pt);
  \draw[thick,draw=black,fill=black]  (8,-1) circle (3pt);
  \draw[thick,draw=black,fill=white]  (9,-1) circle (3pt);
\end{tikzpicture}
\end{center}
\caption{For the proof of \cref{lem:special_vertices}: how to obtain a bipartite cubic graph from a web.} \label{fig:special_vertices}
\end{figure}
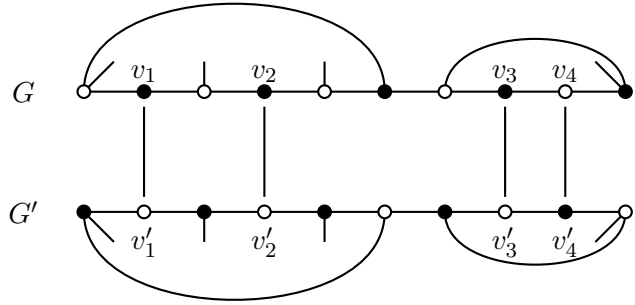

Suppose that $G$ has no pair of adjacent vertices of degree $2$. Then $H$ contains no $4$-cycles, which is impossible by \cref{lem:4-cycles}.
\end{proof}

\begin{lemma} \label{lem:coloring_success}
  \cref{construction:coloring} succeeds on any Kreweras web $\W$.
\end{lemma}
\begin{proof}

We use induction on the number of internal faces and the number of vertices of $\W$. We also freely relabel the boundary vertices via \cref{cor:coloring_rotation}.

If $\W$ has a white vertex which is not contained in an internal face, we replace this vertex by three independent boundary vertices and obtain three graphs $\W_1$, $\W_2$, and $\W_3$, ordered counterclockwise.  We label the boundary vertices of $\W_1$ and $\W_2$ such that the split vertex is the last boundary vertex and those of $\W_3$ such that the split vertex has label $1$.

By induction, \cref{construction:coloring} succeeds on all three graphs.  Moreover, we can color $\W_1$ and $\W_2$ such that the colors of the edges incident to their last boundary vertices are distinct.  Finally, we choose the coloring of $\W_3$ so that the edge to the left of the first boundary edge has the same color as the edge incident to the last boundary vertex of $\W_2$, and, accordingly, the edge to the right of the first boundary edge has the same color as the edge incident to the last boundary vertex of $\W_1$.  It is now clear that we obtain a coloring which coincides with the coloring produced by \cref{construction:coloring}.

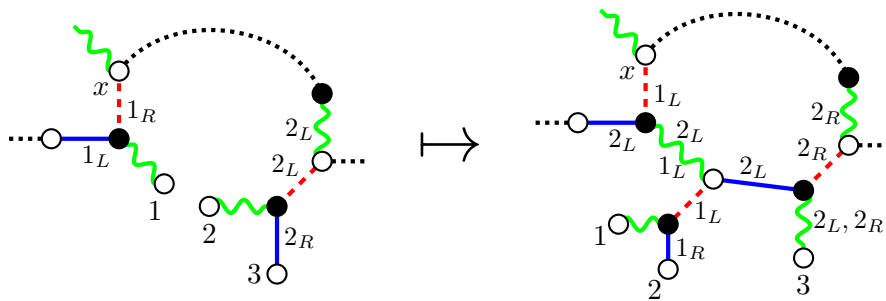
\begin{figure}
\begin{center}
  \begin{tikzpicture}
  \node at (0,0) {\begin{tikzpicture}[scale=1.2]
  \tikzset{snake it/.style={decorate, decoration=snake}};
  \node at (-0.1,-0.3) {$1$};
  \node at (0.5,-0.55) {$2$};
  \node at (1.0,-1.0) {$3$};
  \node at (-0.7,1.05) {$x$};
  \node at (-0.75,0.3) {\small $1_L$};
  \node at (-0.25,0.8) {\small $1_R$};
  \node at (1.5,-0.6) {\small $2_R$};
  \node at (1.35,0.2) {\small $2_L$};
  \node at (1.5,0.6) {\small $2_L$};
  \draw [ultra thick,green,snake it] (0,0) to (-0.5,0.5);
  \draw [ultra thick,red,dashed] (-0.5,0.5) to (-0.5,1.25);
  \draw [ultra thick,blue] (-0.5,0.5) to (-1.25,0.5);
  \draw [ultra thick,green,snake it] (0.5,-0.25) to (1.25,-0.25);
  \draw [ultra thick,blue] (1.25,-0.25) to (1.25,-1.0);
  \draw [ultra thick,red,dashed] (1.25,-0.25) to (1.75,0.25);
  \draw [ultra thick,green,snake it] (1.75,0.25) to (1.75,1.0);
  \draw [ultra thick,black,dotted] (-0.5,1.25) to [bend left =60] (1.75,1.0);
  \draw [ultra thick,black,dotted] (-1.25,0.5) to (-1.75,0.5);
  \draw [ultra thick,black,dotted] (1.75,0.25) to (2.25,0.25);
  \draw [ultra thick,green,snake it] (-0.5,1.25) to (-1.0,1.75);
  \draw[thick,draw=black,fill=white]  (0,0) circle (3pt);
  \draw[thick,draw=black,fill=black]  (-0.5,0.5) circle (3pt);
  \draw[thick,draw=black,fill=white]  (-1.25,0.5) circle (3pt);
  \draw[thick,draw=black,fill=white]  (-0.5,1.25) circle (3pt);
  \draw[thick,draw=black,fill=white]  (0.5,-0.25) circle (3pt);
  \draw[thick,draw=black,fill=black]  (1.25,-0.25) circle (3pt);
  \draw[thick,draw=black,fill=white]  (1.25,-1.0) circle (3pt);
  \draw[thick,draw=black,fill=white]  (1.75,0.25) circle (3pt);
  \draw[thick,draw=black,fill=black]  (1.75,1.0) circle (3pt);
  \end{tikzpicture}
  };
  \node at (3.5,0) {\Huge $\mapsto$};
  \node at (7,0) {\begin{tikzpicture}[scale=1.2]
  \tikzset{snake it/.style={decorate, decoration=snake}};
  \node at (-1.0,-0.75) {$1$};
  \node at (-0.4,-1.35) {$2$};
  \node at (1.25,-1.3) {$3$};
  \node at (-0.7,1.05) {$x$};
  \node at (-0.75,0.3) {\small $2_L$};
  \node at (-0.25,0.8) {\small $1_L$};
  \node at (1.75,-0.6) {\small $2_L,2_R$};
  \node at (1.35,0.2) {\small $2_R$};
  \node at (1.5,0.6) {\small $2_R$};
  \node at (0.0,-0.9) {\small $1_R$};
  \node at (0.2,-0.5) {\small $1_L$};
  \node at (-0.2,0.0) {\small $1_L$};
  \node at (0.0,0.4) {\small $2_L$};
  \node at (0.7,0.0) {\small $2_L$};
  \draw [ultra thick,green,snake it] (0.25,-0.125) to (-0.5,0.5);
  \draw [ultra thick,red,dashed] (-0.5,0.5) to (-0.5,1.25);
  \draw [ultra thick,blue] (-0.5,0.5) to (-1.25,0.5);
  \draw [ultra thick,blue] (0.25,-0.125) to (1.25,-0.25);
  \draw [ultra thick,green,snake it] (1.25,-0.25) to (1.25,-1.0);
  \draw [ultra thick,red,dashed] (1.25,-0.25) to (1.75,0.25);
  \draw [ultra thick,green,snake it] (1.75,0.25) to (1.75,1.0);
  \draw [ultra thick,green,snake it] (-0.8,-0.625) to (-0.25,-0.625);
  \draw [ultra thick,blue] (-0.25,-1.125) to (-0.25,-0.625);
  \draw [ultra thick,red,dashed] (0.25,-0.125) to (-0.25,-0.625);
  \draw [ultra thick,black,dotted] (-0.5,1.25) to [bend left =60] (1.75,1.0);
  \draw [ultra thick,black,dotted] (-1.25,0.5) to (-1.75,0.5);
  \draw [ultra thick,black,dotted] (1.75,0.25) to (2.25,0.25);
  \draw [ultra thick,green,snake it] (-0.5,1.25) to (-1.0,1.75);
  \draw[thick,draw=black,fill=white]  (0.25,-0.125) circle (3pt);
  \draw[thick,draw=black,fill=black]  (-0.25,-0.625) circle (3pt);
  \draw[thick,draw=black,fill=white]  (-0.8,-0.625) circle (3pt);
  \draw[thick,draw=black,fill=white]  (-0.25,-1.125) circle (3pt);
  \draw[thick,draw=black,fill=black]  (-0.5,0.5) circle (3pt);
  \draw[thick,draw=black,fill=white]  (-1.25,0.5) circle (3pt);
  \draw[thick,draw=black,fill=white]  (-0.5,1.25) circle (3pt);
  \draw[thick,draw=black,fill=black]  (1.25,-0.25) circle (3pt);
  \draw[thick,draw=black,fill=white]  (1.25,-1.0) circle (3pt);
  \draw[thick,draw=black,fill=white]  (1.75,0.25) circle (3pt);
  \draw[thick,draw=black,fill=black]  (1.75,1.0) circle (3pt);
  \end{tikzpicture}
  };
  \end{tikzpicture}
\end{center}
\caption{For the proof of \cref{lem:coloring_success}: how to break apart and reattach faces for the coloring in \cref{construction:coloring}.} \label{fig:coloring_reduction}
\end{figure}

Therefore, we can assume that all white vertices of $\W$ are contained in internal faces.  By \cref{lem:special_vertices}, there is an edge separating an internal face from the outer face, such that one of its vertices is black and adjacent to a boundary vertex, and the other vertex is white and adjacent to a black internal vertex which in turn is adjacent to two boundary vertices. Via \cref{cor:coloring_rotation}, we may assume that these latter two boundary vertices are labeled $1$ and $2$, and the former boundary vertex is labeled $3$.

We construct $\W'$ by removing the edge incident to the white internal vertex and the attached boundary vertices, and then splitting this white vertex into two independent boundary vertices.  By construction, the number of internal faces of $\W'$ is one less than the number of faces of $\W$.  We label its boundary vertices so that the two split vertices have labels $1$ and $2$.  By induction, \cref{construction:coloring} succeeds on~$\W'$.

Using the properties of this coloring guaranteed by \cref{lem:4k+2-path}, we can obtain a coloring of $\W$ which coincides with the coloring produced by \cref{construction:coloring}, as depicted in \cref{fig:coloring_reduction}.
\end{proof}

\Cref{lem:success_means_word,lem:coloring_success} implies \cref{thm:kreweras_webs}.

\section{Future directions} \label{sec:future}

In this section we discuss some potential connections and possible threads of future research.

\subsection{Relation of our work to previous work on webs and promotion} \label{subsec:webs}

Webs were introduced by Kuperberg~\cite{kuperberg1996spiders} to study invariant tensors of representations of simple Lie algebras (and their relatives like quantum groups). In~\cite{khovanov1999web}, Khovanov and Kuperberg described a bijection between the set of linear extensions of~$[3]\times [n]$ (i.e., \dfn{standard Young tableaux} (SYTs) of~$3 \times n$ rectangular shape) and irreducible $\mathfrak{sl}_3$-webs with~$3n$ white boundary vertices. In~\cite{petersen2009promotion} Petersen, Pylyavskyy, and Rhoades showed that, under the bijection of Khovanov--Kuperberg, promotion of linear extensions of~$[3]\times [n]$ corresponds to rotation of webs. This should be seen as directly analogous to the fact, mentioned in \cref{sec:intro}, that promotion of linear extensions of~$[2] \times [n]$ (i.e., promotion of Dyck words) corresponds to rotation of noncrossing matchings: indeed, noncrossing matchings can be seen as ``$\mathfrak{sl}_2$-webs.'' Later, Tymoczko~\cite{tymoczko2012simple} gave a different, simpler description of the Khovanov--Kuperberg bijection and used this description to reprove the results of Petersen--Pylyavskyy--Rhoades as well. Building on Tymoczko's work, Russell~\cite{russell2013explicit} (see also Patrias~\cite{patrias2019promotion}) related rotation of irreducible $\mathfrak{sl}_3$-webs with arbitrarily colored boundary vertices to promotion of $3$-rowed \emph{semi}standard (as opposed to standard) tableaux.

At its core, the proof of our main results boils down to showing that for a Kreweras word $w$, the web $\W_{\pro(w)}$ is the rotation of the web $\W_w$. Hence, our work would seem to be closely related to the aforementioned work relating webs and promotion. And indeed, our procedure of obtaining a web from a Kreweras bump diagram by breaking apart its crossings is very similar to Tymoczko's procedure of converting a so-called ``{\sf m}-diagram'' into a web. However, the exact relation between our work and prior work is not clear to us. Let us emphasize some points of contrast.

In the same way that linear extensions of $[2]\times [n]$ naturally correspond to Dyck words, linear extensions of $[3]\times [n]$ naturally correspond to words of length~$3n$ in the letters~$\A$, $\B$, and~$\C$, with equally many $\A$'s, $\B$'s, and $\C$'s, for which every prefix has at least as many~$\A$'s as~$\B$'s and at least as many~$\B$'s as~$\C$'s (this representation is usually called the \dfn{lattice word} or \dfn{Yamanouchi word} of the tableau). In this way, the linear extensions of $[3]\times [n]$ can be viewed as a subset of the Kreweras words of length $3n$. However, promotion of a linear extension of $[3]\times [n]$ is not the same as promotion of its corresponding Kreweras word. Furthermore, the web obtained from a linear extension of~$[3]\times [n]$ via the Khovanov--Kuperberg/Tymoczko bijection is not the same as the web~$\W_w$ for its corresponding Kreweras word~$w$. Indeed, as we have already seen with~\cref{thm:kreweras_webs}, only a subset of irreducible $\mathfrak{sl}_3$-webs with~$3n$ white boundary vertices arise as~$\W_w$ for some Kreweras word $w$ of length~$3n$. And, on the other hand, unlike the situation with SYTs, we also need extra decoration (the $3$-edge-coloring) to recover $w$ from~$\W_w$.

Another way that our work differs from the work mentioned above is
that for us, the trip permutation $\sigma_w$ associated to the web
$\W_w$ plays a central role, in contrast to previous work on webs and
promotion.  In fact, it seems that viewing a web as a plabic graph in
order to extract a trip permutation is a new idea, although
Lam~\cite{lam2015dimers} and Fraser, Lam and
Le~\cite{fraser2019dimers} discuss some relationships between webs
and plabic graphs.  We also note that one could adapt the argument in
\cref{lem:perm_from_growth_diagram} to show that for a linear
extension of $[3]\times [n]$, the trip permutation of its
Khovanov--Kuperberg/Tymoczko web can similarly be read off from its
growth diagram.

At any rate, it would certainly be interesting to understand more precisely the connection between our work and the previous work on webs and promotion, if there is some precise connection.

\subsection{More enumeration, and connection with planar maps} \label{subsec:maps}

As we just saw in \cref{subsec:webs}, the total number of irreducible $\mathfrak{sl}_3$-webs with~$3n$ white boundary vertices is the same as the number of standard Young tableaux of $3 \times n$ rectangular shape, for which there is a well-known product formula $2 \, \frac{(3n)!}{n!(n+1)!(n+2)!}$ (sometimes these numbers are called ``three-dimensional Catalan numbers''). \Cref{cor:web_count} gives a product formula for a weighted enumeration of Kreweras webs. Let us now explain how one can enumerate Kreweras webs, without this weighting. As we will see, certain famous sequences of numbers counting planar maps naturally arise.

We will employ a small amount of generatingfunctionology for this task. When dealing with combinatorial generating functions it is often useful to reduce to ``connected'' objects. We say a (non-empty) Kreweras word $w$ is \dfn{connected} if it contains no proper consecutive substring which is a Kreweras word. Evidently,
\[\textrm{$w$ is connected} \Leftrightarrow \textrm{$\D_w$ is connected} \Leftrightarrow \textrm{$\W_w$ is connected}.\]

Let us form the generating functions
\begin{align*}
K(x) &\coloneqq \sum_{n=0}^{\infty} K_n \, x^{3n} = 1+2x^3+16x^6+192x^9+2816x^{12}+46592x^{15}+\ldots; \\
K^c(x) &\coloneqq \sum_{n=1}^{\infty} K^c_n \, x^{3n} = 2x^3 + 4x^6 + 24x^9 + 208x^{12} +2176x^{15}+ \ldots,
\end{align*}
where
\begin{align*}
K_n &\coloneqq \# \textrm{ Kreweras words of length $3n$};\\
K^c_n &\coloneqq \# \textrm{ connected Kreweras words of length $3n$}.
\end{align*}
As we have seen in \cref{sec:intro},  $K_n =\frac{4^n}{(n+1)(2n+1)}\binom{3n}{n}$ (\url{http://oeis.org/A006335}). Note that every Kreweras word is obtained, in a unique way, from a connected Kreweras word $w$ by inserting an arbitrary (possibly empty) Kreweras word after each letter of $w$. This yields the generating function equation
\[ K(x) = 1+K^c(xK(x)).\]
From the above equation, and the formula for $K_n$, it is possible to use Lagrange inversion to deduce that $K^c_n = 2^{n+1} \, \frac{(4n-3)!}{(3n-1)!n!}$.

Next, we do the same but with Kreweras webs instead of Kreweras words. That is, we form the generating functions
\begin{align*}
W(x) &\coloneqq \sum_{n=0}^{\infty} W_n \, x^{3n} = 1 + x^3 + 5x^6 + 42x^9 +459x^{12}+ 5871x^{15}+ \ldots; \\
W^c(x) &\coloneqq \sum_{n=1}^{\infty} W^c_n \, x^{3n} = x^3 + 2x^6 + 12x^9 + 104x^{12} + 1088x^{15}+\ldots,
\end{align*}
where
\begin{align*}
W_n &\coloneqq \# \textrm{ Kreweras webs with $3n$ boundary vertices};\\
W^c_n &\coloneqq \# \textrm{ connected Kreweras webs with $3n$ boundary vertices}.
\end{align*}
It follows from \cref{thm:kreweras_webs} that $W^c_n = \frac{1}{2}K^c_n$. Moreover, the same reasoning as in the case of Kreweras words implies the generating function equation
\[ W(x) = 1+W^c(xW(x)).\]
From the above equation, and the formula for $W^c_n$, it is possible to use Lagrange inversion to obtain the coefficients $W_n$, although the answer one obtains is not as nice as for $K_n$.

Finally, let us explain the connection with planar maps. Recall that a \dfn{planar map} is a topological equivalence class of embeddings of a connected planar graph in the sphere. The number of rooted, bridgeless, cubic planar maps with $2n$ vertices is
\[ \frac{2^n}{(n+1)(2n+1)}\binom{3n}{n} = 2^{-n} \cdot K_n \qquad \textrm{(\url{http://oeis.org/A000309})}.\]
Bernardi~\cite{bernardi2007bijective} defined a bijection from Kreweras words of length $3n$ to rooted, bridgeless, cubic planar maps with $2n$ vertices decorated with a \dfn{depth tree}, which is a certain kind of spanning tree. He also explained why every such map has exactly $2^{n}$ depth trees, and thus combinatorially explained the above equality. 

Meanwhile, the number of rooted, $3$-connected, cubic planar maps with $2n$ vertices is
\[ 2 \, \frac{(4n-3)!}{(3n-1)!n!} = 2^{-n} \cdot K^c_n = 2^{-(n-1)} \cdot W^c_n \qquad \textrm{(\url{http://oeis.org/A000260})}.\]
We believe that under Bernardi's bijection, a Kreweras word is connected if and only if its corresponding cubic planar map is $3$-connected. Thus Bernardi's bijection also explains, combinatorially, the above equality. However, it would be desirable to \emph{directly} explain why the enumeration of connected Kreweras webs is related to the enumeration of rooted, $3$-connected, cubic planar maps, without going through Kreweras words. There is some reason to hope this is possible because webs seem at least superficially similar to cubic planar maps.

\subsection{An algebraic model} \label{subsec:algebra}

Is there an algebraic model which explains the behavior of promotion on Kreweras words?

The Henriques-Kamnitzer \dfn{cactus group action}~\cite{henriques2006crystals} on the tensor product of crystals of representations of a (simple, finite-dimensional) Lie algebra gives rise to a notion of promotion acting on the highest weight words of weight zero for such a tensor product: see~\cite{fontaine2014cyclic, westbury2016invariant, pfannerer2020promotion}. For example, letting~$V$ be the vector representation of $\mathfrak{sl}_k$, this cactus group promotion action on the weight zero highest weight words of~$\otimes^{kn} V$ corresponds to promotion of linear extensions of $[k] \times [n]$ (i.e., promotion of SYTs of $k \times n$ rectangular shape). Moreover, there are general results (see the references above) which imply that cactus group promotion of weight zero highest weight words always has good behavior.

As mentioned in the previous subsection, Kuperberg first defined webs in order to study invariant tensors in tensor products of representations of Lie algebras. So it is not so unreasonable to think that promotion of Kreweras words could be connected to invariant tensors and the Henriques-Kamnitzer cactus group action in some way. Perhaps the proper algebraic model for Kreweras words will come from representations of some variant of a simple, finite-dimensional Lie algebra, like a Lie superalgebra or a Kac--Moody algebra. A better understanding of the algebraic significance of the ``no $4k$-sided internal faces'' condition for the Kreweras webs could be the key to uncovering the proper algebraic model for Kreweras words.

We note that the cactus group promotion of highest weight words can be described via \dfn{local rules}: see~\cite[\S 4.2]{pfannerer2020promotion}. As we saw in \cref{sec:evac}, promotion of Kreweras words can also be described via local rules. However, local rules alone are not enough to guarantee good behavior of promotion: again, as we saw in \cref{sec:evac}, promotion of the linear extensions of any poset can be described by local rules, but most posets have bad behavior of promotion.

\subsection{Cyclic sieving}

Diagrammatic and/or algebraic models are often useful for establishing \dfn{cyclic sieving} results. Let us recall this notion from Reiner--Stanton--White~\cite{reiner2004cyclic}:

\begin{definition}
Let $X$ be a finite set. Let $C=\langle c \rangle$ be a cyclic group of order $\ell$ acting on $X$, generated by element $c\in C$. Let~$f(q) \in \mathbb{N}[q]$ be a polynomial in $q$ with nonnegative integer coefficients. Then we say the triple $(X,\Phi,f(q))$ exhibits \dfn{cyclic sieving} if for all $k$ we have
\[ \#\{x\in X\colon c^k(x)=x\} = f(\omega^{k}), \]
where $\omega  \coloneqq  e^{2\pi i/\ell}$ is a primitive $\ell$th root of unity.
\end{definition}

Cyclic sieving phenomena (CSPs) involving polynomials which have an expression as a ratio of products of $q$-numbers are especially valued, because they imply that every symmetry class has a product formula.

In the case of promotion of SYTs of  $k\times n$ rectangular shape, Rhoades~\cite{rhoades2010cyclic} obtained such a CSP. He showed $(\{\textrm{SYTs of shape $k\times n$}\},\langle\pro\rangle,f(q))$ exhibits cyclic sieving, where $f(q)$ is the \dfn{major index generating function} for the SYTs of shape $k \times n$, which has the product formula
\[ f(q) = \frac{q^{n\binom{k}{2}}\prod_{j=1}^{kn}(1-q^{j})}{\prod_{j=1}^{k} (1-q^j)^j \prod_{j=k+1}^{n} (1-q^j)^k \prod_{j=n+1}^{n+k}(1-q^j)^{n+k-j}}, \]
assuming by symmetry that $k\leq n$. Note that this product formula is the well-known \dfn{$q$-hook length formula}~\cite[Cor. 7.21.5]{stanley1999ec2}.

We conjecture the following CSP for promotion of Kreweras words:

\begin{conj} \label{conj:cyc_siev}
For all $n\geq 1$, the rational expression
\[ f(q)  \coloneqq  \frac{\prod_{j=1}^{3n} (1-q^{2j})}{\prod_{j=2}^{2n+1}(1-q^j)\prod_{j=2}^{n+1} (1-q^{2j}) }\]
is a polynomial in $q$ with nonnegative integer coefficients, and the triple
\[(\{\textrm{Kreweras words of length $3n$}\}, \langle\pro\rangle,f(q))\]
exhibits cyclic sieving.
\end{conj}

\Cref{conj:cyc_siev} strongly suggests that some good algebraic model for Kreweras word promotion should exist, although we do not know of the precise algebraic or combinatorial significance of the polynomial $f(q)$ appearing in the conjecture.

We also conjecture similarly that there is a product formula for the number of Kreweras words fixed by $\evac$ and $\evac^{*}$:

\begin{conj} \label{conj:evac}
For all $n\geq 1$, the number of Kreweras words of length $3n$ with $\evac^{*}(w)=w$ is
\[ \frac{3^{\lfloor n/2 \rfloor} 4^{\lceil n/2 \rceil} \prod_{j=1}^{\lfloor n/2 \rfloor}(3j-1) \prod_{j=1}^{\lceil n/2 \rceil}(3j-2) }{(n+1)!}.\]
The number of Kreweras words of length $3n$ with $\evac(w)=w$ is this same number if~$n$ is even, and is $0$ if $n$ is odd.
\end{conj}

It is possible that \cref{conj:evac} could be phrased as a ``$q=-1$'' result for a polynomial which has a product formula as a rational expression, although we do not have a candidate for such a polynomial. Note that every poset has a ``$q=-1$'' result for counting self-evacuating linear extensions, where the polynomial is essentially the major index generating function: see~\cite[\S 3]{stanley2009promotion}.

\subsection{Order polynomial product formulas}

It is reasonable to ask ``where is this Kreweras word promotion really coming from?,'' or in other words, ``what is it about the poset $V(n)$ that would lead one to suspect that it has good promotion behavior?'' Let us attempt to answer this question.

Let $P$ be a poset. A \dfn{$P$-partition of height $m$} is a weakly order-preserving map $\pi\colon P \to \{0,1,\ldots,m\}$. It is well-known that the number of $P$-partitions of height~$m$ is given by a polynomial $\Omega_P(m)$ in $m$, called the \dfn{order polynomial} of $P$, whose degree is~$\#P$ and whose leading coefficient is $1/\#P!$ times the number of linear extensions of~$P$ (see, e.g.,~\cite[\S 3.12]{stanley2012ec1}).

In~\cite{kreweras1981solution}, Kreweras--Niederhausen obtained the following product formula for the entire order polynomial of the poset $V(n)$:
\[\Omega_{V(n)}(m) = \frac{\prod_{i=1}^{n}(m+1+i) \prod_{i=1}^{2n}(2m+i+1)}{(n+1)!(2n+1)!}.\]
They deduced the product formula counting Kreweras words (i.e., linear extensions of~$V(n)$) as a corollary.

In~\cite{hopkins2020order}, the first author (S.H.) presented the following heuristic: ``posets with good dynamical behavior = posets with order polynomial product formulas.'' Here ``good dynamical behavior'' includes good behavior of promotion of linear extensions. It is via this heuristic that promotion for Kreweras words was discovered:~S.H. asked a question on MathOverflow about posets with order polynomial product formulas, and was lead to the paper~\cite{kreweras1981solution} and the $V(n)$ poset by an answer of Ira Gessel~\cite{gessel2019MO}.

\subsection{Rowmotion}

\dfn{Rowmotion} is a certain invertible action on the order ideals of any poset $P$ which has been studied by many authors over a number of years, with a renewed interest especially in the last 10 or so years. Rowmotion and promotion are ``similar'' in many respects. For an overview and history of rowmotion see, e.g.,~\cite{striker2012promotion} or~\cite{thomas2019rowmotion}. Einstein and Propp~\cite{einstein2013combinatorial} introduced a piecewise-linear extension of rowmotion, which in particular yields a (piecewise-linear) action of rowmotion on the set of $P$-partitions of height~$m$.

In the ``posets with good dynamical behavior = posets with order polynomial product formulas'' heuristic just mentioned, ``good dynamical behavior'' also includes good behavior of rowmotion of order ideals, and more generally good behavior of (piecewise-linear) rowmotion of $P$-partitions.

In agreement with this heuristic, it (experimentally) appears that the poset $V(n)$ has good behavior of rowmotion of order ideals and $P$-partitions. Conjectures concerning rowmotion for $V(n)$ appeared in the aforementioned paper~\cite{hopkins2020order}.

\bibliography{kreweras}{}
\bibliographystyle{plain}

\end{document}